\documentclass[10pt]{article}

\usepackage[a4paper, top=2cm, bottom=3cm, left=2.5cm, right=2.5cm]{geometry} \usepackage{mathtools}
\usepackage{amssymb}
\usepackage{mathrsfs}
\usepackage{amsthm}
\usepackage{dsfont}

\usepackage[shortlabels]{enumitem}
\usepackage{xcolor}

\usepackage{cite}
\usepackage{hyperref}

\newtheorem{theorem}{Theorem}[section]

\newtheorem{proposition}[theorem]{Proposition}
\newtheorem{lemma}[theorem]{Lemma}
\theoremstyle{definition}
\newtheorem{definition}[theorem]{Definition}
\theoremstyle{remark}
\newtheorem{remark}[theorem]{Remark}

\newcommand{\abs}[1]{\left\lvert#1\right\rvert}
\newcommand{\norm}[1]{\left\lVert#1\right\rVert}
\newcommand{\scp}[2]{\left\langle#1,#2\right\rangle}

\newcommand{\dalt}{\,\mathrm{d}}
\newcommand{\dd}{\;\mathrm{d}}
\newcommand{\N}{\mathbb{N}}

\newcommand{\R}{\mathbb{R}}
\newcommand{\C}{\mathbb{C}}

\DeclareMathOperator{\divergence}{div}

\DeclareMathOperator{\identity}{Id}

\DeclareMathOperator{\distance}{dist}

\DeclareMathOperator{\support}{supp}

\DeclareMathOperator{\linspan}{span}

\newcounter{saveenumi}

\begin{document}
\reversemarginpar

\title{\Large\bfseries
The Stokes Operator with Power Weights Outside the Muckenhoupt Class
}

\author{
Erik Heidrich\thanks{Department of Mathematics, RPTU Kaiserslautern-Landau.
Email address: \texttt{heidr@rptu.de}}
}

\date{}

\maketitle

\begin{abstract}
We investigate the Stokes resolvent problem with no-slip boundary conditions on the half space in weighted $L^p$-spaces, $1<p<\infty,$ with power weights $w_\gamma(x,y)=y^\gamma$. We distinguish two cases: For $-p-1<\gamma<p-1$, we prove resolvent estimates for the velocity on $L^p_{\sigma,w_\gamma}(H)$. We moreover construct a Stokes realization on $L^p_{\sigma,w_\gamma}(H)$ and show that it is the generator of an analytic $C_0$-semigroup and admits a bounded $H^\infty$-calculus.
For $p-1<\gamma<2p-1,$ we show that classical weighted velocity resolvent estimates fail and that there is no Helmholtz decomposition of the ambient weighted space into its solenoidal and gradient subspaces. We instead prove resolvent estimates modulo gradients, and construct a Stokes realization on the corresponding quotient space $L^p_{w_\gamma}(H)^{n+1}/G^p_{w_\gamma}(H).$ This realization is the generator of an analytic $C_0$-semigroup and admits a bounded $H^\infty$-calculus.
\end{abstract}

\medskip


\smallskip


\section{Introduction}\label{sec:Introduction}
On the half space $H \coloneqq \R^n\times (0,\infty),\,n\geq 1,$ we consider the Stokes resolvent problem

\begin{equation}\label{eq:ResolventProblem}
\begin{aligned}
    \lambda u - \Delta u + \nabla \pi &= f \quad \text{in } H,\\
    \divergence u &= 0 \quad \text{in } H,\\
    u|_{\partial H} &= 0,
\end{aligned}
\end{equation}
in weighted $L^p$-spaces $L^p_w(H)^{n+1} \coloneqq L^p(H; w\,\mathrm{d}x)^{n+1},$ where $1 < p < \infty,$ the resolvent parameter $\lambda\in\mathbb{C}$ is from a sector
$$\Sigma_\theta\coloneqq\{z=re^{i\eta}\in\mathbb{C}\,\colon\,r>0,\,\abs{\eta}<\theta\},\;\theta \in (0,\pi),$$
and the weight $w = w_\gamma$ is of the form
\begin{equation}\label{eq:powerWeight}
w_\gamma(x,y) \coloneqq \distance\bigl((x,y),\partial H\bigr)^\gamma =  y^\gamma,\quad (x,y)\in \R^n\times(0,\infty),
\end{equation}
for powers $\gamma \in (-p-1,2p-1).$ In this case we also write $L^p_\gamma(H)^{n+1} \coloneqq L^p_{w_\gamma}(H)^{n+1}.$

For $\lambda\in\mathbb{C}\setminus(-\infty,0]$ and $f\in C^\infty_{0,\sigma}(H)$, we call a pair $(u_\lambda,\pi_\lambda)$ a \emph{strong $L^2$-solution} of \eqref{eq:ResolventProblem} if
\begin{equation*}
\begin{aligned}
\lambda &u_\lambda - \Delta u_\lambda + \nabla \pi_\lambda = f\qquad\text{in } H,\\
u_\lambda &\in L^2_\sigma(H)\cap W^{2,2}(H)^{n+1}\cap W^{1,2}_0(H)^{n+1},\\
\pi_\lambda &\in L^1_{\mathrm{loc}}(\overline{H}),\quad \nabla \pi_\lambda\in L^2(H)^{n+1}.
\end{aligned}
\end{equation*}

For every such $\lambda$ and $f$, there is a unique velocity $u_\lambda$ and a pressure $\pi_\lambda,$ unique up to constants, satisfying these conditions \cite[Theorem 1.3]{FarwigSohr1994}; see also \cite[Theorem 8.1]{McCracken1981}.
The central aim of this article is to investigate whether the corresponding solution operator
$$\mathcal{U}_\lambda\colon C^\infty_{0,\sigma}(H)\longrightarrow L^2_\sigma(H)\cap W^{2,2}(H)^{n+1}\cap W^{1,2}_0(H)^{n+1},\qquad\mathcal{U}_\lambda f\coloneqq u_\lambda,$$
extends to a bounded operator
$$\mathcal{U}_\lambda\,\colon\, L^p_{\sigma,\gamma}(H)\to L^p_{\sigma,\gamma}(H).$$
The weighted estimates studied below are thus properties of this fixed unweighted solution.

Considering partial differential equations in spaces with weights has found many applications, some of which have been known for a long time; see e.g. \cite{LionsMagenes1968Vol1} or the introduction in \cite{Kufner1985}. Somewhat more recently, weighted spaces have proven to be a key tool in the theory of boundary value problems for stochastic partial differential equations; see \cite{Krylov1994, Krylov1999HalfSpace, Kim2014} just to mention a few examples; for an overview see also \cite{Krylov2008Lp, Krylov2008Sobolev} and the references therein. This is due to the fact that the roughness of the noise in stochastic PDEs forces additional compatibility conditions on the boundary terms for a solution to exist in unweighted $L^p$-spaces. This was observed by Flandoli in \cite[Proposition 3.3]{Flandoli1990}; see also \cite[Examples 1.1 and 1.2]{Krylov1994} for two simple examples of this circumstance. To overcome this, weights of the form \eqref{eq:powerWeight} can allow for (and give a precise control over) blow-up of the solution and its derivatives at the boundary.

In respective applications, the range $p - 1 < \gamma < 2p - 1$ has turned out to be of special interest, see e.g. \cite[Theorem 0.1]{Krylov2001}. This poses a challenge to the analysis of \eqref{eq:ResolventProblem}. It is well-known that the class of so-called \emph{Muckenhoupt weights}, or $A_p$-weights for short, is precisely the class of weights in which one has a well-formed harmonic analysis for $L^p(w\, \mathrm{d}x)$ at hand. By this we mean that $w \in A_p$ is \emph{equivalent} to many desirable properties of the space $L^p(w\, \mathrm{d}x),$ such as boundedness on $L^p(w\, \mathrm{d}x)$ of the Hardy-Littlewood maximal operator, boundedness of the Neumann Riesz transforms, validity of Mikhlin-type multiplier theorems, and more, see \cite{GCRdF1985, Grafakos2009}. The weight $w_\gamma$ as defined in \eqref{eq:powerWeight} belongs to $A_p$ if and only if 
$$-1 < \gamma < p - 1,$$
see Section \ref{sec:WeightedSpaces}. This poses a substantial restriction on any theory for \eqref{eq:ResolventProblem} in $L^p(w_\gamma\, \mathrm{d}x)$.

To illustrate this, in the classical theory of the Stokes equation, a central role is played by the \emph{Helmholtz decomposition}
\begin{equation}\label{eq:HelmholtzDecomposition}
L^p(\Omega)^{n+1} \;=\; L^p_\sigma(\Omega)  \oplus G^p(\Omega)\;\simeq\;\bigl(L^p(\Omega)^{n+1}/G^p(\Omega)\bigr) \oplus G^p(\Omega)
\end{equation}
and the associated continuous projection 
$$\mathbb{P}_p \, \colon \, L^p(\Omega)^{n+1} \to L^p_\sigma(\Omega),$$
which are valid under a wide range of combinations of assumptions on $p$ and $\Omega$; see \cite{Galdi1994I} for a rather self-contained exposition or the references in \cite{FarwigKozonoSohr2007Helmholtz}. Here the spaces $L^p_{\sigma,\gamma}(\Omega)$ and $G^p_\gamma(\Omega)$ are defined as
$$C^\infty_{0,\sigma}(\Omega)\coloneqq \{\varphi \in C^\infty_0(\Omega)^{n+1} \,\colon\, \divergence \varphi = 0\},\quad L^p_{\sigma,\gamma}(\Omega) \coloneqq \overline{C^\infty_{0,\sigma}(\Omega)}^{\norm{\cdot}_{L^p_\gamma(\Omega)^{n+1}}}\subseteq L^p_\gamma(\Omega)^{n+1}$$
and
$$G^p_\gamma(\Omega) \coloneqq\bigl\{g\in L^p_\gamma(\Omega)^{n+1} \, \colon \, g=\nabla \varphi\text{ in }\mathcal{D}^\prime(\Omega)\text{ for some }\varphi\in L^1_{\mathrm{loc}}(\overline{\Omega})\bigr\};$$
the respective unweighted spaces correspond to $\gamma=0.$ In \cite{Froehlich2000} it was shown that this decomposition is also valid for $L^p_w(\Omega)^{n+1}$ when
\begin{itemize}
\item $p\in(1,\infty),$
\item $\Omega$ is $\R^{n+1},\, H$ or a bounded or exterior $C^1$-domain,
\item and, crucially, $w \in A_p.$
\end{itemize}
Given the fact that, when interpreted in a suitable way (see e.g. \cite{GigaGu2020}),
$\mathbb{P}_p$ admits the representation
$$\mathbb{P}_p = \identity - \nabla (\Delta^N)^{-1}\divergence,$$
where $\Delta^N$ denotes the Neumann Laplacian, the question of the existence of such a decomposition is connected to weighted estimates for the associated Riesz transforms $R_{N,i}=\partial_i (-\Delta^N)^{-1/2},\, i=1,\dots,{n+1}$. These transforms are unbounded on $L^p_\gamma(H)$ when $ p-1<\gamma<2p-1$; see \cite[Theorem 3.5]{DHLWY2017}. Thus, the standard approach through weighted Riesz-transform estimates is unavailable.

In Proposition \ref{prop:CounterexampleHelmholtzDecomposition} we will show that when $p-1<\gamma<2p-1,$ there exists an element
\begin{equation*}
0\neq h\in L^p_{\sigma,\gamma}(H)\cap G^p_\gamma(H)
\end{equation*}
and, consequently, a Helmholtz decomposition in the form \eqref{eq:HelmholtzDecomposition} cannot hold for the spaces $L^p_\gamma(H)^{n+1},\,p-1<\gamma<2p-1.$ 

Subsequently, we derive two negative results on the bounded extension of the solution operator $\mathcal{U}_\lambda$ to $L^p_\gamma(H)^{n+1}$ in case $\gamma>p-1:$
\begin{proposition}[Unboundedness in the upper range]\label{prop:NegativeResults}
Let $1<p<\infty$ and let $\gamma>p-1.$ 
\begin{enumerate}[(i)]
    \item If $\gamma\geq(n+1)(p-1),$ there exists a function $f\in C_{0,\sigma}^\infty(H)$ such that $$\mathcal{U}_1 f\notin L^p_\gamma(H)^{n+1}.$$\label{eq:MappingFailure}
    \item For every $p-1<\gamma<(n+1)(p-1)$ and $\lambda\in\mathbb{C}\setminus(-\infty,0]$ it holds
    $$\sup_{0\neq f\in C^\infty_{0,\sigma}(H)}\frac{\abs{\lambda}\norm{\mathcal{U}_\lambda f}_{L^p_\gamma(H)^{n+1}}}{\norm{f}_{L^p_\gamma(H)^{n+1}}}=\infty.$$\label{eq:ResolventEstimateFailure}
\end{enumerate}
\end{proposition}
To recover resolvent estimates in the upper range, we consider velocity classes modulo gradients. This choice is motivated by the nontrivial solenoidal gradients constructed in Proposition \ref{prop:CounterexampleHelmholtzDecomposition}. We therefore introduce the state spaces
\begin{equation}\label{eq:ModifiedScale}
X^p_{\sigma,\gamma}(H)\coloneqq
\begin{cases}
L^p_{\sigma,\gamma}(H),\quad&-p-1<\gamma<p-1,\\
L^p_\gamma(H)^{n+1}/G^p_\gamma(H),\quad&\hphantom{-}p-1\leq\gamma<2p-1.
\end{cases}
\end{equation}
In Proposition \ref{prop:AnnihilatorCharacterization} it is shown that $G^p_\gamma(H)$ is a closed subspace of $L^p_\gamma(H)^{n+1}$, hence the above quotient space is a Banach space.
Write
$$q_{p,\gamma}\,\colon\, L^p_\gamma(H)^{n+1}\longrightarrow X^p_{\sigma,\gamma}(H),\qquad q_{p,\gamma}f=[f],\qquad p-1\leq\gamma<2p-1,$$
for the quotient map, with the usual quotient norm.

A key advantage of the spaces $X_{\sigma,\gamma}^p(H)$ is their compatibility with duality: In Section \ref{sec:Sectoriality} we establish an isometric isomorphism
$$\Phi_{p,\gamma}\,\colon\, L^p_\gamma(H)^{n+1}/G^p_\gamma(H)\to \bigl(L^{p^\prime}_{\sigma,\gamma^\prime}(H)\bigr)^\prime,$$
which identifies the upper-range quotient spaces with the duals of the corresponding lower-range solenoidal spaces. The proof of Theorem \ref{th:ResolventEstimates} combines this duality with precise estimates derived from the solution representation of \cite{DHP2001}. We first establish the resolvent estimates in the lower range and then transfer them to the upper range through $\Phi_{p,\gamma}$.

We emphasize that all results concern homogeneous Dirichlet (no-slip) boundary conditions. Our weighted estimates exploit boundary cancellation both in the Dirichlet Laplacian estimates and in the correction kernels arising from the solution representation of \cite{DHP2001}; see Proposition \ref{prop:pointwiseEstimate0} and Section \ref{sec:Sectoriality}. This cancellation is essential to our treatment of weights outside the Muckenhoupt range, and the argument does not extend directly to other boundary conditions.

Our first main result reads as follows.
\begin{theorem}[Resolvent estimates]\label{th:ResolventEstimates}
Let $1<p<\infty,$ let $\theta\in(0,\pi),$ and $-p-1<\gamma<2p-1$. 
Then there exists a constant $C_{n,p,\gamma,\theta}>0$ such that for $f\in C_{0,\sigma}^\infty(H),$
\begin{enumerate}[(i)]
\item if $-p-1<\gamma<p-1$, then
\begin{equation*}
\norm{\mathcal{U}_\lambda f}_{X_{\sigma,\gamma}^p(H)}\leq \frac{C_{n,p,\gamma,\theta}}{\abs{\lambda}} \norm{f}_{X_{\sigma,\gamma}^p(H)},\qquad\lambda\in\Sigma_\theta,
\end{equation*}\label{eq:LowerCoreResolventEstimate}
\item if $p-1\leq\gamma<2p-1$ and $\mathcal{U}_\lambda f\in L^p_\gamma(H)^{n+1}$, then
$$\norm{q_{p,\gamma}\mathcal{U}_\lambda f}_{X_{\sigma,\gamma}^p(H)}\leq \frac{C_{n,p,\gamma,\theta}}{\abs{\lambda}}\norm{q_{p,\gamma}f}_{X^p_{\sigma,\gamma}(H)},\qquad\lambda\in\Sigma_\theta.$$\label{eq:UpperCoreResolventEstimate}
\end{enumerate}
\end{theorem}

In the lower range, Theorem \ref{th:ResolventEstimates}\ref{eq:LowerCoreResolventEstimate} gives a unique bounded extension of $\mathcal U_\lambda$ to $L^p_{\sigma,\gamma}(H)$. In the upper range, fix $p-1\leq\gamma<2p-1$ and define
$$D_\lambda=\{f\in C_{0,\sigma}^\infty(H):\mathcal U_\lambda f\in L^p_\gamma(H)^{n+1}\}.$$
Define the operator $T_\lambda$ by
\begin{equation*}
\begin{cases}
D(T_\lambda)=C_{0,\sigma}^\infty(H),\quad T_\lambda f=\mathcal U_\lambda f,\qquad &-p-1<\gamma<p-1,\\
D(T_\lambda)=q_{p,\gamma} D_\lambda,\quad T_\lambda q_{p,\gamma}f=q_{p,\gamma}\mathcal U_\lambda f,\qquad &\hphantom{-}p-1\leq\gamma<2p-1.
\end{cases}
\end{equation*}
By Theorem \ref{th:ResolventEstimates}, the operator is well-defined and bounded. In the next theorem, we establish that its domain is dense and that its closure is an operator resolvent. That operator can thus be interpreted as a Stokes realization on $X^p_{\sigma,\gamma}(H)$.

\begin{theorem}[Stokes realizations]\label{th:StokesOperator}
Let $1<p<\infty$ and let $-p-1<\gamma<2p-1.$ There exists a unique closed, densely defined operator $\mathcal{A}_{p,\gamma}$ on $X^p_{\sigma,\gamma}(H)$ with $\C\setminus(-\infty,0]\subseteq\rho(\mathcal{A}_{p,\gamma})$ such that, for every $\lambda\in\C\setminus(-\infty,0],$
$$(\lambda-\mathcal{A}_{p,\gamma})^{-1}=\overline{T}_\lambda.$$
In particular, $D(T_\lambda)$ is dense in $X^p_{\sigma,\gamma}(H)$ for every $\lambda\in\C\setminus(-\infty,0].$
Moreover, for every $\theta\in(0,\pi)$ there exists a constant $C_{n,p,\gamma,\theta}>0$ such that
\begin{equation*}
\norm{(\lambda-\mathcal{A}_{p,\gamma})^{-1}}_{\mathcal{L}(X^p_{\sigma,\gamma}(H))}
\leq \frac{C_{n,p,\gamma,\theta}}{\abs{\lambda}}
\qquad\lambda\in\Sigma_\theta.
\end{equation*}
\end{theorem}
In particular, $-\mathcal{A}_{p,\gamma}$ is sectorial of angle zero on $X^p_{\sigma,\gamma}(H)$ and $\mathcal{A}_{p,\gamma}$ is the generator of a bounded holomorphic $C_0$-semigroup on $X^p_{\sigma,\gamma}(H)$ of angle $\pi/2$, cf. e.g. \cite[Theorem II.4.6 and Proposition I.5.3]{EngelNagel2000}.

Furthermore, we prove the stronger property of a bounded $H^\infty$-calculus on the same state spaces.
\begin{theorem}[Bounded $H^\infty$-calculus]\label{th:HinftyCalculus}
Let $1<p<\infty$ and $-p-1<\gamma<2p-1$. The operator $-\mathcal{A}_{p,\gamma}$ admits a bounded $H^\infty$-calculus of angle zero on $X^p_{\sigma,\gamma}(H)$.
\end{theorem}

The $H^\infty$-calculus is an extremely useful tool in the study of (S)PDEs for a multitude of reasons. This has become clear since the extension of the Dore-Venni theorem on the sums of closed operators to the setting where one operator admits a bounded $H^\infty$-calculus in \cite{KaltonWeis2001}. In particular, on a UMD space a bounded $H^\infty$-calculus of angle less than $\pi/2$ implies that the operator has maximal $L^r$-regularity for every $1<r<\infty.$ Moreover, the functional calculus behaves well under perturbations, see e.g. \cite{KKW2006}, and can be used to characterize domains of fractional powers via complex interpolation \cite{Haase2006}. Another important motivation is that for $2\leq p<\infty$ and $2<r<\infty,$ a bounded $H^\infty$-calculus of angle less than $\pi/2$ gives \emph{stochastic} maximal $L^r$-regularity on spaces isomorphic to closed subspaces of an $L^p$-space; $r=2$ is also allowed when $p=2$ \cite[Theorem 1.1 and the paragraph following Theorem 1.2]{NVW2012SMR}. See \cite{NVW2012,AgrestiVeraar2025} for applications to SPDEs.

In the unweighted case, boundedness of the $H^\infty$-calculus of the Stokes operator on the half space is shown in \cite{GigaSohr1991, DHP2001}. In the weighted case this result seems to have not appeared in the literature even within the $A_p$-range. It was proved by Fröhlich in \cite{Froehlich2003,Froehlich2007} that the Stokes operator on $L^p_\sigma(H; w\,\mathrm{d}x)$, where $w$ is an arbitrary weight of class $A_p$, generates a bounded analytic $C_0$-semigroup and has maximal $L^p$-regularity. Outside the class of $A_p$-weights, despite the potential applications to boundary value problems for the stochastic Navier-Stokes equations, there seem to exist no results on the analysis of the Stokes operator yet.

The Laplace operator with homogeneous Dirichlet boundary conditions $\Delta_\mathrm{Dir}$ in $L^p_\gamma(H),\,1<p<\infty,\,-p-1<\gamma<2p-1,$ has been studied in \cite{LindemulderVeraar2020}. There it is shown that $-\Delta_\mathrm{Dir}$ is sectorial and admits a bounded $H^\infty$-calculus in $L^p_\gamma(H)$, and a counterexample is given which demonstrates that the classical heat semigroup associated to $\Delta_\mathrm{Dir}$ does not preserve $L^p_\gamma(H)$ for $\gamma\notin(-1-p,2p-1).$ Boundedness of the $H^\infty$-calculus has furthermore been proven in weighted Sobolev spaces, with either Dirichlet or Neumann boundary condition, and in the setting of rough domains, in \cite{LLRV2025, LLRV2026}. On conical and polygonal domains, a mixed-weight theory addressing both stochastic and geometric boundary singularities has been developed for the Laplacian in \cite{CKL2019, CSW2024, CLW2025}. The weight used in these works incorporates both the distance to the vertices and the distance to the boundary, allowing the two sources of singular behaviour to be measured separately. See also the references in \cite{CSW2024} for further results in that direction. It would be interesting to investigate the Stokes problem in these settings in future work.

The article is structured as follows: In Section \ref{sec:WeightedSpaces} we define the weighted $L^p$-spaces and collect some properties. Section \ref{sec:PointwiseEstimates} recalls preliminary results from \cite{DHP2001} on the representation of the solution and on pointwise kernel estimates. In Section \ref{sec:Counterexamples}, we show in Proposition \ref{prop:CounterexampleHelmholtzDecomposition} that for $p-1<\gamma<2p-1$ there exists
$$0\neq h\in L^p_{\sigma,\gamma}(H)\cap G^p_\gamma(H),$$
and we furthermore prove Proposition \ref{prop:NegativeResults}. Section \ref{sec:Sectoriality} contains the proof of Theorem \ref{th:ResolventEstimates}. It also establishes an annihilator identity in Proposition \ref{prop:AnnihilatorCharacterization}, which might be of interest on its own. In Section \ref{sec:StokesOperator} we prove Theorem \ref{th:StokesOperator}. Proposition \ref{prop:LowWeightHarmonicTriviality} moreover records that, in contrast to the result from Proposition \ref{prop:CounterexampleHelmholtzDecomposition}, the intersection $L^p_{\sigma,\gamma}(H)\cap G^p_\gamma(H),\,-p-1<\gamma\leq-1,$ is trivial. In Proposition \ref{prop:MappingProperties} we complement the statement of Proposition \ref{prop:NegativeResults}\ref{eq:MappingFailure} by proving that the bound on $\gamma$ is sharp. Finally, in Section \ref{sec:HinftyCalculus} we give a proof of Theorem \ref{th:HinftyCalculus}.

\textbf{Notation.}
Let $X$ and $Y$ be Banach spaces over $\mathbb{C}$. By $\mathcal{L}(X,Y)$ we denote the space of all linear and bounded operators from $X$ to $Y$ endowed with the usual operator norm. Write $\mathcal{L}(X) = \mathcal{L}(X,X)$ for short. The dual space $X^\prime$ is understood as the space of continuous antilinear functionals on $X;$ adjoints are taken with respect to this convention. The identity operator is denoted by $\identity$. If $A$ is a linear operator, then domain, range, spectrum and resolvent set are denoted as usual by  $\mathcal{D}(A), \mathcal{R}(A), \sigma(A)$ and $\rho(A)$, respectively. When $\lambda \in \rho(A)$, write $R(\lambda, A) \coloneqq (\lambda - A)^{{-1}}$ for the resolvent operator.
For $p\in(1,\infty)$, denote the Hölder conjugate exponent by $p^\prime\coloneqq \frac{p}{p-1}$.

Denote by $C_0^\infty (H)=
\{\Psi\in C^\infty(H)\,\colon\, \support \Psi\subseteq K\subseteq H,\, K \text{ compact}\}$
the space of compactly supported smooth functions on $H$. When equipped with the usual inductive limit topology, we denote it by $\mathcal{D}(H)$ and by $\mathcal{D}^\prime(H)$ its topological dual. The space $\mathscr{S}(\R^n)$ denotes the Schwartz space of rapidly decaying smooth functions. For the Fourier transform and its inverse we adopt the convention
\begin{align*}
&\widehat f(\xi)=\mathcal F_{x\to\xi}f(\xi)\coloneqq\int_{\R^n}e^{-ix\cdot\xi}f(x)\dd x,\qquad&&f\in\mathscr S(\R^n),\,\xi\in\R^n,\\
&\mathcal F^{-1}_{\xi\to x}\psi(x)\coloneqq\frac{1}{(2\pi)^n}\int_{\R^n}e^{ix\cdot\xi}\psi(\xi)\dd\xi,
\qquad&&\psi\in\mathscr S(\R^n),\, x\in\R^n.
\end{align*}
For $\xi\in\mathbb{R}^n$, set $\rho(\xi)\coloneqq\abs{\xi}$. We simply write $\rho=\rho(\xi)$ when the dependence on $\xi$ is clear from the context.

Denote by $L_{\mathrm{loc}}^1(H)$ the space of all locally integrable functions, i.e. measurable functions on $H$ that are in $L^1(K)$ for any compact subset $K \subseteq H$.

Let $(S, \mathcal{A}, \mu)$ be any measure space. For $p\in [1,\infty)$ define the quasi-norm
$$\norm{f}_{L^{p,\infty}(S,\mu)} \coloneqq \sup_{r>0} r\,\Bigl(\mu(\abs{f} > r)\Bigl)^{\frac{1}{p}}.$$
The space of all $\mu$-measurable functions $f$ for which this value is finite is called \emph{weak-$L^p$} and denoted $L^{p,\infty}(S,\mu)$.

The set of positive integer numbers is denoted by $\N=\{1,2,3,\dots\},$ and $\N_0=\N\cup\{0\}.$ Elements of $\R^n$ or $\R^{n+1}$ are generally understood as column vectors. We write $\mathbb{C}_+$ for the positive complex half plane $\{z\in\mathbb{C}\,\colon\,\operatorname{Re}z>0\}.$ Given a complex number $z=re^{i\eta},\,r>0,\eta\in(-\pi,\pi),$ we take the principal branch $\sqrt{z}\coloneqq\sqrt{r}e^{i\eta/2}$ of the square root.

The symbol $\lesssim$ means that an inequality holds up to a constant; to be more specific, $\lesssim_\Phi$ may be used to state that this constant is allowed to depend on some variable $\Phi$. The symbol $\eqsim$ is used to denote two-sided estimates, that is, $a\eqsim b$ means $c a\leq b\leq Ca$ for constants $c,C>0.$

\section{\texorpdfstring{Weighted $L^p$-Spaces}{Weighted Lp-Spaces}}\label{sec:WeightedSpaces}
Let $\Omega \in \{\R^{n+1}, H\}$. A \emph{weight} on $\Omega$ is a measurable function $w$ with $0 < w < \infty$.

Although we will frequently consider weights outside this range, it will be worthwhile to recall the weight class $A_p$. Let $p \in (1,\infty)$. A weight $w$ is called an \emph{$A_p$-weight} or \emph{Muckenhoupt weight}, if there exists a $C\geq0$ such that
$$\Bigl(\frac{1}{\abs{Q}}\int_Q w(x)\dd x\Bigr)\Bigl(\frac{1}{\abs{Q}}\int_Q w(x)^{-\frac{1}{p-1}}\dd x\Bigr)^{p-1}\leq C$$
for every cube $Q \subset \Omega$ with sides parallel to the coordinate axes; one then writes $w\in A_p(\Omega),$ or simply $w\in A_p$ for short. The smallest such constant will be called the \emph{$A_p$-constant} of $w$ and denoted by $A_p(w)$.

The relevance of the weights $w\in A_p(\R^{n+1})$ stems from the fact that they are exactly the weights for which the maximal operator is bounded on $L^p_w(\R^{n+1}).$
\begin{theorem}[Muckenhoupt, {\cite[Chapter IV, Theorem 2.8]{GCRdF1985}}]
\label{th:maximal}
For $1<p<\infty$, $\mathcal{M}$ is bounded on $L^p_w(\R^{n+1})$ if and only if
$w\in A_p(\R^{n+1})$.
\end{theorem}

By Jensen's inequality the $A_p$-classes form a sequence of sets of weights which is increasing in $p$. Also observe that $w\in A_p(\Omega)$ automatically implies that $w\in L^1_{\mathrm{loc}}(\Omega).$ For $\gamma \in \R$, it is easily checked that weights of the form
\begin{equation*}
w(x) = \abs{x_{n+1}}^\gamma, \quad x \in \Omega,
\end{equation*} 
are Muckenhoupt weights if and only if $-1<\gamma<p-1.$ For further properties of $A_p$-weights see \cite{GCRdF1985, Grafakos2009}.

For a weight $w$ on $\Omega$ define the \emph{weighted $L^p$-space}
$L^p_w(\Omega)$ as the space of all measurable functions on $\Omega$ such that
$$\norm{u}_{L^p_w(\Omega)}=\norm{u}_{p,w,\Omega} = \Bigl(\int_\Omega \abs{u}^p w\dd x\Bigr)^{1/p}< \infty.$$
\begin{proposition}[Density of test functions]\label{prop:density}
Let $1\leq p<\infty$ and $w$ be a weight on $\Omega$ with $w\in L^1_{\mathrm{loc}}(\Omega).$ Then $C_0^\infty(\Omega)$ is dense in $L^p_w(\Omega)$.
\end{proposition}
\begin{proof}
The measure $w\dalt x$ is a Radon measure since $w\in L^1_{\mathrm{loc}}(\Omega),$ see e.g. \cite[Theorem 2.18]{Rudin1987}. Consequently,
$$C_c(\Omega)\coloneqq\{\Psi\in C(\Omega)\,\colon\,\support \Psi\subseteq K\Subset \Omega\}$$
is dense in $L^p_w(\Omega)$ by
\cite[Theorem 3.14]{Rudin1987}. Fix $0\neq g\in C_c(\Omega)$ and choose $\delta>0$ such that
$$K_\delta=\{x\in\R^{n+1}\,\colon\,\distance(x,\support g)\leq \delta\}\Subset\Omega.$$
Extend $g$ by zero; we denote this extension by $g$ again. It belongs to $C_c(\R^{n+1})$ and is uniformly continuous. A standard mollification argument \cite[Section 1.2]{Grafakos2008} then gives a family $g_\varepsilon,\,0<\varepsilon<\delta,$ such that every $g_\varepsilon$ is supported on $K_\delta$ and $g_\varepsilon\overset{\varepsilon\to 0}{\longrightarrow} g$ uniformly on $K_\delta.$ Since $w(K_\delta)<\infty,$ this implies
$$\norm{g_\varepsilon-g}_{p,w,\Omega}\leq\norm{g_\varepsilon-g}_{L^\infty(K_\delta)}\Bigl(\int_{K_\delta}w\dd x\Bigr)^{1/p}\overset{\varepsilon\to 0}{\longrightarrow} 0.$$
\end{proof}
If $1<p<\infty$ is fixed, the \emph{conjugate weight} is defined by
$$w^\prime\coloneqq w^{-\frac{1}{p-1}}.$$
This parallels the role of the conjugate $L^p$-exponent as can be seen by the following straightforward generalization of Hölder's inequality.
\begin{lemma}[Weighted Hölder inequality]\label{lem:weightedHölder}
Let $w$ be a weight and $p \in (1,\infty).$ If $u\in L_w^p(\Omega)$, $v\in L_{w^\prime}^{p^\prime}(\Omega),$ then
$$\norm{uv}_1 \leq \norm{u}_{p,w}\norm{v}_{p^\prime,w^\prime}.$$
\end{lemma}
\begin{proof}
Note that by definition $w^{1/p} (w^\prime)^{1/p^\prime} = 1$ and apply the classical unweighted Hölder inequality.
\end{proof}

\begin{remark}\label{rem:weightedLqDual}
In particular, by Lemma \ref{lem:weightedHölder} and standard $L^p$-duality we have the duality relation
$$(L_w^p(\Omega))^\prime = L_{w^\prime}^{p^\prime}(\Omega)$$
with respect to the dual pairing given by
$$\scp{\cdot}{\cdot}\colon L_{w^\prime}^{p^\prime}(\Omega) \times L_w^p(\Omega) \to \mathbb{C},\quad (u,v) \mapsto \scp{u}{v} \coloneqq \int_\Omega u \cdot \bar{v} \dd x.$$
\end{remark}
\noindent When $\Omega=H$ and $w=w_\gamma=\distance(\cdot,\partial H)^\gamma,$ write
$$\gamma^\prime = -\frac{\gamma}{p-1}.$$
With this notation it holds $(w_\gamma)^\prime = w_{\gamma^\prime}.$ Note that $\gamma^\prime \in (-p^\prime-1,-1)$ when $\gamma \in (p-1,2p-1)$ and vice versa. Hölder's inequality shows that if $w^\prime \in L^1_{\mathrm{loc}}(\Omega),$ we have $L^p_w(\Omega)\subseteq L^1_\mathrm{loc}(\Omega).$ The condition $w^\prime \in L^1_{\mathrm{loc}}(\Omega)$ is in particular satisfied if $\Omega=H$ and $w=w_\gamma,$ since $w_\gamma\in L^1_{\mathrm{loc}}(H)$ for all $\gamma\in\R$ and $(w_\gamma)^\prime=w_{\gamma^\prime}.$ 
\begin{remark}\label{rem:weightedLqReflexive}
The space $L^p_w(\Omega)$ is a reflexive Banach space, because $L^p(\Omega)$ is a reflexive Banach space and the mapping $u\mapsto uw^{1/p}$ is an isometric isomorphism from $L^p_w(\Omega)$ to $L^p(\Omega)$. Consequently, $L_{\sigma,w}^p(\Omega)$ is reflexive as a closed subspace of $L^p_w(\Omega)^{n+1}.$
\end{remark}

\section{Solution Representation and Pointwise Estimates}\label{sec:PointwiseEstimates}
In this section we recall the solution representation and kernel estimates of \cite{DHP2001}. For details on the derivation of these formulas we refer to that article. We split the solution $u=\mathcal{U}_\lambda f,\,f\in C_{0,\sigma}^\infty(H),$ into $u = (v, w)^\mathrm{T},$ where $v=(u_1,\dots,u_n)^\mathrm{T}$ denotes the tangential and $w=u_{n+1}$ the normal component, and correspondingly write the right hand side as $f=(f_v,f_w)^\mathrm{T}$ with $f_v=(f_1,\dots,f_n)^\mathrm{T},\,f_w=f_{n+1}.$

In \cite[Section 2]{DHP2001} it is shown that for $f\in C_{0,\sigma}^\infty(H)$ and $\lambda\in\Sigma_\theta$, the components $v$ and $w$ may be written as

\begin{equation}\label{eq:StokesSolutionFormula}
\begin{aligned}
    v &= (\lambda - \Delta^D)^{-1} f_v + v_2,\\
    w &= (\lambda - \Delta^D)^{-1} f_w + w_2,
\end{aligned}    
\end{equation}
where $\Delta^D$ denotes the Dirichlet Laplacian on $L^p(H)^{n+1}$, and $v_2, w_2$ are remainder terms. Recall the notation $\rho=\rho(\xi)=\abs{\xi},\,\xi\in\R^n.$ Let us define the complex-valued function
$$\omega_\lambda(r)\coloneqq (\lambda+r^2)^{1/2},\qquad r\geq 0,\,\lambda\in\C\setminus(-\infty,0],$$
using the principal square root. In Fourier expressions, we will abbreviate $\omega=\omega_\lambda(\rho(\xi))$ when the dependence on $\lambda$ and $\xi$ is clear from the context. Writing $v_1=(\lambda - \Delta^D)^{-1} f_v,\, w_1=(\lambda - \Delta^D)^{-1} f_w,$ one has
\begin{align}
\widehat{v}_{1}(\xi,y) &= \frac{1}{2\omega}\int_0^\infty (e^{-\omega\abs{y-s}} - e^{-\omega (y+s)})\widehat{f}_{v}(\xi,s)\dd s,\label{eq:v1hat}\\
\widehat{w}_{1}(\xi,y) &= \frac{1}{2\omega}\int_0^\infty (e^{-\omega\abs{y-s}} - e^{-\omega (y+s)})\widehat{f}_{w}(\xi,s)\dd s.\label{eq:w1hat}
\end{align}
In order to give explicit expressions for the remainders, define $\widehat{r}_v \,\colon\, \R^n\times \R_{>0} \times \R_{>0} \times \C\setminus(-\infty,0] \to \mathbb{C}^{n\times n}$ by
$$\widehat{r}_v(\xi,y,\tilde{y},\lambda) = \frac{e^{-\rho y} - e^{-\omega y}}{\omega - \rho} \frac{\xi \xi^\mathrm{T}}{\rho\omega}e^{-\omega \tilde{y}}$$
and
$\widehat{r}_w \,\colon\, \R^n \times \R_{>0} \times \R_{>0} \times \C\setminus(-\infty,0] \to \mathbb{C}^n$ by
\begin{equation}\label{eq:rwHat}
\widehat{r}_w(\xi,y,\tilde{y},\lambda) = \frac{e^{-\rho y} - e^{-\omega y}}{\omega - \rho} \frac{i\xi}{\omega}e^{-\omega \tilde{y}}.
\end{equation}
Here we extend $\frac{\xi\xi^\mathrm{T}}{\rho}$ continuously by zero at $\xi=0.$ For $\xi \in \R^n, y > 0$ the partial Fourier transforms of the remainder terms $v_2, w_2$ are then given by
\begin{align}
\widehat{v}_{2}(\xi,y) &= \int_0^\infty \widehat{r}_v(\xi,y,s,\lambda)\widehat{f}_{v}(\xi,s)\dd s = \frac{e^{-\rho y} - e^{-\omega y}}{\omega - \rho} \frac{\xi\xi^{\mathrm{T}}}{\rho\omega}\int_0^\infty e^{-\omega s}\widehat{f}_{v}(\xi,s)\dd s,\label{eq:v2hat}\\
\widehat{w}_{2}(\xi,y) &= \int_0^\infty \widehat{r}_w(\xi,y,s,\lambda)\cdot \widehat{f}_{v}(\xi,s)\dd s = \frac{e^{-\rho y} - e^{-\omega y}}{\omega - \rho} \frac{i\xi^{\mathrm{T}}}{\omega}\int_0^\infty e^{-\omega s}\widehat{f}_{v}(\xi,s)\dd s.\label{eq:w2hat}
\end{align}
Note that $f$ being solenoidal was used to express $w_2$ in terms of $f_v$ instead of $f_w$. Set
\begin{equation}\label{eq:rvrw}
r_v \coloneqq \mathcal{F}_{\xi \to x}^{-1} \;\widehat{r}_v,\quad r_w \coloneqq \mathcal{F}_{\xi \to x}^{-1} \;\widehat{r}_w.
\end{equation}
Then for $x \in \R^n,\,y > 0,$ the remainder terms $v_2, w_2$ may be represented as
\begin{align}
v_2(x,y) &= \int_0^\infty \int_{\R^n} r_v(x-\tilde{x}, y, s, \lambda) f_v(\tilde{x}, s) \dd \tilde{x} \dd s,\label{eq:v2}\\
w_2(x,y) &= \int_0^\infty \int_{\R^n} r_w(x-\tilde{x}, y, s, \lambda)\cdot f_v(\tilde{x}, s) \dd \tilde{x} \dd s.\label{eq:w2}
\end{align}
In particular, the map $(f_v,f_w)^\mathrm{T} \mapsto (v,w)^\mathrm{T}$ is linear. For the derivation of these formulas, see \cite[Sections 2 and 3]{DHP2001}. Under the change of coordinates
\begin{equation*}
\xi \to \abs{\lambda}^{\frac{1}{2}} \xi, \quad
y \to \frac{y}{\abs{\lambda}^{\frac{1}{2}}}, \quad
\tilde{y} \to \frac{\tilde{y}}{\abs{\lambda}^{\frac{1}{2}}}, \quad
x \to \frac{x}{\abs{\lambda}^{\frac{1}{2}}},
\end{equation*}
it holds
\begin{equation}\label{eq:rv_scaling}
r_v(x,y,\tilde{y},\lambda) = \abs{\lambda}^{\frac{n-1}{2}} r_v\left(\abs{\lambda}^{\frac{1}{2}}x, \abs{\lambda}^{\frac{1}{2}}y, \abs{\lambda}^{\frac{1}{2}}\tilde{y}, \frac{\lambda}{\abs{\lambda}}\right);
\end{equation}
analogously, for $r_w$ one has the scaling
\begin{equation}\label{eq:rw_scaling}
r_w(x,y,\tilde{y},\lambda) = \abs{\lambda}^{\frac{n-1}{2}} r_w\left(\abs{\lambda}^{\frac{1}{2}}x, \abs{\lambda}^{\frac{1}{2}}y, \abs{\lambda}^{\frac{1}{2}}\tilde{y}, \frac{\lambda}{\abs{\lambda}}\right),
\end{equation}
see \cite[pp.120-122]{DHP2001}. Thus, it suffices to consider the behaviour of the kernels for $\abs{\lambda}=1$. We have the following pointwise estimate on $r_v$ and $r_w$.

\begin{proposition}[cf. {\cite[Proposition 3.1 and Remark 3.2]{DHP2001}}]\label{prop:pointwiseEstimate0}
Let $\theta\in(0,\pi)$. Then there exist constants $M,c > 0$ such that
\begin{equation*}
\abs{r_v(x,y,\tilde{y}, \lambda)} + \abs{r_w(x,y,\tilde{y}, \lambda)}\leq M y e^{-c\tilde{y}}\int_0^\infty \frac{\tau^n}{1+y+\tau}e^{-c\tau(\abs{x}+y+\tilde{y})}\dd \tau,
\end{equation*}
where $x\in\R^n, y,\tilde{y}>0$ and $\lambda\in\mathbb{C}$ with $\abs{\lambda}=1$ and $\abs{\arg\lambda}\leq\theta<\pi.$
\end{proposition}
By $H^\infty(\Sigma_\theta)$ we denote the space of all (scalar-valued) bounded holomorphic functions defined on $\Sigma_\theta$, equipped with the sup-norm. Then define
$$H_0^\infty(\Sigma_\theta) = \left\{ f \in H^\infty(\Sigma_\theta) \colon \exists \, s> 0 \colon \abs{f(\lambda)} \lesssim \Bigl (\frac{\abs{\lambda}}{(1+\abs{\lambda})^2} \Bigr )^s \; \forall \lambda \in \Sigma_\theta \right\} \subset H^\infty(\Sigma_\theta).$$
If $A$ is a sectorial operator of angle $\omega_S(A)$ in a Banach space $X$, then for fixed $\theta\in(\omega_S(A),\pi)$ set
$$\Phi_A(f) = f(A) \coloneqq \frac{1}{2\pi \mathrm{i}} \int_{\partial \Sigma_\varphi} f(\lambda) R(\lambda, A) \dd \lambda \quad \in \mathcal{L}(X),$$
where $\varphi\in(\omega_S(A),\theta)$ is arbitrary, $\partial \Sigma_\varphi$ is positively oriented and $f \in H_0^\infty(\Sigma_\theta).$
We say that $A$ admits a \emph{bounded $H^\infty(\Sigma_\theta)$-calculus} if
$$\norm{\Phi_A(f)}_{\mathcal{L}(X)} \lesssim \norm{f}_{H^\infty(\Sigma_\theta)}$$
for all $f \in H_0^\infty(\Sigma_\theta)$. The angle of the $H^\infty$-calculus is defined as
$$\omega_H(A) = \inf\{\theta\in(\omega_S(A),\pi)\,\colon\, A\text{ has a bounded }H^\infty(\Sigma_\theta)\text{-calculus}\}.$$
For more details on the $H^\infty$-calculus, see e.g. \cite{Haase2006,HNVW2017}. Fix $h \in H_0^\infty(\Sigma_\theta)$ and $\varphi\in(0,\theta).$ For $f\in C_{0,\sigma}^\infty(H)$ write
\begin{equation}\label{eq:kThv}
\begin{aligned}
    k_{h,v}(x,y,\tilde{y}) &= \frac{1}{2\pi i} \int_{\partial \Sigma_\varphi} h(\lambda) r_v(x,y,\tilde{y},-\lambda) \dd \lambda, \quad x \in \R^n,\, y, \tilde{y} > 0,\\
    (T_{h,v}f)(x,y) &= \int_0^\infty \int_{\R^n} k_{h,v}(x-\tilde{x}, y, s) f_v(\tilde{x}, s) \dd \tilde{x} \dd s, \quad x \in \R^n,\, y>0,
\end{aligned}
\end{equation}
and likewise
\begin{equation}\label{eq:kThw}
\begin{aligned}
    k_{h,w}(x,y,\tilde{y}) &= \frac{1}{2\pi i} \int_{\partial \Sigma_\varphi} h(\lambda) r_w(x,y,\tilde{y},-\lambda) \dd \lambda, \quad x \in \R^n,\, y, \tilde{y} > 0,\\
    (T_{h,w}f)(x,y) &= \int_0^\infty \int_{\R^n} k_{h,w}(x-\tilde{x}, y, s)\cdot f_v(\tilde{x}, s) \dd \tilde{x} \dd s, \quad x \in \R^n,\, y>0.
\end{aligned}
\end{equation}
In \cite[Section 6]{DHP2001} the pointwise estimates
\begin{equation}\label{eq:pointwiseEstimate1}
\abs{k_{h,v}(x,y,\tilde{y})}+\abs{k_{h,w}(x,y,\tilde{y})} \lesssim \norm{h}_\infty k_1(x,y,\tilde{y}), \quad x \in \R^n, \,y,\tilde{y} > 0,
\end{equation}
\begin{equation}\label{eq:pointwiseEstimate2}
\int_{\R^n} \abs{k_1(x,y,\tilde{y})} \dd x \lesssim \frac{1}{y+\tilde{y}} \log(1+\frac{y}{\tilde{y}}), \quad y,\tilde{y} > 0,
\end{equation}
are proved.

\section{Counterexamples}\label{sec:Counterexamples}
In this section we present challenges for the theory of the Stokes operator on the upper weight range $p-1<\gamma<2p-1.$

We begin by introducing some notation. Let $\Delta_x\coloneqq\sum_{j=1}^n \partial_j^2$ be the Laplacian on $L^p(\R^n),$ and denote the square root of the negative Laplacian by $A\coloneqq(-\Delta_x)^{1/2}.$ Recall that since $A$ is a fractional power of the Laplacian, $-A$ generates a bounded analytic $C_0$-semigroup $P_t\coloneqq e^{-tA}$ in $L^p(\R^n)$ \cite[Theorem 15.2.7]{HNVW2023}, the \emph{Poisson semigroup}. It satisfies the usual semigroup growth estimate
\begin{equation}\label{eq:PoissonSemigroupBounds}
\norm{A^kP_t}_{\mathcal{L}(L^p(\R^n))}\leq C_{k,p}t^{-k},\qquad t>0,\quad k\in\N_0,
\end{equation}
see \cite[Theorem II.4.6]{EngelNagel2000}.

For $1<r\leq\infty$ define
\begin{equation}\label{eq:Er}
E^r(\R^n)\coloneqq \overline{C_0^\infty(\R^n)}^{\norm{\cdot}_{L^r(\R^n)}}.
\end{equation}
For finite $r,$ this is $L^r(\R^n);$ for $r=\infty,$ it is $C_0(\R^n)=\{f\in C(\R^n)\,\colon\,\lim_{\abs{x}\to\infty} f(x)=0\}.$ In every case, its norm is the corresponding $L^r(\R^n)$-norm.

For fixed $\lambda\in\C\setminus(-\infty,0],$ the operator $\lambda-\Delta_x$ is invertible and sectorial of angle $\abs{\arg\lambda}<\pi$ on $E^r(\R^n).$ Thus, by the same theorem as above\cite[Theorem 15.2.7]{HNVW2023}, the operator $A_\lambda\coloneqq(\lambda-\Delta_x)^{1/2}$ is invertible and sectorial of angle less than $\pi/2$ on $E^r(\R^n).$ Hence its associated semigroup is exponentially stable, that is, there are constants $c_r,C_r>0$ such that
\begin{equation}\label{eq:ExponentialStability}
\norm{e^{-tA_\lambda}}_{\mathcal{L}(E^r(\R^n))}\leq C_r e^{-c_rt},\qquad t>0.
\end{equation}
See also \cite[Theorem V.1.10]{EngelNagel2000}.

Denote by $R_j\coloneqq\partial_j A^{-1},\, j=1,\ldots,n,$ the partial Riesz transforms and let $\mathcal{R}=(R_1,\dots,R_n)$. The Riesz transforms $R_j$ are bounded on $L^p(\R^n)$ for every $1<p<\infty,$ see \cite[Corollary 4.2.8]{Grafakos2008}. 

For a fixed $t>0,$ we interpret a function $m\,\colon\,\R^n\times (0,\infty)\to\C,\,(\xi,t)\mapsto m(\xi,t),$ which satisfies $m(\cdot,t)\in L^\infty(\R^n),$ as a bounded linear operator in $\mathcal{L}(L^p(\R^n))$ which acts on $f\in L^p(\R^n)$ via scalar multiplication $(m(\cdot,t)f)(x)\coloneqq m(x,t)f(x),\,x\in \R^n$. Given such an $m(\cdot,t)$ for fixed $t,$ we define its associated \emph{Fourier multiplier operator} as
$$T_{m(\cdot,t)}\,\colon\,\mathscr{S}(\R^n)\to\mathscr{S}^\prime(\R^n),\,f\mapsto T_{m(\cdot,t)}f\coloneqq\mathcal{F}^{-1}_{\xi\to x}\bigl(m(\xi,t)\widehat{f}(\xi)\bigr)$$
and call $m(\cdot,t)$ the \emph{symbol} of the operator. The same notation is used for matrix-valued symbols, with matrix multiplication in the defining formula; in particular a row symbol acts on vector-valued inputs and produces a scalar output. In a slight abuse of notation, we also denote the extension of this operator to $L^p(\R^n),$ if it is well-defined, by $T_{m(\cdot,t)}$ again. Throughout this section, we will make particular use of the family of symbols
$$m_j(\xi,t)\coloneqq \abs{\xi}^je^{-t\abs{\xi}},\qquad \xi\in\R^n,\, t>0,$$
for $j\in\N_0.$

\begin{proposition}[Failure of the Helmholtz decomposition]\label{prop:CounterexampleHelmholtzDecomposition}
Fix $1<p<\infty$ and let $p-1<\gamma<2p-1.$ Then there exists an element
$$0\neq g\in L^p_{\sigma,\gamma}(H)\cap G^p_\gamma(H).$$
\end{proposition}
\begin{proof}
Fix $1<p<\infty$ and let $p-1<\gamma<2p-1.$ Consider $\varphi\in C^\infty_0(\R^n)$ with $\nabla \varphi\neq 0$ and the symbol $m_1(\xi,t)=\abs{\xi}e^{-t\abs{\xi}},\,\xi\in\R^n,t>0$. For $x\in\R^n,\,t>0$ define
$$g(x,t)=\nabla_{x,t}\,T_{m_1(\cdot,t)}\varphi(x)= (\nabla_x\,T_{m_1(\cdot,t)}\varphi(x),-T_{m_2(\cdot,t)}\varphi(x))^{\mathrm{T}}=(\nabla_x AP_t\varphi(x),-A^2P_t\varphi(x))^\mathrm{T},$$
where $m_2(\xi,t)=\abs{\xi}^2e^{-t\abs{\xi}},\,\xi\in\R^n,t>0.$ Note that $g\neq 0$, since otherwise $\abs{\xi}^2e^{-t\abs{\xi}}\widehat\varphi(\xi)=0$, and hence $\varphi=0$, contrary to the choice of $\varphi$. From the definition it follows that $g\in C^\infty(H).$ It holds
$$\Delta_x\bigl(T_{m_1(\cdot,t)}\varphi(x)\bigr)=-T_{m_3(\cdot,t)}\varphi(x)=-\partial_t^2\bigl(T_{m_1(\cdot,t)}\varphi(x)\bigr),\quad m_3(\xi,t)=\abs{\xi}^3e^{-t\abs{\xi}},\,\xi\in\R^n,t>0,$$
so
$$\divergence_{x,t} g(x,t) = \Delta_{x,t}\bigl(T_{m_1(\cdot,t)}\varphi(x)\bigr) = 0.$$

We show $g\in L^p_\gamma(H)^{n+1}$ and $AP_t\varphi\in L^1_{\mathrm{loc}}(\overline H)$ to conclude $g\in G^p_\gamma(H).$ Noting that $\nabla_x AP_t=\mathcal{R}A^2P_t$ and $\partial_tAP_t=-A^2P_t,$ the estimates on the Poisson semigroup \eqref{eq:PoissonSemigroupBounds} and Riesz-transform boundedness give
$$\norm{g(\cdot,t)}_{L^p}\lesssim\norm{A^2\varphi}_{L^p}<\infty\quad\text{for }0<t\leq1\text{ and}\qquad
\norm{g(\cdot,t)}_{L^p}\lesssim t^{-2}\norm{\varphi}_{L^p}\quad\text{for }t\geq1.$$
Therefore,
$$\norm{g}_{L^p_\gamma(H)^{n+1}}^p=\int_0^\infty\norm{g(\cdot,t)}_{L^p(\R^n)^{n+1}}^p t^\gamma\dd t\lesssim \int_0^1 \norm{g(\cdot,t)}_{L^p(\R^n)^{n+1}}^p t^\gamma\dd t + C \norm{\varphi}_{L^p(\R^n)}^p\int_1^\infty t^{\gamma-2p}\dd t.$$
The first integral is finite when $\gamma>-1,$ while the second integral is finite when $\gamma<2p-1,$ which proves $g\in L^p_\gamma(H)^{n+1}.$ Since $\varphi\in C_0^\infty(\R^n)\subseteq\mathcal{D}(A)$, one has $A\varphi\in L^p(\R^n)$ and
$$
AP_t\varphi=P_tA\varphi\overset{t\to 0}{\longrightarrow} A\varphi\qquad\text{in }L^p(\R^n)
$$
by strong continuity of the Poisson semigroup. If $K\subseteq\overline{H}$ is compact, choose a compact set $K_0\subseteq\R^n$ and some $T>0$ such that $K\subseteq K_0\times[0,T].$ Hölder's inequality then gives
$$\int_0^T\int_{K_0}\abs{AP_t\varphi(x)}\dd x\dd t\leq\abs{K_0}^{1/p^\prime}\int_0^T\norm{P_tA\varphi}_{L^p}\dd t\leq T\abs{K_0}^{1/p^\prime}\norm{A\varphi}_{L^p}<\infty,$$
hence $AP_t\varphi\in L^1_{\mathrm{loc}}(\overline H)$. 

It remains to show that we can approximate $g$ by functions in $C_{0,\sigma}^\infty(H).$ We will use a cutoff argument. For $r,R>0$ fix cutoff functions
\begin{align*}
\chi_r\in C^\infty([0,\infty)),\quad &\chi_r\equiv 0\text{ on }[0,r],\quad\chi_r\equiv 1 \text{ on }[2r,\infty),\\
\eta_R\in C^\infty(\R^{n+1}),\quad &\eta_R\equiv 1\text{ on }B_R,\quad\:\:\eta_R\equiv 0 \text{ on }\R^{n+1}\setminus B_{2R},
\end{align*}
with $\abs{\chi_r^{(k)}}\lesssim r^{-k}$ for $k\in\mathbb{N}_{>0}$ and $\abs{D^k\eta_R}\lesssim R^{-k}$ for $k\in\mathbb{N}_{>0}$ (we will only need $k=1,2$). Let $m_0(\xi,t)=e^{-t\abs{\xi}},\,\xi\in\R^n,t>0$ and $\Gamma(x,t)=-T_{m_0(\cdot,t)}\varphi(x)=-P_t\varphi(x),\,x\in\R^n,t>0$. For $r,R>0$ define
$$\Gamma_r(x,t)=\chi_r(t)\Gamma(x,t),\quad \Gamma_{r,R}(x,t)=\eta_R(x,t)\,\Gamma_r(x,t),$$
and
$$g_{r,\infty}=\bigl(\nabla_x\partial_t\Gamma_r,-\Delta_x\Gamma_r\bigr)^{\mathrm{T}},\qquad g_{r,R}=\bigl(\nabla_x\partial_t\Gamma_{r,R},-\Delta_x\Gamma_{r,R}\bigr)^{\mathrm{T}}.$$

Since $\partial_t$ and $\Delta_x$ commute,
$$\divergence_{x,t}g_{r,R}=\divergence_x\nabla_x\partial_t\Gamma_{r,R}-\partial_t\Delta_x\Gamma_{r,R}=0;$$
and since $\Gamma_{r,R}$ is compactly supported away from $t=0$, one thus has $g_{r,R}\in C_{0,\sigma}^\infty(H)$. 

We will show $g_{r,R}\to g_{r,\infty}$ in $L^p_\gamma$ as $R\to\infty$ and $g_{r,\infty}\to g$ in $L^p_\gamma$ as $r\to 0$. Let us begin with the latter. In view of
\begin{align*}
\nabla_x\partial_t\bigl(\chi_r(t)\Gamma(x,t)\bigr)&=\chi_r(t)\nabla_x\partial_t\Gamma(x,t)+\chi^\prime_r(t)\nabla_x\Gamma(x,t),\\
-\Delta_x\bigl(\chi_r(t)\Gamma(x,t)\bigr)&=-\chi_r(t)\Delta_x\Gamma(x,t),
\end{align*}
it holds
$$g_{r,\infty}-g=(\chi_r-1)g+(\chi^\prime_r\nabla_x\Gamma,0)^\mathrm{T}.$$
The first term $(\chi_r-1)g$ tends to $0$ as $r\to 0$ by dominated convergence. For the second term, $P_t$ and $\nabla_x$ commute on Schwartz functions because their Fourier symbols, $e^{-t\abs{\xi}}$ and $i\xi$, commute pointwise. Hence $\norm{\nabla_x\Gamma(\cdot,t)}_{L^p}=\norm{-P_t\nabla_x\varphi}_{L^p}\leq\norm{\nabla\varphi}_{L^p}$
which gives
$$\norm{\chi^\prime_r\nabla_x\Gamma}_{L^p_\gamma}^p\leq C \frac{1}{r^p} \int_r^{2r} t^\gamma\dd t\leq C r^{\gamma+1-p},$$
which tends to zero if $\gamma>p-1$. This shows $g_{r,\infty}\in L^p_\gamma$ and $g_{r,\infty}\overset{r\to 0}{\longrightarrow} g$ in $L^p_\gamma$.

It remains to justify the second cutoff. For fixed $r>0$, the product rule gives
\begin{align*}
\nabla_x\partial_t(\eta_R\Gamma_r)-\nabla_x\partial_t\Gamma_r
&=(\eta_R-1)\nabla_x\partial_t\Gamma_r  +(\nabla_x\eta_R)\partial_t\Gamma_r+(\partial_t\eta_R)\nabla_x\Gamma_r
 +(\nabla_x\partial_t\eta_R)\Gamma_r,\\
-\Delta_x(\eta_R\Gamma_r)+\Delta_x\Gamma_r
&=-(\eta_R-1)\Delta_x\Gamma_r
 -2\nabla_x\eta_R\cdot\nabla_x\Gamma_r-(\Delta_x\eta_R)\Gamma_r.
\end{align*}
All derivatives of $\eta_R$ are supported in $A_R=(B_{2R}\setminus B_R)\cap H$, and therefore these identities together with the derivative bounds for $\eta_R$ imply
\begin{align*}
\norm{g_{r,R}-g_{r,\infty}}_{L^p_\gamma}
&\lesssim \norm{\mathbf{1}_{H\setminus B_R}g_{r,\infty}}_{L^p_\gamma}
+R^{-1}\norm{\mathbf{1}_{A_R}\nabla_{x,t}\Gamma_r}_{L^p_\gamma}
+R^{-2}\norm{\mathbf{1}_{A_R}\Gamma_r}_{L^p_\gamma}.
\end{align*}
The first term tends to zero because $g_{r,\infty}\in L^p_\gamma(H)^{n+1}$. Since we are only interested in cases where $r$ is close to $0$, assume $0<r\leq1$. The semigroup estimates and the derivative of $\chi_r$ give
$$
\norm{\Gamma_r(\cdot,t)}_{L^p}\leq C\norm{\varphi}_{L^p},\qquad
\norm{\nabla_{x,t}\Gamma_r(\cdot,t)}_{L^p}\leq
\begin{cases}
C_r,&r<t<2,\\
Ct^{-1}\norm{\varphi}_{L^p},&t\geq2.
\end{cases}
$$
Noting that $\mathbf{1}_{A_R}(\cdot,t)\nabla_{x,t}\Gamma_r(\cdot,t)=0$ for all $t\notin [r,2R],$ we get for $R\geq2$ and using $\gamma>p-1$,
$$R^{-p}\norm{\mathbf{1}_{A_R}\nabla_{x,t}\Gamma_r}_{L^p_\gamma}^p\leq 2C_rR^{-p}+CR^{-p}\int_2^{2R}t^{\gamma-p}\dd t \leq 2C_r\bigl(R^{-p} +R^{\gamma+1-2p}\bigr).$$
Similarly, since $\gamma>-1$,
$$R^{-2p}\norm{\mathbf{1}_{A_R}\Gamma_r}_{L^p_\gamma}^p\leq CR^{-2p}\int_r^{2R}t^\gamma\dd t\leq C_r\bigl(R^{-2p}+R^{\gamma+1-2p}\bigr).$$
Both expressions tend to zero because $\gamma<2p-1$. Consequently $g_{r,R}\to g_{r,\infty}$ in $L^p_\gamma$ as $R\to\infty$ for each fixed $r$.

Finally, we use a diagonal sequence argument. For every $n\in\mathbb{N}$, choose $r_n\in(0,1/n)$ such that
$$\norm{g_{r_n,\infty}-g}_{L^p_\gamma}<\frac{1}{n}$$
and then choose $R_n\geq n$ so large that
$$\norm{g_{r_n,R_n}-g_{r_n,\infty}}_{L^p_\gamma}<\frac{1}{n}.$$
Set $g_n\coloneqq g_{r_n,R_n}$ and $g_{n,\infty}\coloneqq g_{r_n,\infty}$. Then $g_n\in C_{0,\sigma}^\infty(H)$, $\norm{g_n-g_{n,\infty}}_{L^p_\gamma}\to0$, and $g_{n,\infty}\to g$ in $L^p_\gamma$. Hence $g_n\to g$ as $n\to\infty$, which proves $g\in L^p_{\sigma,\gamma}(H)\cap G^p_\gamma(H)$.
\end{proof}

\begin{remark}
The vector field $g$ has a nonzero normal component on the boundary. Indeed, if $\nu=-e_{n+1}$ denotes the outward unit normal vector,
$$\nu\cdot g(\cdot,t)=A^2 P_t \varphi\overset{t\to 0}{\longrightarrow}A^2\varphi=-\Delta_x\varphi\neq 0\qquad\text{in }L^p(\R^n)$$
by strong continuity of $P_t$ and the fact that, by integration by parts, $0\neq\varphi\in C_0^\infty(\R^n)$ cannot be harmonic. In particular, this shows that a characterization like the one of the unweighted solenoidal $L^p$-space as
$$L^p_\sigma(H)=\{u\in L^p(H)^{n+1}\,\colon\,\divergence u=0,\,\nu\cdot u|_{\partial H}=0\}$$
is not valid in the case of $L^p_{\sigma,\gamma}(H),\,p-1<\gamma<2p-1.$
\end{remark}
In the remainder of this section we prove Proposition \ref{prop:NegativeResults}. Fix $\lambda\in\mathbb{C}\setminus(-\infty,0].$ For functions $b\in C_0^\infty(\R^n)^n,\,h\in C_0^\infty(0,\infty),$ we will consider $f=(f_v,f_w)^\mathrm{T}$ given by
\begin{equation}\label{eq:BaseForcing}
f_v(x,y)=b(x) h^\prime(y),\quad f_w(x,y)=-(\divergence_xb)(x)h(y),\quad x\in\R^n,y>0.
\end{equation}
By construction of $f$ it holds $f\in C_0^\infty(H)^{n+1}$ and
$$\divergence_xf_v+\partial_y f_w=(\divergence_xb)h^\prime-(\divergence_xb)h^\prime=0,$$
so $f\in C_{0,\sigma}^\infty(H).$ If we want to highlight the dependence of $f$ on $b$ and $h,$ we write $f=f[b,h].$
Let $M_-,M_+>0$ be such that $\support h\subseteq [M_-,M_+],$ and define
\begin{equation}\label{eq:mbclambda}
m_b\coloneqq \int_{\R^n}b(x) \dd x,\qquad c_{h,\lambda}\coloneqq\frac{1}{\sqrt{\lambda}}\int_0^\infty e^{-\sqrt{\lambda}s}h(s)\dd s.
\end{equation}
\begin{definition}\label{def:P0}
Define the linear spaces
\begin{align}
\mathscr{P}&\coloneqq\linspan\bigl\{f=f[b,h]\,\colon\,b\in C_0^\infty(\R^n)^n,\,h\in C_0^\infty(0,\infty)\bigr\},\\
\mathscr{P}_0&\coloneqq\linspan\bigl\{f=f[b,h]\,\colon\,b\in C_0^\infty(\R^n)^n,\,h\in C_0^\infty(0,\infty),\,\int_{\R^n} b(x)\dd x=0\bigr\}.
\end{align}
Here, the spans consist of finite linear combinations.
\end{definition}
It holds
$$\mathscr{P}_0\subseteq\mathscr{P}\subseteq C_{0,\sigma}^\infty(H).$$

Let us state a series of preparatory lemmas. For $1\leq p<\infty$ and $\psi\,\colon\,\R^n\to\mathbb{C},$ define the dilation operator
\begin{equation}\label{eq:DilationOperator}
(D_{\varepsilon,p} \psi)(x)\coloneqq\varepsilon^{\frac{n}{p}}\psi(\varepsilon x),\quad\varepsilon>0,\,x\in\R^n.
\end{equation}
It will be useful to observe two properties of this operator.
\begin{lemma}\label{lem:DilationProperties}
Let $1\leq p<\infty$ and $\psi\in L^p(\R^n).$
\begin{enumerate}[(i)]
    \item $\norm{D_{\varepsilon,p}\psi}_{L^p(\R^n)}=\norm{\psi}_{L^p(\R^n)}$, i.e. $D_{\varepsilon,p}$ is an isometry on $L^p(\R^n).$\label{eq:DilationPropertiesi}
    \item If $T_m$ is a Fourier multiplier operator on $L^p(\R^n)$ with symbol $m,$ then
    $$D_{\varepsilon,p}^{-1}T_m D_{\varepsilon,p}=T_{m_\varepsilon},$$
    where $m_\varepsilon(\xi)=m(\varepsilon\xi),\,\xi\in\R^n.$\label{eq:DilationPropertiesii}
\end{enumerate}
\end{lemma}
\begin{proof}
\begin{enumerate}[(i)]
    \item Immediate from the change of variables $x\mapsto \varepsilon x$.
    \item Let $\psi\in\mathscr{S}(\R^n).$ The Fourier transform of $D_{\varepsilon,p}\psi$ can be written as $\widehat{D_{\varepsilon,p} \psi}(\xi)=\varepsilon^{\frac{n}{p}-n}\widehat{\psi}(\frac{\xi}{\varepsilon}),$ thus 
\begin{align*}
\Bigl(D_{\varepsilon,p}^{-1}\mathcal{F}_{\xi\to x}^{-1}\bigl[m(\cdot)\widehat{D_{\varepsilon,p} \psi}(\cdot)\bigr]\Bigr)(x)
&= \Bigl(D_{\varepsilon,p}^{-1}\frac{1}{(2\pi)^n}\int_{\R^n}e^{ix\cdot\xi}\,m(\xi)\,\varepsilon^{\frac{n}{p}-n}\,\widehat{\psi}(\xi/\varepsilon)\dd\xi\Bigr)(x)\\
&= \Bigl(D_{\varepsilon,p}^{-1}\frac{1}{(2\pi)^n}\int_{\R^n}e^{ix\cdot\varepsilon\xi}\,m(\varepsilon\xi)\,\varepsilon^{\frac{n}{p}}\,\widehat{\psi}(\xi)\dd\xi\Bigr)(x)\\
&=\varepsilon^{\frac{n}{p}}\Bigl(D_{\varepsilon,p}^{-1}\mathcal{F}_{\xi\to x}^{-1}\bigl[m(\varepsilon\,\cdot)\widehat{\psi}(\cdot)\bigr]\Bigr)(\varepsilon x)\\
&=\Bigl(\mathcal{F}_{\xi\to x}^{-1}\bigl[m(\varepsilon\,\cdot)\widehat{\psi}(\cdot)\bigr]\Bigr)(x).
\end{align*}
Density of Schwartz functions and $L^p$-boundedness of $T_m$ yield the claim.
\end{enumerate}
\end{proof}
For the next lemma, recall $A_\lambda=(\lambda-\Delta_x)^{1/2}$ and the definition of $E^r(\R^n)$ from \eqref{eq:Er}. We use the notation $\frac{\tau_1}{a}=\infty$ when $a=0.$
\begin{lemma}\label{lem:SemigroupLrEstimate}
Fix $\lambda\in\C\setminus(-\infty,0],\,1<r\leq\infty,\,a\geq 0$ and $0<\tau_1<\tau_2<\infty.$ There exist constants $0<\varepsilon_0<\frac{\tau_1}{a}$ and $c,C>0$ such that for $F\in E^r(\R^n)^n$ and $0<\varepsilon<\varepsilon_0,$
\begin{equation}\label{eq:RescaledHeatSemigroupDecay}
\sup_{t\in[\tau_1,\tau_2]}\norm{\varepsilon^{-(n+1)}(e^{-(t/\varepsilon-a)A_\lambda}F)(\cdot/\varepsilon)}_{L^r(\R^n)^n}\leq C e^{-c/\varepsilon}\norm{F}_{E^r(\R^n)^n}.
\end{equation}
\end{lemma}
\begin{proof}
Note that for $0<\varepsilon<\frac{\tau_1}{a}$ it holds $t/\varepsilon-a\geq\tau_1/\varepsilon-a>0,\,t\in[\tau_1,\tau_2],$ so the left-hand side of \eqref{eq:RescaledHeatSemigroupDecay} is well-defined. The estimate
$$\norm{e^{-tA_\lambda}}_{\mathcal{L}(E^r(\R^n)^n)}\leq C_r e^{-c_rt},\qquad t>0,$$
from \eqref{eq:ExponentialStability} together with the substitution $x\mapsto \varepsilon x$ gives
\begin{align*}
\sup_{t\in[\tau_1,\tau_2]}\norm{\varepsilon^{-(n+1)}(e^{-(t/\varepsilon-a)A_\lambda}F)(\cdot/\varepsilon)}_{L^r(\R^n)^n}
&=\sup_{t\in[\tau_1,\tau_2]}\varepsilon^{-n/r^\prime-1}\norm{e^{-(t/\varepsilon-a)A_\lambda}F}_{E^r(\R^n)^n}\\
&\leq C\varepsilon^{-n/r^\prime-1} e^{-c/\varepsilon}\norm{F}_{E^r(\R^n)^n}.
\end{align*}
Potentially making $\varepsilon_0<\frac{\tau_1}{a}$ smaller and choosing $c^\prime,C^\prime>0$ such that
$$C\varepsilon^{-n/r^\prime-1} e^{-c/\varepsilon}\leq C^\prime e^{-c^\prime/\varepsilon},\qquad 0<\varepsilon<\varepsilon_0,$$
then shows the claim.
\end{proof}
\begin{remark}
The above proof did not depend on $\tau_2,$ so one could also take the supremum in \eqref{eq:RescaledHeatSemigroupDecay} over $t\geq\tau_1.$
\end{remark}
Recall the notation $\omega_\lambda(r)=(\lambda+r^2)^{1/2}.$ We derive a two-sided estimate on $\omega_\lambda(r).$
\begin{lemma}\label{lem:SquareRootEstimates}
Let $\lambda\in\C\setminus(-\infty,0].$ Then there exist constants $c_\lambda,C_\lambda>0$ such that
$$c_\lambda(1+r)\leq\operatorname{Re}\omega_\lambda(r)\leq\abs{\omega_\lambda(r)}\leq C_\lambda(1+r),\qquad r\geq0.$$
\end{lemma}
\begin{proof}
First assume $\abs{\lambda}=1.$ Observe that for $r\geq 0,$ one has
$$\abs{\lambda+r^2}\leq\abs{\lambda}+r^2\leq (1+r)^2,$$
and thus
$$\abs{\omega_\lambda(r)}=\abs{\sqrt{\lambda+r^2}}=\sqrt{\abs{\lambda+r^2}}\leq 1+r.$$
By our assumptions on $\lambda$ it holds $\operatorname{Re}\sqrt{\lambda}=\operatorname{Re}e^{i\arg(\lambda)/2}>0;$ define $c_\lambda\coloneqq\frac{1}{\sqrt{2}}\operatorname{Re}\sqrt{\lambda}>0.$ Then compute
$$1+r^2\operatorname{Re}\lambda=\operatorname{Re}\bigl(\overline{\lambda}(\lambda+r^2)\bigr)\leq\abs{\overline{\lambda}(\lambda+r^2)}=\abs{\lambda+r^2}=\abs{\omega_\lambda(r)}^2,$$
thus together with the double angle formula for the cosine $\cos^2(\eta)=\frac{1}{2}(1+\cos(2\eta))$ and the formula for the real part $\operatorname{Re}z=\frac{1}{2}(z+\overline{z}),$
$$c_\lambda^2(1+r)^2\leq (\operatorname{Re}\sqrt{\lambda})^2(1+r^2)\leq\frac{1}{2}(1+\operatorname{Re}\lambda)(1+r^2)\leq\frac{1}{4}(2\operatorname{Re}\lambda +2r^2+2\abs{\omega_\lambda(r)}^2)=\bigl(\operatorname{Re}\omega_\lambda(r)\bigr)^2.$$
It holds $\operatorname{Re}\omega_\lambda(r)\geq 0,$ so taking the square root altogether gives us
$$c_\lambda(1+r)\leq\operatorname{Re}\omega_\lambda(r)\leq\abs{\omega_\lambda(r)}\leq 1+r,\qquad r\geq0.$$
For arbitrary $\lambda\in\C\setminus(-\infty,0],$ use the preceding estimate together with the scaling 
$$\omega_\lambda(r)=\sqrt{\abs{\lambda}}\omega_{\lambda/\abs{\lambda}}(r/\sqrt{\abs{\lambda}}).$$
Then with $c_\lambda\coloneqq\frac{1}{\sqrt{2}}\operatorname{Re}\sqrt{\lambda/\abs{\lambda}}\min(\sqrt{\abs{\lambda}},1)>0$ and $C_\lambda\coloneqq\max(\sqrt{\abs{\lambda}},1)>0$ it holds
$$c_\lambda(1+r)\leq\operatorname{Re}\omega_\lambda(r)\leq\abs{\omega_\lambda(r)}\leq C_\lambda(1+r),\qquad r\geq0.$$
\end{proof}
Next, recall the solution representation $u=(v,w)^\mathrm{T}$ from \eqref{eq:StokesSolutionFormula}
\begin{equation*}
\begin{aligned}
    v &= (\lambda - \Delta^D)^{-1} f_v + v_2=v_1+v_2,\\
    w &= (\lambda - \Delta^D)^{-1} f_w + w_2=w_1+w_2.
\end{aligned}    
\end{equation*}
For $f=f[b,h],$ let us highlight the dependence on $b$ by writing $v[b]=v_1[b]+v_2[b],\,w[b]=w_1[b]+w_2[b].$ Here and in further notation, we are suppressing the dependence on $h$ and $\lambda,$ since these will always be clear from the context.

Recall the definition of $w_2$ from \eqref{eq:w2}
$$w_2(x,y) = \int_0^\infty \int_{\R^n} r_w(x-\tilde{x}, y, s, \lambda)\cdot f_v(\tilde{x}, s) \dd \tilde{x} \dd s,$$
with the kernel $r_w$ given through \eqref{eq:rwHat} together with \eqref{eq:rvrw}. 
Recall furthermore the representation of $\widehat{w_2}$ from \eqref{eq:w2hat}
$$\widehat{w}_{2}(\xi,y) = \frac{e^{-\rho y} - e^{-\omega y}}{\omega - \rho} \frac{i\xi^{\mathrm{T}}}{\omega}\int_0^\infty e^{-\omega s}\widehat{f}_{v}(\xi,s)\dd s,\qquad \xi\in\R^n,y>0.$$
For $f=f[b,h],$ let us rewrite $w_2(x,y)$ using integration by parts
$$\int_0^\infty e^{-\omega s}h^\prime(s)\dd s=[e^{-\omega s}h(s)]_{s=0}^\infty+\omega\int_0^\infty e^{-\omega s}h(s)\dd s=\omega\int_0^\infty e^{-\omega s}h(s)\dd s,$$
and the identity
$$\frac{1}{\omega-\rho}=\frac{\omega+\rho}{\lambda}$$
to get
\begin{equation}\label{eq:w2MultiplierFormulation}
w_{2}(x,y)=(T_{m_\lambda(\cdot,y)}b)(x),\quad m_\lambda(\xi,y)\coloneqq i\xi^\mathrm{T} \bigl(e^{-\rho y} - e^{-\omega y}\bigr)\frac{\omega + \rho}{\lambda} \int_0^\infty e^{-\omega s}h(s)\dd s,\quad x\in\R^n,\,y>0.
\end{equation}
We introduce some notation for $m_\lambda$. Set
\begin{equation*}
L(z)\coloneqq\int_0^\infty e^{-zs}h(s)\dd s,\:\: z\in\mathbb{C_+},\qquad B_\lambda(r)\coloneqq\frac{\omega_\lambda(r)+r}{\lambda}L\bigl(\omega_\lambda(r)\bigr),\quad r\geq 0.
\end{equation*}
One has
\begin{equation}\label{eq:mLambda}
m_\lambda(\xi,y)=i\xi^\mathrm{T}\bigl(e^{-\rho y}-e^{-\omega_{\lambda}(\rho)y}\bigr) B_\lambda(\rho),\qquad\xi\in\R^n,\,y>0.
\end{equation}
Note that $B_\lambda(0)=c_{h,\lambda}.$ Next, we introduce the rescaled Fourier multiplier operator
\begin{equation*}
T_{\varepsilon,t}^{(w_2)}\coloneqq \varepsilon^{-1}D_{\varepsilon,p}^{-1}T_{m_\lambda(\cdot,t/\varepsilon)}D_{\varepsilon,p},\qquad \varepsilon>0,\,t>0.
\end{equation*}
From Lemma \ref{lem:DilationProperties}\ref{eq:DilationPropertiesii} applied to \eqref{eq:mLambda} we see that its symbol is
\begin{equation*}
m_{\varepsilon,t}(\xi)=i\xi^\mathrm{T}\bigl(e^{-\rho t}-e^{-(t/\varepsilon)\omega_{\lambda}(\varepsilon\rho)}\bigr) B_\lambda(\varepsilon\rho),\qquad\xi\in\R^n.
\end{equation*}
In particular, we remark that the definition of $T_{\varepsilon,t}^{(w_2)}$ is actually independent of the choice of $1\leq p<\infty.$ Define the associated convolution kernel by
\begin{equation*}
k_{\varepsilon,t}(x)\coloneqq \bigl(\mathcal{F}^{-1}_{\xi\to x}m_{\varepsilon,t}\bigr)(x),\qquad x\in\R^n.
\end{equation*}
For $\varepsilon\to0,$ the family of symbols $m_{\varepsilon,t}$ converges pointwise to the symbol
\begin{equation*}
m_{0,t}(\xi)\coloneqq c_{h,\lambda}i\xi^\mathrm{T}e^{-t\abs{\xi}},\qquad\xi\in\R^n,
\end{equation*}
where $c_{h,\lambda}$ is the constant from \eqref{eq:mbclambda}. Its associated Fourier multiplier operator is
\begin{equation*}
T_{0,t}b\coloneqq c_{h,\lambda}\divergence_x(P_t b)\qquad b\in\mathscr{S}(\R^n)^n,\,t>0,
\end{equation*}
where $P_t$ acts componentwise. The associated convolution kernel is
\begin{equation}\label{eq:LimitKernel}
k_{0,t}(x)\coloneqq \bigl(\mathcal{F}^{-1}_{\xi\to x}m_{0,t}\bigr)(x)=c_{h,\lambda}(\nabla k_t(x))^\mathrm{T},\qquad x\in\R^n,
\end{equation}
where
$$k_t(x)=c_n\frac{t}{(t^2+\abs{x}^2)^{(n+1)/2}},\qquad x\in\R^n,\,t>0,\qquad c_n=\frac{\Gamma((n+1)/2)}{\pi^{(n+1)/2}},$$
denotes the Poisson kernel on $H,$ see \cite[Example 2.1.13]{Grafakos2008}. Hence $T_{0,t}b=k_{0,t}\ast b$ extends to an operator $L^r(\R^n)^n\to L^r(\R^n),\,1<r\leq\infty,$ by Young's convolution inequality. In the next lemma, we build upon the observed pointwise symbol convergence.

\begin{lemma}\label{lem:ScaledKernelConvergenceI}
Fix $1<r\leq\infty,\,h\in C_0^\infty(0,\infty),\,0<\tau_1<\tau_2<\infty$ and $\lambda\in\C\setminus(-\infty,0].$ Define $m_{\varepsilon,t},m_{0,t}$ and $k_{\varepsilon,t}, k_{0,t}$ as above. Then there exist $\varepsilon_0>0$ and $C>0$ such that it holds
\begin{equation*}
\sup_{t\in[\tau_1,\tau_2]}\norm{(k_{\varepsilon,t}-k_{0,t})}_{L^r(\R^n)^n}\leq C\varepsilon,\qquad 0<\varepsilon<\varepsilon_0.
\end{equation*}
In particular, for $g\in C_0^\infty(\R^n)^n,$
\begin{equation*}
\sup_{t\in[\tau_1,\tau_2]}\norm{(T_{\varepsilon,t}^{(w_2)}-T_{0,t})g}_{L^r(\R^n)}\leq C\varepsilon\norm{g}_{L^1(\R^n)^n},\qquad 0<\varepsilon<\varepsilon_0.
\end{equation*}
\end{lemma}
\begin{proof}
Fix $1<r\leq\infty,\,\lambda\in\C\setminus(-\infty,0],\,h\in C_0^\infty(0,\infty)$ and $J=[\tau_1,\tau_2]\subset(0,\infty).$ By assumption on $h$ we have $\operatorname{supp}h\subseteq[M_-,M_+]$. Repeated differentiation under the integral gives
$$L^{(j)}(z)=(-1)^j\int_{M_-}^{M_+} (s)^j e^{-zs}h(s)\dd s,\qquad z\in\C_+,\,j\in\N_0,$$
so there are constants $C_j>0,$ depending on $h,$ such that
\begin{equation}\label{eq:LEstimate}
\abs{L^{(j)}(z)}\leq C_je^{-M_-\operatorname{Re}z},\qquad j\in\N_0.
\end{equation}
In order to analyze the term $B_\lambda(\rho(\xi))$ we introduce the scalar-valued functions $\alpha_\lambda$ and $\beta_\lambda$ via their Fourier transforms
$$\widehat{\alpha_\lambda}(\xi)=\frac{\omega_\lambda(\rho(\xi))}{\lambda}L(\omega_\lambda(\rho(\xi))),\qquad \widehat{\beta_\lambda}(\xi)=\frac{1}{\lambda}L(\omega_\lambda(\rho(\xi))),\qquad\xi\in\R^n.$$
These are Schwartz functions because $\omega_\lambda(\rho(\xi))$ is smooth, its derivatives grow at most polynomially, and $L(z)$ and its derivatives decay exponentially in $z$ by \eqref{eq:LEstimate}. Consequently, also $\alpha_\lambda$ and $\beta_\lambda$ are Schwartz functions. By definition, we moreover have the identities
\begin{equation}\label{eq:AlphaBeta}
\int_{\R^n} \alpha_\lambda(\xi)\dd\xi=\widehat{\alpha}_\lambda(0)=c_{h,\lambda},\qquad i\xi^\mathrm{T}B_\lambda(\rho(\xi))=\bigl(\widehat{\nabla\alpha_\lambda}(\xi)+\rho(\xi)\widehat{\nabla\beta_\lambda}(\xi)\bigr)^\mathrm{T}.
\end{equation}
Recall that $-A=-(-\Delta_x)^{1/2}$ was the generator of the Poisson semigroup. Let us write
$$H_\lambda=(\nabla\alpha_\lambda+A\nabla\beta_\lambda)^\mathrm{T}.$$
Then $H_\lambda\in E^r(\R^n)^n;$ this follows from the fact that for $j\in\{1,\dots,n\},$ $\partial_j\alpha_\lambda\in\mathscr{S}(\R^n)\subset E^r(\R^n)$ and
$$\partial_j\beta_\lambda\in\mathscr{S}(\R^n)\subset\mathcal{D}(-\Delta_x)\subset\mathcal{D}(A)\subset E^r(\R^n),$$
so also $A\partial_j\beta_\lambda\in E^r(\R^n).$ Define the dilations
$$\alpha_{\lambda,\varepsilon}(x)=(D_{\varepsilon,1}^{-1}\alpha_\lambda)(x)=\varepsilon^{-n}\alpha_\lambda(x/\varepsilon),\qquad \beta_{\lambda,\varepsilon}(x)=(D_{\varepsilon,1}^{-1}\beta_\lambda)(x)=\varepsilon^{-n}\beta_\lambda(x/\varepsilon),\qquad x\in\R^n,\,\varepsilon>0.$$
Note that $D_{\varepsilon,1}^{-1}H_\lambda=(\varepsilon\nabla \alpha_{\lambda,\varepsilon} + \varepsilon^2 A\nabla\beta_{\lambda,\varepsilon})^\mathrm{T}.$ Let $K_t=(\nabla k_t)^\mathrm{T}$ and split the difference $k_{\varepsilon,t}-k_{0,t}$ into
\begin{equation}\label{eq:KernelSplit}
k_{\varepsilon,t}-k_{0,t}=(K_t\ast\alpha_{\lambda,\varepsilon}-c_{h,\lambda}K_t)-\varepsilon(\partial_tK_t)\ast\beta_{\lambda,\varepsilon}-E_{\varepsilon,t}.
\end{equation}
This split is immediate by splitting 
$$m_{\varepsilon,t}(\xi)-m_{0,t}(\xi)=i\xi^\mathrm{T}e^{-t\rho}(\widehat{\alpha}_\lambda(\varepsilon\xi)-c_{h,\lambda})+\varepsilon\rho\,i\xi^\mathrm{T}e^{-t\rho}\widehat{\beta}_\lambda(\varepsilon\xi)-\widehat{E}_{\varepsilon,t}(\xi)$$
on the Fourier side and applying the inverse Fourier transform. The remainder term $E_{\varepsilon,t}$ is given through
$$\widehat{E}_{\varepsilon,t}(\xi)=i\xi^\mathrm{T}e^{-(t/\varepsilon)\omega_\lambda(\varepsilon\rho)}B_\lambda(\varepsilon\rho)=\varepsilon^{-1}e^{-(t/\varepsilon)\omega_\lambda(\varepsilon\rho)}\widehat{H_\lambda}(\varepsilon\xi),\qquad \xi\in\R^n.$$
From Lemma \ref{lem:DilationProperties}\ref{eq:DilationPropertiesii} applied to \eqref{eq:AlphaBeta} we get
\begin{align*}
E_{\varepsilon,t}(x)
=\varepsilon^{-1}\bigl(D_{\varepsilon,1}^{-1}e^{-(t/\varepsilon)A_\lambda}H_\lambda\bigr)(x)=\varepsilon^{-n-1}\bigl(e^{-(t/\varepsilon)A_\lambda}H_\lambda \bigr)(x/\varepsilon),\qquad x\in\R^n.
\end{align*}
Apply Lemma \ref{lem:SemigroupLrEstimate} with $a=0$ and $F=H_\lambda$ to obtain
\begin{equation}\label{eq:EDecay}
\sup_{t\in J}\norm{E_{\varepsilon,t}}_{L^r(\R^n)^n}\leq C e^{-c/\varepsilon}\norm{H_\lambda}_{E^r(\R^n)^n}\leq C\varepsilon,\qquad 0<\varepsilon<\varepsilon_0
\end{equation}
after decreasing $\varepsilon_0$ for the second inequality if necessary.

Regarding the other two terms in \eqref{eq:KernelSplit}, from the decay of the Poisson semigroup it follows that $\nabla_{x,t}K_t\in L^r(\R^n),$ with norm uniformly bounded in $t\in[\tau_1,\tau_2].$ The first term can then be estimated by
\begin{equation}\label{eq:FirstTerm}
\begin{aligned}
\norm{K_t\ast\alpha_{\lambda,\varepsilon}-c_{h,\lambda}K_t}_{L^r(\R^n)}
&=\norm{\varepsilon^{-n}\int_{\R^n}K_t(\cdot-z)\alpha_\lambda(z/\varepsilon)\dd z-\int_{\R^n}K_t(\cdot)\alpha_\lambda(z)\dd z}_{L^r(\R^n)}\\
&=\norm{\int_{\R^n}\bigl(K_t(\cdot-\varepsilon z)-K_t(\cdot)\bigr)\alpha_\lambda(z)\dd z}_{L^r(\R^n)}\\
&\leq\int_{\R^n}\abs{\alpha_\lambda(z)}\norm{K_t(\cdot-\varepsilon z)-K_t}_{L^r(\R^n)}\dd z\\
&\leq\varepsilon\norm{\nabla K_t}_{L^r(\R^n)^{n\times n}}\int_{\R^n}\abs{z}\,\abs{\alpha_\lambda(z)}\dd z\\
&\leq C\varepsilon,
\end{aligned}
\end{equation}
where we used the substitution $z\mapsto\varepsilon z$ and \eqref{eq:AlphaBeta} together with Minkowski's integral inequality and the fact that $\alpha_\lambda$ is a Schwartz function. Similarly, Young's convolution inequality gives
\begin{equation}\label{eq:SecondTerm}
\varepsilon\norm{(\partial_t K_t)\ast\beta_{\lambda,\varepsilon}}_{L^r(\R^n)}\leq \varepsilon\norm{(\partial_t K_t)}_{L^r(\R^n)}\norm{\beta_{\lambda,\varepsilon}}_{L^1(\R^n)}\leq C\varepsilon.
\end{equation}
Observe that the estimates \eqref{eq:EDecay}, \eqref{eq:FirstTerm} and \eqref{eq:SecondTerm} are uniform in $t\in[\tau_1,\tau_2].$ Combining these estimates we thus receive the desired inequality for the remaining cases of $r.$

Concerning the second assertion, by Young's inequality we have
$$\sup_{t\in[\tau_1,\tau_2]}\norm{(k_{\varepsilon,t}-k_{0,t})\ast g}_{L^r(\R^n)}\leq C\varepsilon\norm{g}_{L^1(\R^n)}.$$
For fixed $\varepsilon, t>0,$ the symbols $m_{0,t}$ and $m_{\varepsilon,t}$ belong to $L^1(\R^n)^n.$ Since $g\in L^1(\R^n)^n$ implies $\widehat{g}\in L^\infty(\R^n)^n,$ the products $m_{0,t}\widehat{g}$ and $m_{\varepsilon,t}\widehat{g}$ are integrable. The above bound on $(k_{\varepsilon,t}-k_{0,t})\ast g$ and Fubini's theorem therefore give for $g\in L^1(\R^n)^n$
$$(k_{\varepsilon,t}-k_{0,t})\ast g=\mathcal{F}^{-1}_{\xi\to x}\Bigl((m_{\varepsilon,t}-m_{0,t})\widehat{g}\Bigr)=(T_{\varepsilon,t}^{(w_2)}-T_{0,t})g,$$
which proves the second claim.
\end{proof}

We will also need to control the Dirichlet Laplace part $(v_1,w_1)^\mathrm{T}=(\lambda-\Delta^D)^{-1}(f_v,f_w)^\mathrm{T}$ in \eqref{eq:StokesSolutionFormula}. Fix $h\in C_0^\infty(0,\infty).$ For $b\in C_0^\infty(\R^n)^n,$ recall that we denote the tangential and normal components of the Dirichlet contribution corresponding to the forcing $f[b,h]=\bigl(bh^\prime,-(\divergence_x b)h\bigr)^\mathrm{T}$ by $v_1[b]$ and $w_1[b],$ respectively, suppressing the dependence on the fixed function $h.$ For $g\in C_0^\infty(\R^n)^n,\,\varepsilon>0,$ and $t>0$ define the rescaled operators
\begin{equation*}
\begin{cases}
T_{\varepsilon,t}^{(v_1)}g&\coloneqq\varepsilon^{-1}D_{\varepsilon,p}^{-1}\bigl(v_1[D_{\varepsilon,p}g](\cdot,t/\varepsilon)\bigr),\\
T_{\varepsilon,t}^{(w_1)}g&\coloneqq\varepsilon^{-1}D_{\varepsilon,p}^{-1}\bigl(w_1[D_{\varepsilon,p}g](\cdot,t/\varepsilon)\bigr).
\end{cases}
\end{equation*}
Here the inner dilation acts componentwise on $g,$ while the outer dilation acts on the horizontal slice at height $t/\varepsilon.$ Like $T_{\varepsilon,t}^{(w_2)},$ the definition of these operators is independent of $1\leq p<\infty$ by Lemma \ref{lem:DilationProperties}\ref{eq:DilationPropertiesii}. For $g\in C_0^\infty(\R^n)^n$ set
\begin{equation}\label{eq:BBvar}
g_\varepsilon(x)\coloneqq(D_{\varepsilon,1}^{-1} g)(x)=\varepsilon^{-n}g(x/\varepsilon),\qquad g_{\varepsilon,p}(x)\coloneqq(D_{\varepsilon,p} g)(x)=\varepsilon^{n/p}g(\varepsilon x),\qquad1\leq p<\infty,\,\varepsilon>0.
\end{equation}
The next lemma formulates the key link between the rescaled operators $T_{\varepsilon,t}^{(v_1)},\,T_{\varepsilon,t}^{(w_j)}$ and the rescaled solution terms $v_1,\,w_j,\,j=1,2.$ See Remark \ref{rem:vReplacesw} for the reason we omit consideration of $v_2$ here.
\begin{lemma}\label{lem:RescaledOperatorIdentities}
Let $g\in C_0^\infty(\R^n)^n.$ For $\varepsilon>0$ and $t>0$ we have the identities
\begin{equation*}
T_{\varepsilon,t}^{(w_j)}g_\varepsilon = \varepsilon^{-n-1} w_j[g](\cdot/\varepsilon,t/\varepsilon),\qquad T_{\varepsilon,t}^{(w_j)}g = \varepsilon^{-n/p-1} w_j[g_{\varepsilon,p}](\cdot/\varepsilon,t/\varepsilon),\qquad j=1,2,
\end{equation*}
as well as
\begin{equation*}
T_{\varepsilon,t}^{(v_1)}g = \varepsilon^{-n/p-1} v_1[g_{\varepsilon,p}](\cdot/\varepsilon,t/\varepsilon).
\end{equation*}
\end{lemma}
\begin{proof}
Fix $g\in C_0^\infty(\R^n)^n$ and observe that $D_{\varepsilon,p}g_\varepsilon=\varepsilon^{-n/p^\prime}g.$ Applying Lemma \ref{lem:DilationProperties}\ref{eq:DilationPropertiesii} to $T_{\varepsilon,t}^{(w_1)}$ gives
\begin{equation*}
\begin{aligned}
(T_{\varepsilon,t}^{(w_1)}g_\varepsilon)(x)
&=\varepsilon^{-1}(D_{\varepsilon,p}^{-1}w_1[D_{\varepsilon,p} g_\varepsilon])(x,t/\varepsilon)\\
&=\varepsilon^{-1}\varepsilon^{-n/p^\prime}(D_{\varepsilon,p}^{-1}w_1[g])(x,t/\varepsilon)\\
&=\varepsilon^{-(n+1)}w_1[g](x/\varepsilon,t/\varepsilon),\qquad x\in\R^n,
\end{aligned}
\end{equation*}
and
\begin{equation*}
\begin{aligned}
(T_{\varepsilon,t}^{(w_1)}g)(x)
&=\varepsilon^{-1}(D_{\varepsilon,p}^{-1}w_1[D_{\varepsilon,p} g])(x,t/\varepsilon)\\
&=\varepsilon^{-1}(D_{\varepsilon,p}^{-1}w_1[g_{\varepsilon,p}])(x,t/\varepsilon)\\
&=\varepsilon^{-n/p-1}w_1[g_{\varepsilon,p}](x/\varepsilon,t/\varepsilon),\qquad x\in\R^n.
\end{aligned}
\end{equation*}
Repeating the same calculations verbatim for $T_{\varepsilon,t}^{(w_2)}$ and $T_{\varepsilon,t}^{(v_1)}$ concludes the claim.
\end{proof}
\begin{proposition}\label{prop:DirichletContribution}
Fix $1<r\leq\infty,\,0<\tau_1<\tau_2<\infty,\,h\in C_0^\infty(0,\infty)$ and $\lambda\in\C\setminus(-\infty,0].$ There exist constants $c,C>0$ and $\varepsilon_0>0$ such that for $g\in C_0^\infty(\R^n)^n,$
$$\sup_{t\in[\tau_1,\tau_2]}\Biggl(\norm{T_{\varepsilon,t}^{(v_1)}g}_{L^r(\R^n)}+\norm{T_{\varepsilon,t}^{(w_1)}g}_{L^r(\R^n)}\Biggr)\leq C e^{-c/\varepsilon}\norm{g}_{L^1(\R^n)^n},\qquad 0<\varepsilon<\varepsilon_0.$$
\end{proposition}
\begin{proof}
Let $1\leq p<\infty$ and $\lambda\in\C\setminus(-\infty,0].$ Let furthermore $g\in C_0^\infty(\R^n)^n$ and set $g_{\varepsilon,p}=D_{\varepsilon,p}g.$ Recall that $\support h\subseteq[M_-,M_+]$. By Lemma \ref{lem:RescaledOperatorIdentities},
\begin{equation}\label{eq:TIdentityI}
(T_{\varepsilon,t}^{(w_1)}g)(x)=\varepsilon^{-n/p-1}w_1[g_{\varepsilon,p}](x/\varepsilon,t/\varepsilon).
\end{equation}
The Green formula for the Dirichlet Laplace part \eqref{eq:w1hat} gives, for $y>M_+$,
\begin{equation}\label{eq:DirichletLaplaceGreenFormula}
\widehat{w}_1[g_{\varepsilon,p}](\xi,y)= -\frac{i\xi^{\mathrm T}\widehat{g_{\varepsilon,p}}(\xi)}{2\omega}
\int_{M_-}^{M_+}\bigl(e^{-\omega(y-s)}-e^{-\omega(y+s)}\bigr)h(s)\dd s.
\end{equation}
Choose $y_0>M_+$ and set
\begin{equation*}
j_{y_0}(\xi)\coloneqq -\frac{i\xi^{\mathrm T}}{2\omega}
\int_{M_-}^{M_+}\bigl(e^{-\omega(y_0-s)}-e^{-\omega(y_0+s)}\bigr)h(s)\dd s,\qquad \xi\in\R^n,
\end{equation*}
and $J_\lambda(x)=\mathcal{F}_{\xi\to x}^{-1}j_{y_0},\,x\in\R^n.$ Then, $J_\lambda$ is a Schwartz function in each component. To see this, observe that $\xi\mapsto\omega_\lambda(\rho(\xi))$ and its reciprocal are smooth and its derivatives grow at most polynomially. Choose $\delta=y_0-M_+>0.$ Differentiation under the integral together with Lemma \ref{lem:SquareRootEstimates} yields for every multi-index $\alpha\in\N_0^n$ a constant $N_\alpha\in\N_0$ such that
$$\abs{\partial_\xi^\alpha j_{y_0}(\xi)}\leq C_\alpha(1+\rho(\xi))^{N_\alpha} e^{-\delta c_\lambda(1+\rho(\xi))}.$$
Thus $j_{y_0}\in\mathscr{S}(\R^n)^{1\times n},$ which implies that $J_\lambda\in\mathscr{S}(\R^n)^{1\times n}.$ 
From \eqref{eq:TIdentityI} and \eqref{eq:DirichletLaplaceGreenFormula} we see that the convolution kernel of $T_{\varepsilon,t}^{(w_1)}$ is
$$k_{\varepsilon,t}^{(w_1)}=\varepsilon^{-(n+1)}\bigl(e^{-(t/\varepsilon-y_0)A_\lambda}J_\lambda\bigr)(\cdot/\varepsilon).$$
$J_\lambda$ is a Schwartz function, so it belongs to $E^r(\R^n)^n.$
Apply Lemma \ref{lem:SemigroupLrEstimate} with $a=y_0,\,F=J_\lambda.$ This gives
$$\sup_{t\in[\tau_1,\tau_2]}\norm{k_{\varepsilon,t}^{(w_1)}}_{L^r(\R^n)}\leq C e^{-c/\varepsilon}.$$
By Young's inequality it follows
$$\sup_{t\in[\tau_1,\tau_2]}\norm{T_{\varepsilon,t}^{(w_1)}g}_{L^r(\R^n)}\leq C e^{-c/\varepsilon}\norm{g}_{L^1(\R^n)}.$$

Repeating the same definitions and arguments with $v_1,\,k_{\varepsilon,t}^{(v_1)},\,T_{\varepsilon,t}^{(v_1)}$ in place of $w_1,\,k_{\varepsilon,t}^{(w_1)},\,T_{\varepsilon,t}^{(w_1)},$ respectively, proves the claim.
\end{proof}
Define
\begin{equation}\label{eq:LimitProfiles}
G^{(1)}(\cdot,t)\coloneqq k_{0,t}m_b= T_{0,t}(m_b\delta_0),\qquad G^{(2)}(\cdot,t)\coloneqq T_{0,t}b, \qquad t>0,
\end{equation}
where $\delta_0$ denotes the Dirac distribution at the origin and $m_b\delta_0=((m_b)_1\delta_0,\dots,(m_b)_n\delta_0)^\mathrm{T}.$
\begin{lemma}\label{lem:ScaledKernelConvergenceII}
Fix $1<p<\infty,\,1<r\leq\infty,\,0<\tau_1<\tau_2<\infty,\,b\in C_0^\infty(\R^n)^n,\,h\in C_0^\infty(0,\infty)$ and $\lambda\in\C\setminus(-\infty,0].$ Let $b_\varepsilon,\,b_{\varepsilon,p}$ be as in \eqref{eq:BBvar}. Then there exist $\varepsilon_0>0$ and $C>0$ such that
\begin{equation*}
\sup_{t\in[\tau_1,\tau_2]}\norm{\varepsilon^{-n-1}w[b](\cdot/\varepsilon,t/\varepsilon)-G^{(1)}(\cdot,t)}_{L^r(\R^n)}\leq C\varepsilon\Bigl(\norm{b}_{L^1(\R^n)}+\int_{\R^n} \abs{z} \abs{b(z)}\dd z\Bigr),\qquad 0<\varepsilon<\varepsilon_0,
\end{equation*}
and
\begin{equation*}
\sup_{t\in[\tau_1,\tau_2]}\norm{\varepsilon^{-n/p-1}w[b_{\varepsilon,p}](\cdot/\varepsilon,t/\varepsilon)-G^{(2)}(\cdot,t)}_{L^r(\R^n)}\leq C\varepsilon\norm{b}_{L^1(\R^n)},\qquad 0<\varepsilon<\varepsilon_0.
\end{equation*}
\end{lemma}
\begin{proof}
Fix $1<p<\infty,\,1<r\leq\infty,\,0<\tau_1<\tau_2<\infty,\,b\in C_0^\infty(\R^n)^n,\,h\in C_0^\infty(0,\infty)$ and $\lambda\in\C\setminus(-\infty,0].$ Observe the dilation identities
$$\int_{\R^n}b_\varepsilon(x)\dd x=\int_{\R^n}b(x)\dd x=m_b,\qquad \norm{b_\varepsilon}_{L^1(\R^n)^n}=\norm{b}_{L^1(\R^n)^n},$$
which are immediate from the substitution $x\mapsto \varepsilon x.$ Using the representations from Lemma \ref{lem:RescaledOperatorIdentities}, for the first assertion
we consider the decomposition
$$\varepsilon^{-n-1}w[b](\cdot/\varepsilon,t/\varepsilon)-G^{(1)}(\cdot,t)=T_{\varepsilon,t}^{(w_1)}b_\varepsilon + (T_{\varepsilon,t}^{(w_2)}-T_{0,t})b_\varepsilon + (T_{0,t}b_\varepsilon-k_{0,t}m_b),$$
while for the second assertion we decompose
$$\varepsilon^{-n/p-1}w[b_{\varepsilon,p}](\cdot/\varepsilon,t/\varepsilon)-G^{(2)}(\cdot,t)=T_{\varepsilon,t}^{(w_1)}b + (T_{\varepsilon,t}^{(w_2)}-T_{0,t})b.$$
Proposition \ref{prop:DirichletContribution} controls the first term in both decompositions: Choose $\varepsilon_0$ small enough that $Ce^{-c/\varepsilon}\leq C\varepsilon$ for all $0<\varepsilon<\varepsilon_0,$ where $c,C$ are the constants from Proposition \ref{prop:DirichletContribution}. Then
\begin{equation}\label{eq:FirstEstimate}
\begin{aligned}
\norm{T_{\varepsilon,t}^{(w_1)}b_\varepsilon}_{L^r(\R^n)}\leq C \varepsilon\norm{b}_{L^1(\R^n)^n},\\
\norm{T_{\varepsilon,t}^{(w_1)}b}_{L^r(\R^n)}\leq C \varepsilon\norm{b}_{L^1(\R^n)^n},
\end{aligned}
\end{equation}
uniformly in $t\in[\tau_1,\tau_2].$ Lemma \ref{lem:ScaledKernelConvergenceI} gives
$$\norm{(T_{\varepsilon,t}^{(w_2)}-T_{0,t})g}_{L^r(\R^n)}\leq C\varepsilon\norm{g}_{L^1(\R^n)}$$
uniformly in $t\in[\tau_1,\tau_2].$ Applying this to $g=b_\varepsilon$ and $g=b,$ respectively, yields
\begin{equation}\label{eq:SecondEstimate}
\begin{aligned}
\sup_{t\in[\tau_1,\tau_2]}\norm{(T_{\varepsilon,t}^{(w_2)}-T_{0,t})b_\varepsilon}_{L^r(\R^n)}\leq C\varepsilon\norm{b}_{L^1(\R^n)},\\
\sup_{t\in[\tau_1,\tau_2]}\norm{(T_{\varepsilon,t}^{(w_2)}-T_{0,t})b}_{L^r(\R^n)}\leq C\varepsilon\norm{b}_{L^1(\R^n)}.
\end{aligned}
\end{equation}
This already proves the second assertion. For the first assertion, it remains to estimate $T_{0,t}b_\varepsilon-k_{0,t}m_b.$ Recall the definition of the kernel
$$k_{0,t}(x)= \bigl(\mathcal{F}^{-1}_{\xi\to x}m_{0,t}\bigr)(x)=c_{h,\lambda}(\nabla k_t(x))^\mathrm{T},\qquad x\in\R^n,$$
where $k_t$ denotes the Poisson kernel on $H,$ from \eqref{eq:LimitKernel} and observe that by definition, $T_{0,t}b_\varepsilon=k_{0,t}\ast b_\varepsilon.$ Using Minkowski's integral inequality, we compute
\begin{align*}
\norm{k_{0,t}\ast b_\varepsilon-k_{0,t}m_b}_{L^r(\R^n)}
&=\norm{\varepsilon^{-n}\int_{\R^n}k_{0,t}(\cdot-\tilde{x})\,b(\tilde{x}/\varepsilon)\dd \tilde{x}-\int_{\R^n}k_{0,t}(\cdot)\,b(z)\dd z}_{L^r(\R^n)}\\
&=\norm{\int_{\R^n}\bigl(k_{0,t}(\cdot-\varepsilon z)-k_{0,t}(\cdot)\bigr)\,b(z)\dd z}_{L^r(\R^n)}\\
&\leq\int_{\R^n}\norm{k_{0,t}(\cdot-\varepsilon z)-k_{0,t}}_{L^r(\R^n)}\abs{b(z)}\dd z.
\end{align*}
For $r=\infty,$ the same inequality follows directly by taking the supremum inside the integral. For fixed $z,$ the fundamental theorem of calculus along the line segment from $x$ to $x-\varepsilon z$ gives
$$k_{0,t}(x-\varepsilon z)-k_{0,t}(x)=-\varepsilon\int_0^1(z\cdot\nabla_x) k_{0,t}(x-\theta\varepsilon z)\dd \theta.$$
Hence,
$$\norm{k_{0,t}(\cdot-\varepsilon z)-k_{0,t}}_{L^r(\R^n)}\leq \varepsilon\abs{z}\norm{\nabla_x k_{0,t}}_{L^r(\R^n)},$$
and we conclude
\begin{equation}\label{eq:ThirdEstimate}
\norm{T_{0,t} b_\varepsilon-k_{0,t}m_b}_{L^r(\R^n)}\leq \varepsilon\norm{\nabla_x k_{0,t}}_{L^r(\R^n)}\int_{\R^n} \abs{z}\abs{b(z)}\dd z.
\end{equation}
The term $\norm{\nabla_x k_{0,t}}_{L^r(\R^n)}$ is uniformly bounded in $t\in[\tau_1,\tau_2]$ since $k_t(x)=t^{-n}k_1(x/t),\, x\in\R^n,$ see \cite[Example 2.1.13]{Grafakos2009}, and therefore
$$\norm{\nabla_x k_{0,t}}_{L^r(\R^n)}=\norm{\nabla_x \bigl(c_{h,\lambda}(\nabla_x k_t)^\mathrm{T}\bigr)}_{L^r(\R^n)}=\abs{c_{h,\lambda}}t^{-n-2+n/r}\norm{\nabla_x^2k_1}_{L^r(\R^n)}.$$
Here we set $1/r=0$ for $r=\infty.$ Combining \eqref{eq:FirstEstimate}, \eqref{eq:SecondEstimate}, \eqref{eq:ThirdEstimate} concludes the claim.
\end{proof}

\begin{lemma}\label{lem:ProfileProperties}
Fix $\lambda\in\C\setminus(-\infty,0].$ Assume $b\in C_0^\infty(\R^n)^n,\,h\in C_0^\infty(0,\infty)$ are such that $c_{h,\lambda}\neq 0$ and $m_b\neq 0.$ For every compact interval $J=[\tau_1,\tau_2]\subset(0,\infty),$ both $G^{(1)}$ and $G^{(2)}$ are continuous on $\R^n\times J.$ Moreover, for every $t>0,$ neither $G^{(1)}(\cdot,t)$ nor $G^{(2)}(\cdot,t)$ is identically zero. 
\end{lemma}
\begin{proof}
Fix a compact interval $J=[\tau_1,\tau_2]\subset(0,\infty).$ The Fourier transforms of $G^{(1)}$ and $G^{(2)}$ are
\begin{align*}
\widehat{G^{(1)}}(\xi,t)&= c_{h,\lambda} i\xi^\mathrm{T} m_b e^{-t\rho(\xi)},\\
\widehat{G^{(2)}}(\xi,t)&= c_{h,\lambda} i\xi^\mathrm{T} \widehat{b}(\xi) e^{-t\rho(\xi)}\qquad \xi\in\R^n,\,t>0.
\end{align*}
Since $\abs{m_b}\leq\norm{b}_{L^1(\R^n)},\,|\widehat{b}(\xi)|\leq \norm{b}_{L^1(\R^n)},$ both Fourier symbols satisfy
$$\sup_{t\in J}\abs{\widehat{G^{(j)}}(\xi,t)}\leq C\rho(\xi) e^{-\tau_1\rho(\xi)},\qquad \xi\in\R^n,\,j=1,2.$$
Fix $(x,t)\in\R^n\times J$ and take any sequence $(x_k,t_k)\to(x,t)$ in $\R^n\times J.$ For every fixed $\xi\in\R^n,$ continuity of the exponential factor gives
$$e^{ix_k\cdot\xi}\widehat{G^{(j)}}(\xi,t_k)\to e^{ix\cdot\xi}\widehat{G^{(j)}}(\xi,t)\qquad\text{as }k\to\infty,\qquad j=1,2.$$
At the same time,
$$\abs{e^{ix_k\cdot\xi}\widehat{G^{(j)}}(\xi,t_k)-e^{ix\cdot\xi}\widehat{G^{(j)}}(\xi,t)}\leq 2C\rho(\xi) e^{-\tau_1\rho(\xi)},\qquad j=1,2.$$
The right-hand side is integrable in $\xi.$ The dominated convergence theorem therefore yields
$$G^{(j)}(x_k,t_k)\to G^{(j)}(x,t)\qquad\text{as }k\to\infty,\qquad j=1,2.$$

For the second assertion choose $j\in\{1,\dots,n\}$ with $(m_b)_j\neq 0.$ Along $\xi=s e_j,$ the first expression is nonzero for $s\neq 0.$ The second is nonzero for sufficiently small $s\neq 0,$ since $\widehat{b}_j(s e_j)\to (m_b)_j.$ Injectivity of the Fourier transform gives the assertion.
\end{proof}

We are now in a position to prove Proposition \ref{prop:NegativeResults}.
\begin{proof}[Proof of Proposition \ref{prop:NegativeResults}.]
We give explicit choices for the functions $b$ and $h$ from above. These choices will be employed in the proof of both assertions.

Let $\varphi\in C_0^\infty(\R^n)$ be nonnegative with $\int_{\R^n}\varphi\dd x=1.$ Set
$$a=\frac{1}{\sqrt{n}}(1,\dots,1)^\mathrm{T}\in\R^n$$
and choose $b=a\varphi.$ In particular, observe that $m_b=\int_{\R^n} b(x)\dd x=a\neq 0.$ Moreover, for $\lambda\in\C\setminus(-\infty,0]$ let $0\neq h\in C_0^\infty(0,\infty)$ with $h\geq 0$ and $\support h\subseteq [M,2M].$ Here, choose $M=M(\lambda)>0$ such that $M\bigl|\operatorname{Im}\sqrt{\lambda}\bigr|<\frac{\pi}{4}.$
Then $c_{h,\lambda}\neq0.$ Indeed, for every $y\in\support h$ it holds $\bigl|\operatorname{Im}\sqrt{\lambda}\bigr|y<\pi/2$ and hence $\cos(y\operatorname{Im}\sqrt{\lambda})>0.$ Since $0\neq h\geq 0,$
$$\operatorname{Re}\int_0^\infty e^{-\sqrt{\lambda}s}h(s)\dd s=\int_M^{2M} e^{-s\operatorname{Re}\sqrt{\lambda}} \cos(s\operatorname{Im}\sqrt{\lambda}) h(s)\dd s>0.$$
We start with assertion \ref{eq:MappingFailure}. Fix $1<p<\infty,\,\gamma\geq(n+1)(p-1)$ and $\lambda=1.$ 

Lemma \ref{lem:ProfileProperties} implies that with the choice $J=[\frac{1}{2},\frac{3}{2}]\subset(0,\infty),$ we can choose a bounded ball $B\subseteq\R^n$ and a constant $c_{\lambda,0}>0$ such that
\begin{equation}\label{eq:Gbound}
\abs{G^{(1)}(x,t)}\geq 2c_{\lambda,0}\quad \text{for all }(x,t)\in B\times J.
\end{equation}
Noting that $\norm{b}_{L^1(\R^n)}=1$ and $\int_{\R^n}\abs{z}\abs{b(z)}\dd z<\infty$ since $b$ is compactly supported, an application of Lemma \ref{lem:ScaledKernelConvergenceII} yields $\varepsilon_0>0$ such that
\begin{equation*}
\sup_{t\in[\frac{1}{2},\frac{3}{2}]}\norm{\varepsilon^{-(n+1)}w[b](\cdot/\varepsilon,t/\varepsilon)-G^{(1)}(\cdot,t)}_{L^\infty(\R^n)} \leq C\varepsilon,\qquad 0<\varepsilon<\varepsilon_0.
\end{equation*}
It will be sufficient to consider $t=1$ for the remaining part of our argument. Combined with \eqref{eq:Gbound} we get
$$\sup_{x\in B}\abs{\varepsilon^{-(n+1)}w[b](x/\varepsilon, 1/\varepsilon)-G^{(1)}(x,1)}\leq c_{\lambda,0}\qquad\text{for all }0<\varepsilon<\varepsilon_0.$$
Therefore, letting $y=1/\varepsilon$ and $y_0=1/\varepsilon_0,$
$$\abs{w[b](yx,y)}\geq c_{\lambda,0}y^{-(n+1)}\qquad y>y_0,\,x\in B.$$
Fix an arbitrary $R\geq y_0.$ Since $\abs{\mathcal{U}_\lambda f[b,h]}\eqsim\abs{v[b]}+\abs{w[b]}\geq\abs{w[b]},$
\begin{align*}
\norm{\mathcal{U}_\lambda f[b,h]}^p_{L^p_\gamma(H)^{n+1}}
&\geq\int_{y_0}^R\int_{yB}\abs{w[b](x,y)}^p \dd x y^\gamma\dd y\\
&\geq c_{\lambda,0}^p\int_{y_0}^R y^{\gamma-(n+1)p}\int_{yB}\dd x\dd y\\
&\geq c_{\lambda,0}^p\abs{B}\int_{y_0}^R y^{\gamma-(n+1)p}y^n\dd y
\end{align*}
and the integral tends to infinity for $R\to\infty$ if $\gamma\geq(n+1)(p-1).$

\vspace{8pt}
Next, we show that for $p-1<\gamma<(n+1)(p-1)$ there exists no bounded extension of the solution operator $\mathcal{U}_\lambda$ to the whole space $L^p_\gamma(H)^{n+1}.$ This is assertion \ref{eq:ResolventEstimateFailure}. Fix $1<p<\infty$ and let $p-1<\gamma<(n+1)(p-1).$ Moreover, fix $\lambda\in\mathbb{C}\setminus(-\infty,0].$ 

Recall the definition of $b_{\varepsilon,p}(x)=\varepsilon^{\frac{n}{p}}b(\varepsilon x).$ We are going to show that the data $f[b_{\varepsilon,p},h]$ given by 
$$f_v(x,y)=b_{\varepsilon,p}(x)h^\prime(y),\quad f_w(x,y)=-(\divergence_x b_{\varepsilon,p})(x)h(y),\qquad x\in\R^n,y>0,\qquad 0<\varepsilon,$$
satisfy
$$c<\norm{f[b_{\varepsilon,p},h]\,}_{L^p_\gamma(H)^{n+1}}<C,\qquad0<\varepsilon< 1,$$
for some constants $c,C>0$ depending on $\varphi,\,h,$ and for the normal part $w[b_{\varepsilon,p}](x,y)$ of the corresponding solution it holds
$$\norm{w[b_{\varepsilon,p}]\,}_{L^p_\gamma(H)}^p\geq C_{\lambda,p,\gamma,b,h}\varepsilon^{p-\gamma-1},\qquad 0<\varepsilon<\varepsilon_0,$$
for some $\varepsilon_0>0$ and a positive constant independent of $\varepsilon$.

First we compute the $L^p_\gamma$-norm of $f[b_{\varepsilon,p},h].$ One has
$$\norm{b_{\varepsilon,p}}_{L^p(\R^n)^n}=\norm{b}_{L^p(\R^n)^n},\quad \norm{\divergence_x b_{\varepsilon,p}}_{L^p(\R^n)}=\varepsilon\norm{\divergence_x b}_{L^p(\R^n)}$$
from the scaling properties of $D_{\varepsilon,p}.$ Thus,
$$\norm{f[b_{\varepsilon,p},h]\,}^p_{L^p_\gamma(H)^{n+1}}\eqsim\norm{b}^p_{L^p(\R^n)^n}\int_0^\infty\abs{h^\prime(s)}^p(s)^\gamma\dd s+\varepsilon^p\norm{\divergence_x b}^p_{L^p(\R^n)}\int_0^\infty\abs{h(s)}^p(s)^\gamma\dd s=C_1+\varepsilon^p C_2$$
for some constants $C_1, C_2>0.$ This shows that there are $c,C>0$ such that $c\leq\norm{f[b_{\varepsilon,p},h]\,}_{L^p_\gamma(H)^{n+1}}\leq C$ uniformly in $\varepsilon\in(0,1).$ We have already seen that by construction $f[b_{\varepsilon,p},h]\in C^\infty_{0,\sigma}(H).$

Next we calculate the $L^p_\gamma$-norm of $w[b_{\varepsilon,p}].$ Recall the definition of the rescaled operator
$$T_{\varepsilon,t}^{(w_2)}\coloneqq\varepsilon^{-1}D_{\varepsilon,p}^{-1}T_{m_\lambda(\cdot,t/\varepsilon)}D_{\varepsilon,p},\qquad \varepsilon>0,t>0.$$
Lemma \ref{lem:RescaledOperatorIdentities} shows
\begin{equation*}
\bigl(T_{\varepsilon,t}^{(w_1)}+T_{\varepsilon,t}^{(w_2)}\bigr)b = \varepsilon^{-n/p-1} w[b_{\varepsilon,p}](\cdot/\varepsilon,t/\varepsilon).
\end{equation*}
Lemma \ref{lem:ProfileProperties} implies that we can choose a bounded ball $B\subseteq\R^n,$ an interval $J=[\tau_1,\tau_2],\,0<\tau_1<\tau_2<\infty,$ and a constant $c_{\lambda,0}>0$ such that
$$\abs{G^{(2)}(x,t)}\geq 2c_{\lambda,0}\quad \text{for all }(x,t)\in B\times J.$$
Since $\norm{b}_{L^1(\R^n)^n}=1,$ the second assertion of Lemma \ref{lem:ScaledKernelConvergenceII} with $r=\infty$ then shows that 
there exists $\varepsilon_0>0$ such that
$$\sup_{(x,t)\in B\times J}\abs{\varepsilon^{-n/p-1} w[b_{\varepsilon,p}](x/\varepsilon,t/\varepsilon)-G^{(2)}(x,t)}\leq c_{\lambda,0}\qquad\text{for all }0<\varepsilon<\varepsilon_0.$$
Hence
$$\abs{\varepsilon^{-n/p-1} w[b_{\varepsilon,p}](x/\varepsilon,t/\varepsilon)}\geq c_{\lambda,0}\qquad\text{for all }0<\varepsilon<\varepsilon_0,\,(x,t)\in B\times J,$$
and thus
\begin{align*}
\norm{\mathcal{U}_\lambda f[b_{\varepsilon,p},h]}_{L^p_\gamma(H)^{n+1}}^p
&\geq\varepsilon^{-\gamma-1}\int_J\int_{\varepsilon^{-1}B}\abs{w[b_{\varepsilon,p}](x,t/\varepsilon)}^p\dd x\,t^\gamma\dd t\\
&=\varepsilon^{p-\gamma-1}\int_J \int_B \abs{\varepsilon^{-n/p-1} w[b_{\varepsilon,p}](z/\varepsilon,t/\varepsilon)}^p\dd z t^\gamma\dd t\\
&\geq c_{\lambda,0}^p\abs{B}\,\varepsilon^{p-\gamma-1}\int_Jt^\gamma\dd t.
\end{align*}
The uniformly bounded input norms therefore give
$$\frac{\abs{\lambda}\norm{\mathcal{U}_\lambda f[b_{\varepsilon,p},h]}_{L^p_\gamma(H)^{n+1}}}{\norm{f[b_{\varepsilon,p},h]}_{L^p_\gamma(H)^{n+1}}}\geq C\abs{\lambda}\varepsilon^{(p-\gamma-1)/p}\overset{\varepsilon\to 0}{\longrightarrow}\infty$$
since $\gamma> p-1.$
\end{proof}

\begin{remark}\label{rem:vReplacesw}
Note that $v_2=-\mathcal{R}w_2$ and $w_2=\sum_{j=1}^n R_j(v_2)_j,$ see \cite[p.119]{DHP2001}. Since the Riesz transforms $R_j$ are bounded on $L^p(\R^n)$ and Proposition \ref{prop:DirichletContribution} applies to both $v_1$ and $w_1$, the above proofs can be given with the tangential component $v$ in place of $w$ using the same arguments.
\end{remark}

\begin{remark}
The norm $\norm{\mathcal{U}_\lambda f[b_{\varepsilon,p},h]}_{L^p_\gamma(H)^{n+1}}^p$ in the proof of Proposition \ref{prop:NegativeResults}\ref{eq:ResolventEstimateFailure} is understood on the extended positive real line $[0,\infty].$
\end{remark}

\begin{remark}
The argument excluding cancellation with the Dirichlet Laplace part uses Green's formula for this term \eqref{eq:DirichletLaplaceGreenFormula} directly, so it does not require a global weighted Dirichlet resolvent estimate. The latter would through \cite[Theorem 5.7]{LindemulderVeraar2020} only be available for $\gamma<2p-1$.
\end{remark}

\section{Resolvent Estimates}
\label{sec:Sectoriality}
The aim of this section is to prove Theorem \ref{th:ResolventEstimates}.

\subsection{Lower Range}\label{sec:ResolventLow}
We begin with assertion \ref{eq:LowerCoreResolventEstimate}. Let $p \in (1,\infty),\,\theta\in(0,\pi)$ and let $\gamma\in(-p-1,p-1).$ Recall the definition of the convolution kernels $r_v,\,r_w$ from \eqref{eq:rvrw} and define
\begin{equation*}
\begin{aligned}
I_v(\lambda) &\coloneqq \int_0^\infty \left(\int_0^\infty \norm{r_v(\cdot,y,s,\lambda)}^{p^\prime}_{L^1(\R^n)}\, s^{\gamma^\prime} \dd s\right)^{p/p^\prime}\:y^\gamma\mathrm{d}y,\qquad \lambda\in\Sigma_\theta,\\
I_w(\lambda) &\coloneqq \int_0^\infty \left(\int_0^\infty \norm{r_w(\cdot,y,s,\lambda)}^{p^\prime}_{L^1(\R^n)}\, s^{\gamma^\prime} \dd s\right)^{p/p^\prime}\:y^\gamma\mathrm{d}y,\qquad \lambda\in\Sigma_\theta.
\end{aligned}
\end{equation*}
Then we have the following estimate on the nonlocal correction terms $v_2$ and $w_2$ from \eqref{eq:v2}, \eqref{eq:w2}.
\begin{lemma}\label{lem:v2w2Estimate}
Let $1<p<\infty,\,\theta\in(0,\pi)$ and let $-p-1<\gamma<p-1.$ Then
$$\norm{v_2}^p_{L^p_\gamma(H)^n} + \norm{w_2}^p_{L^p_\gamma(H)}\leq \bigl(I_v(\lambda) + I_w(\lambda)\bigr)\norm{f_v}^p_{L^p_\gamma(H)^n},\qquad\lambda\in\Sigma_\theta.$$
\end{lemma}
\begin{proof}
In view of the representation \eqref{eq:v2} of $v_2$, one may estimate
\begin{align*}
\norm{v_2}^p_{L^p_\gamma(H)^n} &= \int_0^\infty \int_{\R^n} \abs{\int_0^\infty \int_{\R^n} r_{v}(x-\tilde{x}, y, s, \lambda)f_v(\tilde{x}, s) \dd \tilde{x} \dd s}^p \dd x\:y^\gamma\mathrm{d}y\\
&\leq \int_0^\infty \int_{\R^n} \left(\int_0^\infty \abs{\int_{\R^n} r_v(x-\tilde{x}, y, s, \lambda)f_v(\tilde{x}, s) \dd \tilde{x}} \dd s\right)^p \dd x\:y^\gamma\mathrm{d}y\\
&= \int_0^\infty \int_{\R^n} \left(\int_0^\infty \abs{(r_v\ast f_v)(x, y, s, \lambda)} \dd s\right)^p \dd x\:y^\gamma\mathrm{d}y\\
&= \int_0^\infty \int_{\R^n} \norm{(r_v\ast f_v)(x, y, \cdot, \lambda)}_{L^1(0,\infty)^n}^p \dd x\:y^\gamma\mathrm{d}y\\
&\leq \int_0^\infty \left(\int_0^\infty \norm{(r_v\ast f_v)(\cdot, y, s, \lambda)}_{L^p(\R^n)} \dd s\right)^p\:y^\gamma\mathrm{d}y.
\end{align*}
Here we used Minkowski's integral inequality in the last step. An application of Young's convolution inequality and then the weighted Hölder inequality gives
\begin{align*}
&\int_0^\infty \left(\int_0^\infty \norm{(r_v\ast f_v)(\cdot, y, s, \lambda)}_{L^p(\R^n)} \dd s\right)^p\:y^\gamma\mathrm{d}y\\
\leq &\int_0^\infty \left(\int_0^\infty \norm{r_v(\cdot,y,s, \lambda)}_{L^1(\R^n)} \norm{f_v(\cdot,s)}_{L^p(\R^n)} \dd s\right)^p\:y^\gamma\mathrm{d}y\\
\leq &\norm{f_v}^p_{L^p_\gamma(H)^n} \int_0^\infty \left(\int_0^\infty \norm{r_v(\cdot,y,s, \lambda)}^{p^\prime}_{L^1(\R^n)}\, s^{\gamma^\prime} \dd s\right)^{p/p^\prime}\:y^\gamma\mathrm{d}y.
\end{align*}
This shows the estimate for $v_2.$ Replacing $r_v$ by $r_w$ in the above argument gives the estimate for $w_2$.
\end{proof}
It remains to show that the expressions $I_v(\lambda)$ and $I_w(\lambda)$ scale with $1/\abs{\lambda}^p.$ We perform the computations in a series of lemmas \ref{lem:ExponentialIntegral}, \ref{lem:rvrwNormEstimate}, \ref{lem:IvIwEstimate}.
\begin{lemma}\label{lem:ExponentialIntegral}
Define
$$E_1(x) = \int_x^\infty \frac{e^{-t}}{t}\dd t,\qquad x>0.$$
Then it holds
\begin{equation*}
e^x E_1(x) \leq \frac{1}{x},\qquad x>0,
\end{equation*}
and there is a constant $0<c_E<1$ such that
\begin{equation*}
\frac{1}{2}(x-\log x)\leq E_1(x)+c_E \leq x-\log x,\qquad 0<x\leq 1.
\end{equation*}
\end{lemma}
\begin{proof}
Clearly, $E_1(x)$ is well-defined and non-negative for all $x>0$ and decreasing in $x.$ The first inequality $$E_1(x) \leq \frac{1}{x}\int_x^\infty e^{-t} \dd t = \frac{e^{-x}}{x}$$
is immediate. Furthermore, it holds for all $x>0$
\begin{equation*}
\int_0^x \frac{1-e^{-t}}{t} \dd t - \int_0^1 \frac{1-e^{-t}}{t} \dd t = \int_1^x \frac{1-e^{-t}}{t} \dd t = \log(x) - \int_1^x \frac{e^{-t}}{t}\dd t = \log(x) - E_1(1) + E_1(x),
\end{equation*}
thus we find
\begin{equation*}
E_1(x) + c_E = \int_0^x \frac{1-e^{-t}}{t} \dd t -\log(x),
\end{equation*}
where
$$c_E \coloneqq \int_0^1 \frac{1-e^{-t}}{t} \dd t - \int_1^\infty \frac{e^{-t}}{t}\dd t.$$
Integration by parts gives
$$c_E= \int_0^\infty e^{-t} (-\log(t)) \dd t<\int_0^1 (-\log(t)) \dd t=1.$$
We have the two-sided estimate
$$\frac{3}{4}x\leq\int_0^x \frac{1-e^{-t}}{t} \dd t \leq x,\qquad 0<x\leq 1.$$
This follows immediately from the power series comparison $1-t\leq e^{-t}\leq 1-t+\frac{t^2}{2},\,0<t<1$ and integration. In particular, from the case $x=1$ we see $c_E>0.$ 
\end{proof}
\begin{remark}
$E_1(x)$ is called the \emph{exponential integral} and the constant $c_E$ is called the \emph{Euler-Mascheroni} constant. See e.g. \cite{AbramowitzStegun1964}, \cite{Jameson2015}.
\end{remark}
Next we derive a tangential $L^1$-estimate on the kernels $r_v$ and $r_w.$ Set
\begin{equation*}
C(y,\tilde{y},\lambda) \coloneqq c \abs{\lambda}^{\frac{1}{2}}(1+\abs{\lambda}^{\frac{1}{2}}y)(y+\tilde{y}),\qquad y,\tilde{y}>0,\,\lambda\in\Sigma_\theta.
\end{equation*}
\begin{lemma}\label{lem:rvrwNormEstimate}
Let $1<p<\infty,\,\theta\in(0,\pi)$ and let $-p-1<\gamma<p-1.$ There exists a constant $C>0$ such that for $y,\tilde{y}>0$ and $\lambda\in\Sigma_\theta,$ 
$$\int_{\R^n} \abs{r_v(x,y,\tilde{y},\lambda)} \dd x + \int_{\R^n} \abs{r_w(x,y,\tilde{y},\lambda)} \dd x \leq C y e^{-c\abs{\lambda}^{\frac{1}{2}}\tilde{y}} e^{C(y,\tilde{y},\lambda)}\int_{C(y,\tilde{y},\lambda)}^\infty \frac{e^{-t}}{t} \dd t,$$
where $c$ is the constant from Proposition \ref{prop:pointwiseEstimate0}.
\end{lemma}
\begin{proof}
Fix $y,\tilde{y}>0$ and $\lambda\in\Sigma_\theta.$ Combining the scaling \eqref{eq:rv_scaling}, \eqref{eq:rw_scaling} and Proposition \ref{prop:pointwiseEstimate0} one gets
\begin{align*}
\int_{\R^n} \abs{r_v(x,y,\tilde{y},\lambda)} \dd x &= \abs{\lambda}^{\frac{n-1}{2}} \int_{\R^n} \abs{r_v\left(\abs{\lambda}^{\frac{1}{2}}x, \abs{\lambda}^{\frac{1}{2}}y, \abs{\lambda}^{\frac{1}{2}}\tilde{y}, \frac{\lambda}{\abs{\lambda}}\right)} \dd x\\
&\leq M \abs{\lambda}^{\frac{n-1}{2}} \int_{\R^n} \abs{\lambda}^{\frac{1}{2}}y e^{-c\abs{\lambda}^{\frac{1}{2}}\tilde{y}}\int_0^\infty \frac{\tau^n}{1+\abs{\lambda}^{\frac{1}{2}}y+\tau}e^{-c\tau\abs{\lambda}^{\frac{1}{2}}(\abs{x}+y+\tilde{y})}\dd \tau \dd x\\
&= M\abs{\lambda}^{\frac{n}{2}}y e^{-c\abs{\lambda}^{\frac{1}{2}}\tilde{y}} \int_0^\infty \frac{\tau^n}{1+\abs{\lambda}^{\frac{1}{2}}y+\tau}e^{-c\tau\abs{\lambda}^{\frac{1}{2}}(y+\tilde{y})}\int_{\R^n} e^{-c\tau\abs{\lambda}^{\frac{1}{2}}\abs{x}}\dd x \dd \tau,
\end{align*}
where we used the Fubini-Tonelli theorem in the last line. Transforming into polar coordinates $x\mapsto(r,\vartheta)$ and then repeatedly integrating by parts, noting that
$$r^je^{-c\tau\abs{\lambda}^{\frac{1}{2}}r}\big|_{r=0}^\infty = 0 \quad \text{for all} \quad j>0,$$
gives
\begin{align*}
&\hphantom{=}M\abs{\lambda}^{\frac{n}{2}}y e^{-c\abs{\lambda}^{\frac{1}{2}}\tilde{y}} \int_0^\infty \frac{\tau^n}{1+\abs{\lambda}^{\frac{1}{2}}y+\tau}e^{-c\tau\abs{\lambda}^{\frac{1}{2}}(y+\tilde{y})}\int_{\R^n} e^{-c\tau\abs{\lambda}^{\frac{1}{2}}\abs{x}}\dd x \dd \tau\\
&=|\mathbb{S}^{n-1}|\,M\abs{\lambda}^{\frac{n}{2}}y e^{-c\abs{\lambda}^{\frac{1}{2}}\tilde{y}} \int_0^\infty \frac{\tau^n}{1+\abs{\lambda}^{\frac{1}{2}}y+\tau}e^{-c\tau\abs{\lambda}^{\frac{1}{2}}(y+\tilde{y})}\int_0^\infty e^{-c\tau\abs{\lambda}^{\frac{1}{2}}r}r^{n-1} \dd r \dd \tau\\
&=C_{n}Mc^{-n}\abs{\lambda}^{\frac{n}{2}}y e^{-c\abs{\lambda}^{\frac{1}{2}}\tilde{y}} \int_0^\infty  \frac{1}{1+\abs{\lambda}^{\frac{1}{2}}y+\tau}\frac{1}{\abs{\lambda}^{\frac{n}{2}}}e^{-c\tau\abs{\lambda}^{\frac{1}{2}}(y+\tilde{y})} \dd \tau\\
&=C\,y e^{-c\abs{\lambda}^{\frac{1}{2}}\tilde{y}} \int_0^\infty  \frac{e^{-c\tau\abs{\lambda}^{\frac{1}{2}}(y+\tilde{y})}}{1+\abs{\lambda}^{\frac{1}{2}}y+\tau} \dd \tau,
\end{align*}
where $\mathbb{S}^{n-1}$ is the unit sphere in $\R^n$ and $C_n\coloneqq |\mathbb{S}^{n-1}|\,\Gamma(n),\, C=M C_n c^{-n}.$ After the substitutions $u = 1+\abs{\lambda}^{\frac{1}{2}}y+\tau$ and then $t = c\abs{\lambda}^{\frac{1}{2}}(y+\tilde{y}) u$ one has
\begin{align*}
\int_0^\infty  \frac{e^{-c\tau\abs{\lambda}^{\frac{1}{2}}(y+\tilde{y})}}{1+\abs{\lambda}^{\frac{1}{2}}y+\tau} \dd \tau &= \int_{1+\abs{\lambda}^{\frac{1}{2}}y}^\infty  \frac{e^{-c\abs{\lambda}^{\frac{1}{2}}(u-1-\abs{\lambda}^{\frac{1}{2}}y)(y+\tilde{y})}}{u} \dd u\\
&= e^{C(y,\tilde{y},\lambda)} \int^\infty_{C(y,\tilde{y},\lambda)} \frac{e^{-t}}{t}\dd t.
\end{align*}
Since $r_w$ exhibits the same scaling \eqref{eq:rw_scaling} and Proposition \ref{prop:pointwiseEstimate0} also applies to $r_w$, repeating the same argument for $r_w$ in place of $r_v$ proves the claim.
\end{proof}
We are now ready to prove the estimate on $I_v$ and $I_w.$
\begin{lemma}\label{lem:IvIwEstimate}
Let $1<p<\infty,\,\theta\in(0,\pi)$ and let $-p-1<\gamma<p-1.$ There is a constant $C>0$ such that
$$I_v(\lambda) + I_w(\lambda) \leq \frac{C}{\abs{\lambda}^p},\qquad\lambda\in\Sigma_\theta.$$
\end{lemma}
\begin{proof}
Fix $\lambda\in\Sigma_\theta.$ In order to apply Lemma \ref{lem:ExponentialIntegral}, set $\kappa=\min(1/2,1/(4c))$ and choose $\delta=\delta(\lambda) \coloneqq \kappa\abs{\lambda}^{-\frac{1}{2}} > 0.$ Then for $y,\tilde{y}\in(0,\delta)$ it holds $0<C(y,\tilde{y},\lambda)\leq 2c\kappa(1+\kappa)\leq 1;$ in particular, $\log C(y,\tilde{y},\lambda)<0.$ Lemma \ref{lem:ExponentialIntegral} thus provides
\begin{equation*}
e^{C(y,\tilde{y},\lambda)}\int_{C(y,\tilde{y},\lambda)}^\infty \frac{e^{-t}}{t} \dd t \leq \frac{1}{C(y,\tilde{y},\lambda)},\qquad y,\tilde{y}>0,\,\lambda\in\Sigma_\theta,
\end{equation*}
and
\begin{equation*}
\frac{1}{2}\bigl(C(y,\tilde{y},\lambda)-\log C(y,\tilde{y},\lambda)\bigr)\leq\int_{C(y,\tilde{y},\lambda)}^\infty \frac{e^{-t}}{t} \dd t + c_E \leq C(y,\tilde{y},\lambda)-\log C(y,\tilde{y},\lambda),\qquad 0<y,\tilde{y}<\delta,\,\lambda\in\Sigma_\theta,
\end{equation*}
and Lemma \ref{lem:rvrwNormEstimate} shows
\begin{equation*}
\int_{\R^n} \abs{r_v(x,y,\tilde{y},\lambda)} \dd x \leq C y e^{-c\abs{\lambda}^{\frac{1}{2}}\tilde{y}} e^{C(y,\tilde{y},\lambda)}\int_{C(y,\tilde{y},\lambda)}^\infty \frac{e^{-t}}{t} \dd t,\qquad y,\tilde{y}>0,\,\lambda\in\Sigma_\theta.
\end{equation*}
Split $I_v$ as
\begin{align*}
I_v &= \int_0^\infty \left(\int_0^\infty \norm{r_v(\cdot,y,s,\lambda)}^{p^\prime}_1\, s^{\gamma^\prime} \dd s\right)^{p/p^\prime}\:y^\gamma\mathrm{d}y\\
&\lesssim \int_\delta^\infty \left(\int_0^\infty \left(\frac{y e^{-c\abs{\lambda}^{\frac{1}{2}}s}}{C(y,s,\lambda)}\right)^{p^\prime} s^{\gamma^\prime} \dd s\right)^{p/p^\prime}\:y^\gamma\mathrm{d}y\\
&\qquad + \int_0^\delta \left(\int_\delta^\infty \left(\frac{y e^{-c\abs{\lambda}^{\frac{1}{2}}s}}{C(y,s,\lambda)}\right)^{p^\prime} s^{\gamma^\prime} \dd s\right)^{p/p^\prime}\:y^\gamma\mathrm{d}y\\
&\qquad + \int_0^\delta \left(\int_0^\delta \left( y e^{-c\abs{\lambda}^{\frac{1}{2}}s} e^{C(y,s,\lambda)}\bigl(C(y,s,\lambda)-\log C(y,s,\lambda)-c_E\bigr) \right)^{p^\prime} s^{\gamma^\prime} \dd s\right)^{p/p^\prime}\:y^\gamma\mathrm{d}y\\
&\eqqcolon I_1 + I_2 + I_3.
\end{align*}
Let us estimate the integrals separately.
\begin{align*}
I_1 &= \int_\delta^\infty \left(\frac{y}{c \abs{\lambda}^{\frac{1}{2}}(1+\abs{\lambda}^{\frac{1}{2}}y)}\right)^p \left(\int_0^\infty \left(\frac{e^{-c\abs{\lambda}^{\frac{1}{2}}s}}{y+s}\right)^{p^\prime} s^{\gamma^\prime} \dd s\right)^{p/p^\prime}y^\gamma\mathrm{d}y\\
&\leq \int_\delta^\infty \left(\frac{y}{c\abs{\lambda}^{\frac{1}{2}}(1+\abs{\lambda}^{\frac{1}{2}}y)}\right)^p \left(\int_0^\infty \left(\frac{e^{-c\abs{\lambda}^{\frac{1}{2}}s}}{y}\right)^{p^\prime} s^{\gamma^\prime} \dd s\right)^{p/p^\prime}y^\gamma\mathrm{d}y\\
&= \int_\delta^\infty \left(\frac{1}{c\abs{\lambda}^{\frac{1}{2}}(1+\abs{\lambda}^{\frac{1}{2}}y)}\right)^p \left(\int_0^\infty e^{-ctp^\prime} t^{\gamma^\prime} \frac{1}{\abs{\lambda}^{\gamma^\prime/2}} \frac{1}{\abs{\lambda}^{1/2}} \dd t\right)^{p-1}y^\gamma\mathrm{d}y\\
&= c^{-p}\,C_1(\gamma)\, \abs{\lambda}^{-\frac{p}{2} + \frac{\gamma}{2} - \frac{p-1}{2}}\int_\delta^\infty \frac{1}{\bigl(1+\abs{\lambda}^{\frac{1}{2}}y\bigr)^p}\:y^\gamma\mathrm{d}y\\
&= c^{-p}\,C_1(\gamma)\, \abs{\lambda}^{-\frac{p}{2} + \frac{\gamma}{2} - \frac{p-1}{2} - \frac{\gamma}{2} - \frac{1}{2}} \int_\kappa^\infty \frac{s^{\gamma}}{(1+s)^p} \:\mathrm{d}s\\
&= c^{-p}\,\frac{C_1(\gamma) C_2(\gamma)}{\abs{\lambda}^p},
\end{align*}
where $C_1(\gamma) \coloneqq \bigl(\int_0^\infty e^{-ctp^\prime} t^{\gamma^\prime} \dd t\bigr)^{p-1}$ and $C_2(\gamma) \coloneqq \int_\kappa^\infty \frac{s^\gamma}{\left(1+s\right)^p}\mathrm{d}s$ are finite by our assumptions on $\gamma.$ The second term $I_2$ can be estimated as
\begin{align*}
I_2 &= \int_0^\delta \left(\frac{y}{c(1+\abs{\lambda}^{\frac{1}{2}}y)}\right)^p \left(\int_\delta^\infty \left(\frac{e^{-c\abs{\lambda}^{\frac{1}{2}}s}}{\abs{\lambda}^{\frac{1}{2}}(y+s)}\right)^{p^\prime} s^{\gamma^\prime} \dd s\right)^{p/p^\prime}y^\gamma\mathrm{d}y\\
&\leq \int_0^\delta \left(\frac{y}{c(1+\abs{\lambda}^{\frac{1}{2}}y)}\right)^p \left(\int_\kappa^\infty \left(\frac{e^{-ct}}{t}\right)^{p^\prime} t^{\gamma^\prime} \frac{1}{\abs{\lambda}^{\gamma^\prime/2}} \frac{1}{\abs{\lambda}^{1/2}} \dd t\right)^{p-1} y^\gamma\mathrm{d}y\\
&= c^{-p}\,C_3(\gamma)\abs{\lambda}^{\frac{\gamma}{2} - \frac{p-1}{2}} \int_0^\delta \frac{y^p}{\bigl(1+\abs{\lambda}^{\frac{1}{2}}y\bigr)^p}\:y^\gamma\mathrm{d}y\\
&=c^{-p}\,C_3(\gamma)\abs{\lambda}^{\frac{\gamma}{2} - \frac{p-1}{2} - \frac{p+\gamma}{2} - \frac{1}{2}} \int_0^\kappa \frac{s^{p+\gamma}}{(1+s)^p} \mathrm{d}s\\
&=c^{-p}\frac{C_3(\gamma)\,C_4(\gamma)}{\abs{\lambda}^p},
\end{align*}
where $C_3(\gamma)\coloneqq\left(\int_\kappa^\infty \left(\frac{e^{-ct}}{t}\right)^{p^\prime} t^{\gamma^\prime} \dd t\right)^{p-1}$ is finite for every fixed $\gamma\in\R$ and $C_4(\gamma) \coloneqq \int_0^\kappa \frac{s^{p+\gamma}}{(1+s)^p} \mathrm{d}s$ is finite 
since $p+\gamma>-1.$ Lastly, recall that $\log C(y,\tilde{y},\lambda) < 0$ for $0<y,\tilde{y}<\delta,$ so $I_3>0.$ We estimate
\begin{align*}
I_3 &= \int_0^\delta \left(\int_0^\delta \left( y e^{-c\abs{\lambda}^{\frac{1}{2}}s} e^{C(y,s,\lambda)}\bigl(C(y,s,\lambda)-\log C(y,s,\lambda)-c_E\bigr) \right)^{p^\prime} s^{\gamma^\prime} \dd s\right)^{p/p^\prime}y^\gamma\mathrm{d}y\\
&\lesssim \int_0^\delta \left(\int_0^\delta \bigl(1-\log \abs{\lambda}^{\frac{1}{2}}s\bigr)^{p^\prime} s^{\gamma^\prime} \dd s\right)^{p/p^\prime}y^{p+\gamma}\mathrm{d}y\\
&= \int_0^\delta \left(\int_0^\kappa \bigl(1-\log t\bigr)^{p^\prime} t^{\gamma^\prime} \frac{1}{\abs{\lambda}^{\gamma^\prime/2}}\frac{1}{\abs{\lambda}^{1/2}} \dd t\right)^{p/p^\prime}y^{p+\gamma}\mathrm{d}y\\
&= \abs{\lambda}^{\frac{\gamma}{2} - \frac{p-1}{2}} \int_0^\delta C_5(\gamma)\:y^{p+\gamma}\mathrm{d}y\\
&= \abs{\lambda}^{\frac{\gamma}{2} - \frac{p-1}{2} - \frac{p+\gamma}{2} - \frac{1}{2}}\: C_5(\gamma) \int_0^\kappa s^{p+\gamma}\mathrm{d}s\\
&= \frac{C_5(\gamma)\,C_6(\gamma)}{\abs{\lambda}^p},
\end{align*}
where $C_5(\gamma) \coloneqq \left(\int_0^\kappa \bigl(1-\log t\bigr)^{p^\prime} t^{\gamma^\prime} \dd t\right)^{p/p^\prime}$ and $C_6(\gamma) \coloneqq \int_0^\kappa s^{p+\gamma}\mathrm{d}s$ are again finite by our assumptions on $\gamma.$
\end{proof}

\begin{proof}[Proof of Theorem \ref{th:ResolventEstimates}\ref{eq:LowerCoreResolventEstimate}.]
Let $p \in (1,\infty),\,\theta\in(0,\pi)$ and let $\gamma\in(-p-1,p-1).$ Denote by $T=(T(z))_{z\in\mathbb{C}_+}$ the analytic continuation of the Dirichlet heat semigroup on $L^p_\gamma(H)^{n+1}$ and by $\Delta^D_{p,\gamma}$ its generator. In \cite[Proposition 5.1]{LindemulderVeraar2020} it is shown that for $f\in C_0^\infty(H)^{n+1}$
\begin{equation}\label{eq:LaplaceSectoriality}
\norm{(\lambda-\Delta^D_{p,\gamma})^{-1}f}_{L^p_\gamma(H)^{n+1}}\leq\frac{C}{\abs{\lambda}}\norm{f}_{L^p_\gamma(H)^{n+1}},\qquad\lambda\in\Sigma_\theta.
\end{equation}
For $g\in L^p(H)\cap L^p_\gamma(H)$, the resolvent in \eqref{eq:LaplaceSectoriality} agrees with the unweighted Dirichlet resolvent occurring in the representation formula \eqref{eq:StokesSolutionFormula}. Indeed, for positive $\lambda$, both operators are given by the Laplace transform of the same Dirichlet heat semigroup, and consistency for $\lambda\in\Sigma_\theta$ follows by analyticity.
We now use the decomposition from \eqref{eq:StokesSolutionFormula},
$$v=(\lambda-\Delta_{p,\gamma}^D)^{-1}f_v+v_2,\qquad w=(\lambda-\Delta_{p,\gamma}^D)^{-1}f_w+w_2.$$
Applying \eqref{eq:LaplaceSectoriality} componentwise gives
\begin{equation*}
\begin{aligned}
\norm{v}_{L^p_\gamma(H)^{n}}+\norm{w}_{L^p_\gamma(H)}
&\leq\frac{C}{\abs{\lambda}}\left(\norm{f_v}_{L^p_\gamma(H)^{n}}+\norm{f_w}_{L^p_\gamma(H)}\right)\\
&\quad+\norm{v_2}_{L^p_\gamma(H)^{n}}+\norm{w_2}_{L^p_\gamma(H)},\qquad\lambda\in\Sigma_\theta.
\end{aligned}
\end{equation*}
Now invoking Lemma \ref{lem:v2w2Estimate} together with Lemma \ref{lem:IvIwEstimate} concludes the claim.
\end{proof}

\begin{remark}
Let us mention that the estimates from \cite{DHP2001} for the unweighted situation could be adapted rather directly to the Muckenhoupt case $-1<\gamma<p-1.$ Our estimates need to be more subtle, however, in order to cover the values of $\gamma$ for which $w_\gamma$ lies outside the Muckenhoupt range.
\end{remark}

\begin{remark}
The argument above allows for inclusion of the case $\gamma=-1.$ What is expected to be different at this endpoint of the Muckenhoupt range (and similarly for the other endpoint $\gamma=p-1$) is the domain of the corresponding Stokes realization. Compare also the inclusion of $\gamma\in\{-1,p-1\}$ in the statements on the heat semigroup in \cite[Section 5.1]{LindemulderVeraar2020} with omission of these cases in the subsequent treatment of the Dirichlet Laplacian.
\end{remark}

We conclude this subsection with a statement on the mapping properties of $\mathcal{U}_\lambda.$ From the unweighted theory it is clear that for $g\in C_{0,\sigma}^\infty(H),$
$$\mathcal{U}_{\overline{\lambda}} g\in L_\sigma^p(H).$$
The following important lemma shows that on the lower range, we can approximate elements of the intersection $L^p_\sigma(H) \cap L^p_{\gamma}(H)^{n+1}$ by solenoidal test functions simultaneously in $L^p(H)^{n+1}$ and $L^p_{\gamma}(H)^{n+1}.$ This implies in particular that $\mathcal{U}_{\overline{\lambda}} g\in L^p_\gamma(H)^{n+1}$ is also an element of the weighted solenoidal space.
\begin{lemma}\label{lem:CoreProperty}
Let $1<p<\infty$ and let $\gamma<p-1.$ Consider the space $X_0 = L^p_\sigma(H) \cap L^p_{\gamma}(H)^{n+1},$ equipped with the intersection norm
$$\norm{f}_{X_0}\coloneqq\norm{f}_{L^p(H)^{n+1}}+\norm{f}_{L^p_{\gamma}(H)^{n+1}},\qquad f\in X_0.$$
Then $C_{0,\sigma}^\infty(H)$ is dense in $X_0$ with respect to $\norm{\cdot}_{X_0}.$ In particular,
$$\overline{X_0}^{L^p_{\gamma}(H)^{n+1}}\!\!\!= L^p_{\sigma,\gamma}(H).$$
\end{lemma}
\begin{proof}
Let $1<p<\infty.$

First assume $-1<\gamma<p-1.$ We will identify the space $(X_0,\norm{\cdot}_{X_0})$ with the solenoidal subspace of a Muckenhoupt-weighted $L^p$-space. Consider the weight $w_\gamma(x,y)=\abs{y}^\gamma$ as a weight on all of $\R^{n+1}$ and define the auxiliary weight $v_\gamma(x,y)=1+w_\gamma(x,y).$ It is easily seen that $w_\gamma\in A_p(\R^{n+1}).$ Let $Q\subseteq \R^{n+1}$ be a cube with sides parallel to the coordinate axes. Since $v_\gamma\geq \max(1,w_\gamma)$ we have
$$v_\gamma^{-\frac{1}{p-1}}\leq \min\bigl(1,w_\gamma^{-\frac{1}{p-1}}\bigr),$$
hence
\begin{align*}
\Bigl(\frac{1}{\abs{Q}}\int_Q v_\gamma\dd x\dd y\Bigr)\Bigl(\frac{1}{\abs{Q}}\int_Q v_\gamma^{-\frac{1}{p-1}}\dd x\dd y\Bigr)^{p-1}
&= \Bigl(1+\frac{1}{\abs{Q}}\int_Q w_\gamma \dd x\dd y\Bigr)\Bigl(\frac{1}{\abs{Q}}\int_Q v_\gamma^{-\frac{1}{p-1}}\dd x\dd y\Bigr)^{p-1}\\
&\leq 1 + \Bigl(\frac{1}{\abs{Q}}\int_Q w_\gamma \dd x\dd y\Bigr)\Bigl(\frac{1}{\abs{Q}}\int_Q w_\gamma^{-\frac{1}{p-1}}\dd x\dd y\Bigr)^{p-1}\\
&\leq 1 + A_p(w_\gamma)
\end{align*}
which shows that also $v_\gamma\in A_p(\R^{n+1}).$

Now let $f\in X_0.$ Then $f\in L^p_{v_\gamma}(H)^{n+1}$ since
\begin{equation}\label{eq:NormIdentification}
\norm{f}_{L^p_{v_\gamma}(H)^{n+1}}\eqsim_p\norm{f}_{L^p(H)^{n+1}}+\norm{f}_{L^p_{\gamma}(H)^{n+1}}=\norm{f}_{X_0}.
\end{equation}
By \cite[Theorem 5]{Froehlich2000}, there exists a decomposition
$$f=u+\nabla\pi,\qquad u\in L^p_{\sigma,v_\gamma}(H),\qquad \nabla\pi\in L^p_{v_\gamma}(H)^{n+1}.$$
Since the norm on $L^p_{v_\gamma}(H)^{n+1}$ is stronger than the one on $L^p(H)^{n+1}$ it follows
$$u\in L^p_\sigma(H),\qquad \nabla\pi\in L^p(H)^{n+1}$$
so $f=u+\nabla\pi$ is also an unweighted Helmholtz decomposition of $f.$ Since $f\in L^p_\sigma(H),$ uniqueness of that decomposition gives $\nabla\pi=0,\,f=u\in L^p_{\sigma,v_\gamma}(H).$ This means that there exists a sequence $(\varphi_k)_{k\in\N}\subseteq C_{0,\sigma}^\infty(H)$ converging to $f$ in $L^p_{v_\gamma}(H)^{n+1};$ by \eqref{eq:NormIdentification} this is precisely density of $C_{0,\sigma}^\infty(H)$ in $X_0$ with respect to $\norm{\cdot}_{X_0}.$

\vspace{8pt}
Now assume $\gamma\leq -1.$ Let $f\in X_0$ and let $(\varphi_k)_{k\in\N}\subseteq C_{0,\sigma}^\infty(H)$ converge to $f$ in unweighted $L^p.$ For $h>0$ and $(x,y)\in H$ consider the right translation
\begin{equation*}
(\tau_h g)(x,y)=
\begin{cases}
0 &0<y\leq h,\\
g(x,y-h)\quad &y>h.
\end{cases}
\end{equation*}
Observe that for all $h>0$ also $\tau_h \varphi_k\in C_{0,\sigma}^\infty(H),\,k\in\N.$ Moreover, since $\gamma<0,$
\begin{equation}\label{eq:TranslationContractivity}
\norm{\tau_h f}_{L^p_\gamma(H)^{n+1}}^p=\int_H\abs{f}^p(s+h)^\gamma\dd x\dd s\leq \norm{f}_{L^p_\gamma(H)^{n+1}}^p.
\end{equation}
By dominated convergence this shows that for $\phi\in C_0^\infty(H)^{n+1}$ one has
\begin{equation}\label{eq:TranslatedWeightedConvergenceI}\tau_h \phi\overset{h\to 0}{\longrightarrow}\phi\quad \text{in } L^p_\gamma(H)^{n+1},
\end{equation}
since $\tau_h \phi\overset{h\to 0}{\longrightarrow}\phi$ pointwise. Density of $C_0^\infty(H)^{n+1}$ in $L^p_\gamma(H)^{n+1}$ and \eqref{eq:TranslationContractivity} then imply
\begin{equation}\label{eq:TranslationWeightedConvergenceII}
\tau_h f\overset{h\to 0}{\longrightarrow}f\quad \text{in } L^p_\gamma(H)^{n+1}.
\end{equation}
Furthermore, for fixed $h>0$ it holds
\begin{equation}\label{eq:TranslationEstimate}
\norm{\tau_h(\varphi_k-f)}_{L^p_\gamma(H)^{n+1}}^p\leq h^\gamma\norm{\varphi_k-f}_{L^p(H)^{n+1}}^p
\end{equation}
because $(s+h)^\gamma\leq h^\gamma$ for all $s>0.$ Fix a sequence $(h_j)_{j\in\N}\to 0$ and then choose the sequence $(k_j)_{j\in\N}$ so that
$$\norm{\varphi_{k_j}-f}_{L^p(H)^{n+1}}\leq \frac{1}{j\,(1+(h_j)^{\gamma/p})},\qquad j\in\N.$$
Then \eqref{eq:TranslatedWeightedConvergenceI}, \eqref{eq:TranslationWeightedConvergenceII}, \eqref{eq:TranslationEstimate} show that $\tau_{h_j}\varphi_{k_j}\to f$ in $X_0.$

\vspace{8pt}
To conclude, for arbitrary $\gamma<p-1$ the density of $C_{0,\sigma}^\infty(H)$ in $X_0$ with respect to $\norm{\cdot}_{X_0}$ implies in particular
$$X_0=\overline{C_{0,\sigma}^\infty(H)}^{X_0}\!\subseteq\overline{C_{0,\sigma}^\infty(H)}^{L^p_\gamma(H)^{n+1}}\!\!\!=L^p_{\sigma,\gamma}(H)$$
and taking the closure of both sides with respect to $\norm{\cdot}_{L^p_\gamma(H)^{n+1}}$ gives $\overline{X_0}^{L^p_{\gamma}(H)^{n+1}}\!\!\!\subseteq L^p_{\sigma, \gamma}(H).$ The reverse inclusion is immediate since $X_0$ contains $C_{0,\sigma}^\infty(H).$
\end{proof}
\subsection{Upper Range}\label{sec:ResolventHigh}
It remains to prove Theorem \ref{th:ResolventEstimates}\ref{eq:UpperCoreResolventEstimate}. Let $1<p<\infty$ and let $p-1\leq\gamma<2p-1.$ Set
$$q=p^\prime,\qquad \beta=-\frac{\gamma}{p-1}.$$
Then $-q-1<\beta\leq -1.$ Lemma \ref{lem:CoreProperty} shows that for $g\in C_{0,\sigma}^\infty(H),$
\begin{equation*}
\mathcal{U}_{\overline{\lambda}} g\in L_{\sigma,\beta}^q(H).
\end{equation*}

We generalize a result due to Kakizawa \cite[Lemma 3.2]{Kakizawa2018}, in which $w\in A_p$ is assumed, to spaces with possibly non-$A_p$ power weights, including both the lower and upper range. For a subset $M\subseteq X$ of a Banach space $X,$ denote its annihilator by $M^\perp\coloneqq\{f\in X^\prime\,\colon\,f(x)=0\text{ for every }x\in M\}.$
\begin{proposition}\label{prop:AnnihilatorCharacterization}
Let $\gamma\in(-p-1,2p-1).$ It holds
$$\bigl(L^p_{\sigma,\gamma}(H)\bigr)^\perp=G^{p^\prime}_{\gamma^\prime}(H),\quad \bigl(G^p_{\gamma}(H)\bigr)^\perp=L^{p^\prime}_{\sigma,\gamma^\prime}(H).$$
In particular, $L^p_{\sigma,\gamma}(H)$ and $G^p_\gamma(H)$ are closed (as subspaces of $L^p_\gamma(H)^{n+1}$).
\end{proposition}
\begin{proof}
We prove both inclusions of the first identity. Density of $C_{0,\sigma}^\infty(H)\subseteq L^p_{\sigma,\gamma}(H)$ and integration by parts directly show $\bigl(L^p_{\sigma,\gamma}(H)\bigr)^\perp\supseteq G^{p^\prime}_{\gamma^\prime}(H).$
 
Conversely, assume that $\psi\in\bigl(L^p_{\sigma,\gamma}(H)\bigr)^\perp,$ that is $\psi\in L^{p^\prime}_{\gamma^\prime}(H)^{n+1}$ with
$$\int_H \varphi\cdot\bar{\psi} \dd x\dd y= 0 \quad\forall\varphi\in C_{0,\sigma}^\infty(H).$$
In particular $\psi\in L^1_{\mathrm{loc}}(H)^{n+1}$, so by the theorem of de Rham (see e.g. \cite[Theorem 1.1]{SimaderSohr1992} and \cite[Lemma III.1.1]{Galdi1994I}) there exists $\Phi\in W^{1,1}_{\mathrm{loc}}(H)$ such that $\psi=\nabla\Phi$.  It remains to check that $\Phi$ has a representative which is locally integrable up to the boundary.

Fix a compact $K\subseteq\R^n$ and $t_1>0,$ and choose a bounded open set $U\subseteq \R^n$ such that $K\subseteq U$ compactly. Then for any $t_0\in(0,t_1),$ $\Phi\in W^{1,1}(U\times (t_0,t_1))$; choose a representative $\Phi$ such that $\Phi(\cdot,t)\in L^1(K)$ and $t\mapsto\Phi(x,t)$ is absolutely continuous on $(t_0,t_1)$ for almost every $x\in K$. The fundamental theorem of calculus therefore gives, for almost every $x\in K$,
\begin{equation}\label{eq:FundamentalTheorem}
\Phi(x,t)=\Phi(x,t_1)-\int_t^{t_1}\psi_{n+1}(x,s)\dd s\qquad 0<t<t_1.
\end{equation}
Fubini's theorem and the weighted and classical Hölder inequalities give
\begin{align*}
&\int_0^{t_1}\int_K\int_t^{t_1}\abs{\psi_{n+1}(x,s)}\dd s\dd x\dd t
=\int_0^{t_1} s\norm{\psi_{n+1}(\cdot,s)}_{L^1(K)}\dd s\\
&\qquad\leq \abs{K}^{1/p}\norm{\psi_{n+1}}_{L^{p^\prime}_{\gamma^\prime}(K\times(0,t_1))}\left(\int_0^{t_1}s^{p}s^\gamma\dd s\right)^{1/p}<\infty.
\end{align*}
The last one-dimensional integral is finite because $\gamma>-p-1$. Integrating \eqref{eq:FundamentalTheorem} now yields
$$\norm{\Phi}_{L^1(K\times(0,t_1))}=\int_0^{t_1}\norm{\Phi(\cdot,t)}_{L^1(K)}\dd t\leq t_1\norm{\Phi(\cdot,t_1)}_{L^1(K)}+\int_0^{t_1}\int_K\int_t^{t_1}\abs{\psi_{n+1}(x,s)}\dd s\dd x\dd t<\infty.$$
Since $K$ and $t_1$ were arbitrary, this shows $\Phi\in L^1_{\mathrm{loc}}(\overline H)$. Therefore $\psi\in G^{p^\prime}_{\gamma^\prime}(H)$.

Now the second identity follows from $\bigl(G^p_{\gamma}(H)\bigr)^\perp=\bigl(L^{p^\prime}_{\sigma,\gamma^\prime}(H)\bigr)^{\perp\perp}=L^{p^\prime}_{\sigma,\gamma^\prime}(H).$
\end{proof}

\begin{remark}
The above statement was proved in \cite[Lemma 3.2]{Kakizawa2018} for general $w\in A_p$. Our proof is slightly different also if $-1<\gamma<p-1$.
\end{remark}

Consider the map
$$Q_{p,\gamma}\,\colon\,L^p_\gamma(H)^{n+1}\to\bigl(L^{p^\prime}_{\sigma,\gamma^\prime}(H)\bigr)^\prime,\quad (Q_{p,\gamma}f)(g)=\int_H f\cdot \overline{g}\dd x\dd y\text{ for }g\in L^{p^\prime}_{\sigma,\gamma^\prime}(H).$$
Proposition \ref{prop:AnnihilatorCharacterization} shows that $\ker Q_{p,\gamma}=\bigl(L^{p^\prime}_{\sigma,\gamma^\prime}(H)\bigr)^\perp=G^p_\gamma(H).$ Moreover, the map $Q_{p,\gamma}$ is continuous by Hölder's inequality and surjective by the Hahn-Banach theorem. Indeed, since $L^{p^\prime}_{\sigma,\gamma^\prime}(H)\subseteq L^{p^\prime}_{\gamma^\prime}(H)^{n+1}$ is a closed subspace, one may extend $\ell_0\in\bigl(L^{p^\prime}_{\sigma,\gamma^\prime}(H)\bigr)^\prime$ to $\ell\in \bigl(L^{p^\prime}_{\gamma^\prime}(H)^{n+1}\bigr)^\prime$ with equal norms. Weighted $L^p$-duality gives $f\in L^p_\gamma(H)^{n+1}$ such that
$$\ell(g)=\int_H f\cdot \overline{g}\dd x\dd y,\qquad g\in L^{p^\prime}_{\gamma^\prime}(H)^{n+1}.$$
Restricting this identity to $L^{p^\prime}_{\sigma,\gamma^\prime}(H)$ yields $Q_{p,\gamma}f=\ell_0,$ proving surjectivity. Moreover, the construction gives $\norm{f}_{L^p_\gamma(H)^{n+1}}=\norm{\ell_0}.$ Consequently,
\begin{align}\label{eq:SolenoidalQuotientSpace}
\Phi_{p,\gamma}\,\colon\, L^p_\gamma(H)^{n+1}/G^p_\gamma(H)\to \bigl(L^{p^\prime}_{\sigma,\gamma^\prime}(H)\bigr)^\prime
\end{align}
is an isometric isomorphism and $\bigl(L^{p^\prime}_{\sigma,\gamma^\prime}(H)\bigr)^\prime$ is isometrically normed by the quotient norm
$$\norm{Q_{p,\gamma}f}_{\left(L^{p^\prime}_{\sigma,\gamma^\prime}(H)\right)^\prime}=\norm{[f]}_{L^p_\gamma(H)^{n+1}/G^p_\gamma(H)}=\inf_{g\in G^p_\gamma(H)}\norm{f-g}_{L^p_\gamma(H)^{n+1}}.$$
Note that the roles of $p, \gamma$ and their conjugate exponents $p^\prime, \gamma^\prime$ are interchangeable in the above argument, so an analogous statement holds for $\bigl(L^p_{\sigma,\gamma}(H)\bigr)^\prime.$ Write
\begin{equation}\label{eq:QuotientMap}
q_{p,\gamma}\,\colon\, L^p_\gamma(H)^{n+1}\longrightarrow X^p_{\sigma,\gamma}(H),\qquad q_{p,\gamma}f=[f],\qquad p-1\leq\gamma<2p-1,
\end{equation}
for the quotient map, equipped with the usual quotient norm; recall the definition of $X^p_{\sigma,\gamma}(H)$ from \eqref{eq:ModifiedScale}. Set
$$q=p^\prime,\qquad \beta=-\frac{\gamma}{p-1}.$$
Then $-q-1<\beta\leq -1.$ The isomorphism $\Phi_{p,\gamma}$ identifies the space $X^p_{\sigma,\gamma}(H)=L^p_{\gamma}(H)^{n+1}/G^p_\gamma(H)$ with the dual of $L^{p^\prime}_{\sigma,\gamma^\prime}(H)=X^q_{\sigma,\beta}(H)$.

\begin{lemma}\label{lem:QuotientNormCharacterization}
Let $1<p<\infty,$ let $\lambda\in\C\setminus(-\infty,0],$ and let $p-1\leq\gamma<2p-1.$ If $f\in C_{0,\sigma}^\infty(H)$ satisfies $\mathcal{U}_\lambda f\in L^p_\gamma(H)^{n+1},$ then
\begin{equation*}
\norm{q_{p,\gamma}\mathcal{U}_\lambda f}_{X^p_{\sigma,\gamma}(H)}=\!\!\!\sup_{\substack{g\in C_{0,\sigma}^\infty(H)\\\norm{g}_{L^q_\beta(H)^{n+1}}\leq 1}}\abs{\int_H \mathcal{U}_{\lambda}f \cdot \overline{g}\dd x\dd y}.
\end{equation*}
Here we used the notation $q=p^\prime$ and $\beta=-\frac{\gamma}{p-1}.$
\end{lemma}
\begin{proof} 
This follows from the fact that $\Phi_{p,\gamma}$ is an isometric isomorphism together with density of $C_{0,\sigma}^\infty(H)$ in $L^q_{\sigma,\beta}(H).$
\end{proof}

After these preparations, we are ready to prove the remaining part of Theorem \ref{th:ResolventEstimates}.
\begin{proof}[Proof of Theorem \ref{th:ResolventEstimates}\ref{eq:UpperCoreResolventEstimate}.]
Let $1<p<\infty,\,\theta\in(0,\pi)$ and let $p-1\leq\gamma<2p-1.$ Set
$$q=p^\prime,\qquad \beta=-\frac{\gamma}{p-1}.$$
Then $-q-1<\beta\leq -1.$ Fix $\lambda\in\Sigma_\theta$ and $f,g\in C_{0,\sigma}^\infty(H).$ We define
$$u^{(1)}\coloneqq\mathcal{U}_\lambda f,\qquad u^{(2)}\coloneqq\mathcal{U}_{\overline{\lambda}} g.$$
By Lemma \ref{lem:CoreProperty} and Theorem \ref{th:ResolventEstimates}\ref{eq:LowerCoreResolventEstimate} it holds
\begin{equation}\label{eq:LowerSolutionProperties}
u^{(2)}\in L^q_{\sigma,\beta}(H),\qquad \norm{u^{(2)}}_{L^q_\beta(H)^{n+1}}\leq \frac{C}{\abs{\lambda}}\norm{g}_{L^q_\beta(H)^{n+1}}.
\end{equation}
By definition, both $u^{(1)}$ and $u^{(2)}$ are strong unweighted $L^2$-solutions, that is, there are $\pi_f,\,\pi_g\in L^1_{\mathrm{loc}}(\overline{H})$ with $\nabla\pi_f,\,\nabla\pi_g\in L^2(H)^{n+1}$ such that
\begin{align}
\lambda u^{(1)}-\Delta u^{(1)}+\nabla\pi_f&=f,\label{eq:TestingEquationI}\\
\overline{\lambda} u^{(2)}-\Delta u^{(2)}+\nabla\pi_g&=g,\label{eq:TestingEquationII}
\end{align}
and it holds $u^{(i)}\in L^2_\sigma(H)\cap W^{2,2}(H)^{n+1}\cap W^{1,2}_0(H)^{n+1},\,i\in\{1,2\}.$ Thus testing \eqref{eq:TestingEquationI} against $u^{(2)}$ and \eqref{eq:TestingEquationII} against $u^{(1)}$ yields
\begin{equation}\label{eq:TestingIdentity}
\int_H \mathcal{U}_{\lambda} f \cdot \overline{g}\dd x\dd y=\int_H f\cdot \overline{\mathcal{U}_{\overline{\lambda}} g}\dd x\dd y.
\end{equation}
Let $\pi\in G^p_\gamma(H).$ By Proposition \ref{prop:AnnihilatorCharacterization}, $\pi\in(L^q_{\sigma,\beta}(H))^\perp.$ Combined with equation \eqref{eq:TestingIdentity}, the properties of $u^{(2)}=\mathcal{U}_{\overline{\lambda}} g$ from \eqref{eq:LowerSolutionProperties} and Hölder's inequality, this yields
\begin{align*}
\abs{\int_H \mathcal{U}_{\lambda} f \cdot \overline{g}\dd x\dd y}
&=\abs{\int_H f\cdot \overline{\mathcal{U}_{\overline{\lambda}} g}\dd x\dd y}\\
&=\abs{\int_H (f-\pi)\cdot \overline{\mathcal{U}_{\overline{\lambda}} g}\dd x\dd y}\\
&\leq \norm{f-\pi}_{L^p_\gamma(H)^{n+1}}\norm{\mathcal{U}_{\overline{\lambda}} g}_{L^q_\beta(H)^{n+1}}\\
&\leq\frac{C_{n,p,\gamma,\theta}}{\abs{\lambda}}\norm{f-\pi}_{L^p_\gamma(H)^{n+1}}\norm{g}_{L^q_\beta(H)^{n+1}}.
\end{align*}
Taking the infimum over $\pi\in G^p_\gamma(H)$ gives
\begin{equation}\label{eq:GeneralResolventEstimate}
\abs{\int_H \mathcal{U}_{\lambda}f \cdot \overline{g}\dd x\dd y}\leq\frac{C_{n,p,\gamma,\theta}}{\abs{\lambda}}\norm{q_{p,\gamma}f}_{X^p_{\sigma,\gamma}(H)}\norm{g}_{L^q_\beta(H)^{n+1}}.
\end{equation}

If it holds $\mathcal{U}_{\lambda}f\in L^p_\gamma(H)^{n+1},$ we may conclude from \eqref{eq:GeneralResolventEstimate} together with Lemma \ref{lem:QuotientNormCharacterization}
\begin{align*}
\norm{q_{p,\gamma}\mathcal{U}_\lambda f}_{X^p_{\sigma,\gamma}(H)}=\sup_{\substack{g\in C_{0,\sigma}^\infty(H)\\\norm{g}\leq 1}}\abs{\int_H \mathcal{U}_{\lambda}f \cdot \overline{g}\dd x\dd y}\leq\frac{C_{n,p,\gamma,\theta}}{\abs{\lambda}}\norm{q_{p,\gamma}f}_{X^p_{\sigma,\gamma}(H)}
\end{align*}
as desired.
\end{proof}

\section{\texorpdfstring{The Stokes Operator on $X^p_{\sigma,\gamma}(H)$}{The Stokes Operator}}
\label{sec:StokesOperator}
Let $1<p<\infty.$ In this section we construct realizations of the Stokes operator on $X^p_{\sigma,\gamma}(H),$
$$-p-1<\gamma<2p-1.$$

\subsection{Lower Range}\label{sec:StokesLow}
We begin with the case $\gamma \in (-p-1,p-1)$ where $X^p_{\sigma,\gamma}(H)=L^p_{\sigma,\gamma}(H).$ Fix $1<p<\infty$ and $\gamma\in(-p-1,p-1).$ We recall the unweighted Stokes theory on $H$. For $1<p<\infty$, let $\mathbb{P}_p$ denote the Helmholtz projection onto $L^p_\sigma(H)$, and define
$$\mathcal{A}_pu=\mathbb{P}_p\Delta u,\qquad D(\mathcal{A}_p)=L^p_\sigma(H)\cap W^{2,p}(H)^{n+1}\cap W^{1,p}_0(H)^{n+1}.$$
This operator is closed by \cite[Corollary 1.6(iii)]{FarwigSohr1994}, and densely defined since $C^\infty_{0,\sigma}(H)\subseteq D(\mathcal{A}_p)$. Moreover, by \cite[Theorem 1.3]{FarwigSohr1994}, every $\lambda\in\C\setminus(-\infty,0]$ belongs to $\rho(\mathcal{A}_p)$, and, for each $\theta\in(0,\pi)$,
$$\norm{(\lambda-\mathcal{A}_p)^{-1}}_{\mathcal{L}(L^p_\sigma(H))}\leq \frac{C_{n,p,\theta}}{\abs{\lambda}},\qquad \lambda\in\Sigma_\theta.$$
The same theorem implies consistency across exponents:
$$(\lambda-\mathcal{A}_p)^{-1}f=(\lambda-\mathcal{A}_q)^{-1}f,\qquad f\in L^p_\sigma(H)\cap L^q_\sigma(H),\quad 1<p,q<\infty.$$
In particular, the definition of $\mathcal{U}_\lambda$ gives
$$(\lambda-\mathcal{A}_p)^{-1}f=\mathcal{U}_\lambda f,\qquad f\in C^\infty_{0,\sigma}(H).$$
Finally, let us remark that $\mathcal{A}_p\varphi=\Delta\varphi$ for every $\varphi\in C^\infty_{0,\sigma}(H).$

\vspace{8pt}
We are now going to define an operator $\mathcal{A}_{p,\gamma}$ on $L^p_{\sigma,\gamma}(H).$ Let
\begin{equation*}
\begin{aligned}
D(\mathcal{A}_{p,\gamma}^0) &= \{u\in D(\mathcal{A}_p)\cap L^p_{\sigma,\gamma}(H)\,\colon\, \mathcal{A}_p u\in L^p_{\sigma,\gamma}(H)\},\\
\mathcal{A}_{p,\gamma}^0 u &= \mathcal{A}_p u.
\end{aligned}
\end{equation*}
Let us write $X=L^p_{\sigma,\gamma}(H).$ Fix $\theta\in(0,\pi).$ For every $\lambda\in\Sigma_\theta,$ Theorem \ref{th:ResolventEstimates}\ref{eq:LowerCoreResolventEstimate} gives a unique bounded extension
\begin{equation}\label{eq:ResolventExtension}
\mathrm{R}_\lambda\in\mathcal{L}(X)
\end{equation}
of $\mathcal{U}_\lambda$ and it holds
$$\norm{\mathrm{R}_\lambda}_{\mathcal{L}(X)}\leq \frac{C_{n,p,\gamma,\theta}}{\abs{\lambda}}.$$
It is clear that the operator $\mathcal{A}_{p,\gamma}^0$ is densely defined in $X$ since $C_{0,\sigma}^\infty(H)\subseteq D(\mathcal{A}_p)\cap X\subseteq X$ and $\mathcal{A}_p\bigl(C_{0,\sigma}^\infty(H)\bigr)\subseteq C_{0,\sigma}^\infty(H).$ We show that this operator is also closable in $X.$

For $\lambda>0$ and $\varphi\in C_{0,\sigma}^\infty(H)$ it holds
$$\lambda\mathrm{R}_\lambda\varphi-\varphi=(\lambda-\mathcal{A}_p)\mathrm{R}_\lambda\varphi +\mathrm{R}_\lambda\mathcal{A}_p\varphi-\varphi=\mathrm{R}_\lambda\mathcal{A}_p\varphi.$$
Since $\mathcal{A}_p\varphi=\Delta\varphi\in C_{0,\sigma}^\infty(H),$ Theorem \ref{th:ResolventEstimates}\ref{eq:LowerCoreResolventEstimate} implies
$$\norm{\lambda\mathrm{R}_\lambda\varphi-\varphi}_X\leq \frac{C_{n,p,\gamma,\theta}}{\abs{\lambda}}\norm{\mathcal{A}_p\varphi}_X.$$
Thus, density of $C_{0,\sigma}^\infty(H)$ in $X$ and $\sup_{\lambda>0}\norm{\lambda\mathrm{R}_\lambda}<\infty$ give
\begin{equation}\label{eq:ResolventConvergence}
\lambda\mathrm{R}_\lambda f\overset{\lambda\to\infty}{\longrightarrow} f,\qquad f\in X.
\end{equation}
Now suppose $(u_k)_{k\in\N}\subseteq D(\mathcal{A}_{p,\gamma}^0)$ is such that
$$u_k\to 0,\qquad \mathcal{A}_{p,\gamma}^0 u_k\to f\qquad\text{in }X.$$
Then for each fixed $\lambda>0,$ by Lemma \ref{lem:CoreProperty}
$$u_k=\mathrm{R}_\lambda(\lambda u_k-\mathcal{A}_{p,\gamma}^0 u_k)\overset{k\to\infty}{\longrightarrow} -\mathrm{R}_\lambda f$$
Thus $\mathrm{R}_\lambda f=0$ for every $\lambda>0,$ and \eqref{eq:ResolventConvergence} implies $f=0.$ This proves closability.

Thus, it makes sense to define
\begin{equation}\label{eq:StokesDefinitionLow}
\mathcal{A}_{p,\gamma}=\overline{\mathcal{A}_{p,\gamma}^0}^{X\times X}.
\end{equation}
This operator is closed and densely defined in $X$. We show that the resolvent extension from \eqref{eq:ResolventExtension} is indeed the resolvent of the operator $\mathcal{A}_{p,\gamma}.$ Fix an arbitrary $\lambda\in\C\setminus(-\infty,0].$ For $u\in D(\mathcal{A}_{p,\gamma}^0)$ one has $\lambda u-\mathcal{A}_p u\in L^p_\sigma(H)\cap X,$ so $\mathrm{R}_\lambda(\lambda-\mathcal{A}_{p,\gamma}^0)u=u.$ Passing to limits in the graph closure yields
$$\mathrm{R}_\lambda(\lambda-\mathcal{A}_{p,\gamma})u=u,\qquad u\in D(\mathcal{A}_{p,\gamma}).$$
Conversely, let $f\in X,$ and choose $f_k\in C_{0,\sigma}^\infty(H)$ with $f_k\to f$ in $X.$ Set
$$u_k=(\lambda-\mathcal{A}_p)^{-1}f_k=\mathrm{R}_\lambda f_k.$$
Then $u_k\in D(\mathcal{A}_p)\cap X$ and $\mathcal{A}_p u_k=\lambda u_k-f_k\in X,$ so $u_k\in D(\mathcal{A}_{p,\gamma}^0).$ Moreover,
$$u_k\to \mathrm{R}_\lambda f,\qquad \mathcal{A}_{p,\gamma}^0 u_k\to \lambda\mathrm{R}_\lambda f-f\qquad\text{in }X.$$
Therefore
$$\mathrm{R}_\lambda f\in D(\mathcal{A}_{p,\gamma})\qquad\text{and}\qquad (\lambda-\mathcal{A}_{p,\gamma})\mathrm{R}_\lambda f=f.$$
This proves $\mathrm{R}_\lambda=(\lambda-\mathcal{A}_{p,\gamma})^{-1}.$

Consistency determines this bounded inverse uniquely on the dense subspace $L^p_\sigma(H)\cap X,$ and therefore determines $\mathcal{A}_{p,\gamma}$ uniquely.

\begin{proof}[Proof of Theorem \ref{th:StokesOperator} for $-p-1<\gamma<p-1$.]
We have seen that the constructed realizations $\mathcal{A}_{p,\gamma}$ on $X^p_{\sigma,\gamma}(H),\,-p-1<\gamma<p-1,$ are closed and densely defined with $\C\setminus(-\infty,0]\subseteq\rho(\mathcal{A}_{p,\gamma}).$ It follows immediately from Theorem \ref{th:ResolventEstimates}\ref{eq:LowerCoreResolventEstimate} that for the resolvent operators
$$\mathrm{R}_{p,\gamma}(\lambda)=(\lambda-\mathcal{A}_{p,\gamma})^{-1}\in\mathcal{L}(X^p_{\sigma,\gamma}(H)),\qquad\lambda\in\Sigma_\theta,$$
there exists a constant $C_{n,p,\gamma,\theta}>0$ such that
\begin{equation*}
\norm{\mathrm{R}_{p,\gamma}(\lambda)F}_{X^p_{\sigma,\gamma}(H)} \leq \frac{C_{n,p,\gamma,\theta}}{\abs{\lambda}} \norm{F}_{X^p_{\sigma,\gamma}(H)},\qquad \lambda\in\Sigma_\theta,\quad F\in X^p_{\sigma,\gamma}(H).
\end{equation*}
This shows that the operator $-\mathcal{A}_{p,\gamma}$ is sectorial of angle zero on $X^p_{\sigma,\gamma}(H).$ Thus, $\mathcal{A}_{p,\gamma}$ generates a bounded analytic $C_0$-semigroup $S_{p,\gamma}$ of angle $\pi/2,$ bounded on every smaller sector. Since $\mathrm{R}_{p,\gamma}(\lambda)$ was constructed as the unique extension $\mathrm{R}(\lambda)=\overline{T}_\lambda$ of $\mathcal{U}_\lambda$ the identity
$$\mathrm{R}_{p,\gamma}(\lambda)f=\mathcal{U}_\lambda f,\qquad f\in C_{0,\sigma}^\infty(H),$$
follows.
\end{proof}
The constructed realizations form a consistent family of Stokes realizations and are also consistent with the realizations $(\mathcal{A}_{p,w}, D(\mathcal{A}_{p,w}))$ on $L^p$-spaces with Muckenhoupt weights constructed in \cite[Theorem 1.1]{Froehlich2003}. Define
$$\mathscr{W}_p=\{w|_H\,\colon\,w\in A_p(\R^{n+1})\}\cup\{w_\gamma\,\colon\,-p-1<\gamma<p-1\},\qquad w_\gamma(x,y)=y^\gamma.$$
For $w\in\mathscr W_p$, let
$$L^p_{\sigma,w}(H)=\overline{C^\infty_{0,\sigma}(H)}^{L^p_w(H)^{n+1}},$$
which is consistent with the definition for $L^p_{\sigma,\gamma}(H).$ The Stokes realization on $L^p_{\sigma,w}(H),\,w\in A_p(\R^{n+1}),$ given in \cite{Froehlich2003}, is
\begin{align*}
D(\mathcal{A}_{p,w}) &\coloneqq W_{w}^{2,p}(H)^{n+1} \cap L_{\sigma,w}^p(H) \cap \{u \in W_{w}^{1,p}(H)^{n+1} \colon u|_{\partial H} = 0 \},\\
\mathcal{A}_{p,w} &\coloneqq \mathbb{P}_{p,w} \Delta \quad \text{on } D(\mathcal{A}_{p,w}),
\end{align*}
where $\mathbb{P}_{p,w} \,\colon\, L^p_w(H)^{n+1} \to L^p_{\sigma,w}(H)$ denotes the Helmholtz projection on $L^p_w(H)^{n+1}$ from \cite[Theorem 5]{Froehlich2000}. Note that we changed the sign in the operator definition to match the sign convention used throughout the present article.

On $C_{0,\sigma}^\infty(H),$ we have an explicit representation of both $(\lambda-\mathcal{A}_{p,w})^{-1}$ and $(\lambda-\mathcal{A}_{q,\nu})^{-1}$ via their common convolution kernel; this means that the transpose kernel forms a bounded linear functional on $C_{0,\sigma}^\infty(H).$ Our aim is to exploit this fact by showing that the transpose kernel extends to a bounded linear functional on $L^p_{\sigma,w}(H) + L^q_{\sigma,\nu}(H).$ Due to $L^p_w$-duality and density of $C_0^\infty(H)^{n+1},$ we obtain this by verifying that it is an element of $L^{p^\prime}_{w^\prime}(H)^{n+1}\cap L^{q^\prime}_{\nu^\prime}(H)^{n+1}.$

Let us first set up some notation for the transpose kernel. For $f\in C_0^\infty(H)^{n+1},$ define the operators
\begin{equation*}
(\mathcal{C}_\lambda f)(x,y)=
\begin{pmatrix}
\int_H r_v(x-\tilde{x},y,s,\lambda)\,f_v(\tilde{x},s)\dd \tilde{x}\dd s\\
\int_H r_w(x-\tilde{x},y,s,\lambda)\cdot f_v(\tilde{x},s)\dd \tilde{x}\dd s
\end{pmatrix}
\qquad x\in\R^n,\,y>0,
\end{equation*}
and $\mathcal{K}_\lambda f=\bigl(\lambda-\Delta^D\bigr)^{-1} f + \mathcal{C}_\lambda f.$
When $f\in C_{0,\sigma}^\infty(H),$ \eqref{eq:StokesSolutionFormula} states that
$$\mathcal{K}_\lambda f=\mathcal{U}_\lambda f,\qquad \lambda\in\C\setminus(-\infty,0].$$

Now fix $\varphi=(\varphi_v,\varphi_w)\in C^\infty_0(H)^{n+1}$ and $\lambda\in\C\setminus(-\infty,0].$ Let us introduce the notation
$$B_{\lambda,\varphi}(x,\tilde{y})=\int_0^\infty\int_{\R^n}\bigl(r_v(\tilde{x}-x,s,\tilde{y},\lambda)^{\mathrm T}\varphi_v(\tilde{x},s)+r_w(\tilde{x}-x,s,\tilde{y},\lambda)\varphi_w(\tilde{x},s)\bigr)\dd \tilde{x}\dd s$$
for the transpose kernel of the remainder and
$$\Psi_{\lambda,\varphi}(x,\tilde{y})=\bigl(\lambda-\Delta^D\bigr)^{-1}\varphi(x,\tilde{y})+\bigl(B_{\lambda,\varphi}(x,\tilde{y}),\,0\bigr)^\mathrm{T}$$
for the full transpose kernel. Indeed, symmetry of the Green kernel of the Dirichlet Laplacian gives
$$\int_H \bigl(\lambda-\Delta^D\bigr)^{-1}f\cdot \varphi\dd x\dd y=\int_H f\cdot \bigl(\lambda-\Delta^D\bigr)^{-1} \varphi\dd x\dd y,\qquad f\in C_0^\infty(H)^{n+1},$$
while for the remainder part we have
$$\int_H\mathcal{C}_\lambda f\cdot \varphi\dd x\dd y=\int_H f_v\cdot B_{\lambda,\varphi}\dd x\dd y,\qquad f\in C_0^\infty(H)^{n+1}.$$
Therefore
$$\int_H \mathcal{K}_\lambda f\cdot \varphi\dd x\dd y=\int_H f\cdot \Psi_{\lambda,\varphi}\dd x\dd y,\qquad f\in C_0^\infty(H)^{n+1}$$
as intended.
\begin{proposition}[Consistency]\label{prop:LowerRangeConsistency}
For $1<p,q<\infty$ let $w\in\mathscr W_p$ and $\nu\in\mathscr W_q.$ Then for $f\in L^p_{\sigma,w}(H)\cap L^q_{\sigma,\nu}(H),$ one has
\begin{equation*}
\begin{aligned}
(\lambda-\mathcal{A}_{p,w})^{-1}f&=(\lambda-\mathcal{A}_{q,\nu})^{-1}f,&&\lambda\in\C\setminus(-\infty,0],\\
S_{p,w}(t)f&=S_{q,\nu}(t)f,&&t\geq 0.
\end{aligned}
\end{equation*}
Moreover,
$$\mathcal{A}_{p,w}u=\mathcal A_{q,\nu}u\quad\text{for }u\in D(\mathcal{A}_{p,w})\cap D(\mathcal{A}_{q,\nu}).$$
\end{proposition}
\begin{proof}
Write
$$\mathrm R_{p,w}(\lambda)=(\lambda-\mathcal A_{p,w})^{-1},\qquad 1<p<\infty,\quad w\in\mathscr W_p.$$
Fix $\varphi\in C_0^\infty(H)^{n+1}$ and $\lambda\in\C\setminus(-\infty,0].$

We begin with the power-weight case. Let $1<p<\infty$ and let $-p-1<\gamma<p-1.$ It is easy to check that the arguments in Section \ref{sec:ResolventLow} do not use the fact that $f$ is solenoidal; thus one obtains the analogous estimate for fixed $\theta\in(0,\pi)$
$$\norm{\mathcal{C}_\lambda f}_{L^p_\gamma(H)^{n+1}}\leq \frac{C_{n,p,\gamma,\theta}}{\abs{\lambda}}\norm{f}_{L^p_\gamma(H)^{n+1}},\qquad f\in C_0^\infty(H)^{n+1},\,\lambda\in\Sigma_\theta$$
by repeating those arguments. Combined with density of $C_0^\infty(H)^{n+1}$ in $L^p_\gamma(H)^{n+1}$ and the Dirichlet resolvent estimate from \cite[Proposition 5.1]{LindemulderVeraar2020}, this shows that $\mathcal{K}_\lambda$ is a bounded operator on the \emph{full} weighted space which satisfies the resolvent estimate
\begin{equation}\label{eq:GeneralizedResolventEstimate}
\norm{\mathcal{K}_\lambda f}_{L^p_\gamma(H)^{n+1}}\leq \frac{C_{n,p,\gamma,\theta}}{\abs{\lambda}}\norm{f}_{L^p_\gamma(H)^{n+1}},\qquad f\in L^p_\gamma(H)^{n+1},\,\lambda\in\Sigma_\theta.
\end{equation}
Choosing $\theta\in(0,\pi)$ such that $\theta>\abs{\arg\lambda},$ using this resolvent estimate and the weighted Hölder inequality we see
$$\abs{\int_H f\cdot\Psi_{\lambda,\varphi}\dd x\dd y}\leq\frac{C}{\abs{\lambda}}\norm{f}_{L^p_\gamma(H)^{n+1}}\norm{\varphi}_{L^{p^\prime}_{\gamma^\prime}(H)^{n+1}},\qquad f\in C_0^\infty(H)^{n+1}.$$
Density of $C_0^\infty(H)^{n+1}$ in $L^p_\gamma(H)^{n+1}$ and $L^p_w$-duality therefore yield
\begin{equation}\label{eq:TransposePowerDuality}
\norm{\Psi_{\lambda,\varphi}}_{L^{p^\prime}_{\gamma^\prime}(H)^{n+1}}\leq\frac{C}{\abs{\lambda}}\norm{\varphi}_{L^{p^\prime}_{\gamma^\prime}(H)^{n+1}}.
\end{equation}

\vspace{8pt}
Now we consider the case of Muckenhoupt weights. Let $1<p<\infty$ and let $w\in A_p(\R^{n+1}).$ Let $K=\support\varphi\Subset H$. There exist $a,b,R>0$ such that
$$K\subseteq\{(x,y):\abs{x}\leq R,\ a\leq y\leq b\}.$$
Then using the scaling \eqref{eq:rv_scaling}, \eqref{eq:rw_scaling} and Proposition \ref{prop:pointwiseEstimate0} we see
$$\abs{r_v(x,y,\tilde{y},\lambda)} + \abs{r_w(x,y,\tilde{y},\lambda)}\leq C_{\lambda,K}e^{-c \tilde{y}}\int_0^\infty\tau^n e^{-c\tau(\abs{x}+y+\tilde{y})}\dd\tau,\qquad x\in\R^n,\,a\leq y\leq b,\,\tilde{y}>0.$$
Repeated integration by parts gives
$$\int_0^\infty \tau^n e^{-B\tau}\dd\tau=\frac{n!}{B^{n+1}},\qquad B>0,$$
so we obtain
$$\abs{r_v(x,y,\tilde{y},\lambda)} + \abs{r_w(x,y,\tilde{y},\lambda)}\leq C_{\lambda,K}\frac{e^{-c \tilde{y}}}{(\abs{x}+y+\tilde{y})^{n+1}},\qquad x\in\R^n,\,a\leq y\leq b,\,\tilde{y}>0.$$
Moreover,
$$1+\abs{x}+\tilde{y}\leq 1+R+\abs{x-\tilde{x}}+\tilde{y}\leq C_K(\abs{x-\tilde{x}}+y+\tilde{y}),$$
because $y\geq a$. Thus
\begin{equation}\label{eq:BlambdavarphiEstimate}
\abs{B_{\lambda,\varphi}(x,\tilde{y})}\leq C_{\lambda,K}\norm{\varphi}_{L^1(H)}\frac{e^{-c \tilde{y}}}{(1+\abs{x}+\tilde{y})^{n+1}}\qquad x\in\R^n,\,\tilde{y}>0.
\end{equation}
For the Dirichlet contribution, let $\varphi^{\mathrm{odd}}$ be the odd extension of $\varphi$ across the boundary. Since $\varphi$ is supported away from the boundary,
$$\varphi^{\mathrm{odd}}\in C^\infty_0(\R^{n+1})^{n+1},$$
and we may write
$$\bigl(\lambda-\Delta^D\bigr)^{-1}\varphi=\mathcal{F}^{-1}\Bigl((\lambda+\abs{\xi}^2)^{-1}\widehat{\varphi^{\mathrm{odd}}}(\xi)\Bigr)\Big|_H,$$
where the Fourier transform and its inverse are taken in $\R^{n+1}.$
The multiplier is smooth with polynomially bounded derivatives. Its product with $\widehat{\varphi^{\mathrm{odd}}}$ is therefore a Schwartz function, and so is its inverse Fourier transform. Combining this with \eqref{eq:BlambdavarphiEstimate} gives
$$\abs{\Psi_{\lambda,\varphi}(x,\tilde{y})}\leq C_{\lambda,\varphi}(1+\abs{(x,\tilde{y})})^{-(n+1)},\qquad x\in\R^n,\,\tilde{y}>0.$$
Let $B$ be the unit ball in $\R^{n+1}$. The ball centered on $(x,\tilde{y})$ with radius $1+\abs{(x,\tilde{y})}$ contains $B,$ thus choosing it in the supremum of the maximal operator shows
$$(1+\abs{(x,\tilde{y})})^{-(n+1)}=\frac{|B|}{|B|(1+\abs{(x,\tilde{y})})^{n+1}}\leq C_n (\mathcal{M}\mathds{1}_B)(x,\tilde{y}).$$
The boundedness of the maximal operator from Theorem \ref{th:maximal} consequently yields
\begin{equation}\label{eq:TransposeMuckenhouptDuality}
\norm{\Psi_{\lambda,\varphi}}_{L^{p^\prime}_{w^\prime}(H)^{n+1}} \leq C_{\lambda,\varphi}\norm{\mathcal{M}\mathds{1}_B}_{L^{p^\prime}_{w^\prime}(\R^{n+1})} \leq C_{\lambda,\varphi,p,w}\Bigl(\int_B w^\prime\dd x\dd y\Bigr)^{1/p^\prime}<\infty.
\end{equation}

\vspace{8pt}
We can now conclude the desired consistency. Let $1<p<\infty$, $w\in\mathscr W_p$, and $f\in L^p_{\sigma,w}(H)$. Choose $(g_j)_{j\in\N}\subseteq C^\infty_{0,\sigma}(H)$ such that
$$g_j\longrightarrow f \qquad\text{in }L^p_w(H)^{n+1}.$$
By Theorem~\ref{th:StokesOperator} in the power-weight case, and \cite[Theorem 1.1(ii)]{Froehlich2003} in the Muckenhoupt case,
$$\mathrm R_{p,w}(\lambda)g_j =\mathcal U_\lambda g_j =\mathcal K_\lambda g_j.$$
Boundedness of $\mathrm R_{p,w}(\lambda)$ and \eqref{eq:TransposePowerDuality}, \eqref{eq:TransposeMuckenhouptDuality} therefore give
\begin{equation*}
\begin{aligned}
\int_H\mathrm R_{p,w}(\lambda)f\cdot\varphi\dd x\dd y &= \lim_{j\to\infty} \int_H\mathcal K_\lambda g_j\cdot\varphi\dd x\dd y\\
&=\lim_{j\to\infty} \int_H g_j\cdot\Psi_{\lambda,\varphi}\dd x\dd y\\
&=\int_H f\cdot\Psi_{\lambda,\varphi}\dd x\dd y.
\end{aligned}
\end{equation*}
Consequently, for $1<p,q<\infty$, $w\in\mathscr W_p$, $\nu\in\mathscr W_q$, and $f\in L^p_{\sigma,w}(H)\cap L^q_{\sigma,\nu}(H)$,
\begin{equation*}
\begin{aligned}
\int_H\mathrm R_{p,w}(\lambda)f\cdot\varphi\dd x\dd y
&=\int_H f\cdot\Psi_{\lambda,\varphi}\dd x\dd y\\
&=\int_H\mathrm R_{q,\nu}(\lambda)f\cdot\varphi\dd x\dd y.
\end{aligned}
\end{equation*}
Since $\varphi\in C_0^\infty(H)^{n+1}$ was arbitrary,
$$\mathrm R_{p,w}(\lambda)f = \mathrm R_{q,\nu}(\lambda)f \qquad\text{in }\mathcal D^\prime(H)^{n+1}.$$
Each side belongs to its respective solenoidal space. Their common value therefore belongs to $L^p_{\sigma,w}(H)\cap L^q_{\sigma,\nu}(H)$. Let $t>0$ and $f\in L^p_{\sigma,w}(H)\cap L^q_{\sigma,\nu}(H)$. Now use the standard semigroup representation
$$S_{p,w}(t)f = \lim_{k\to\infty} \left(\frac{k}{t}\mathrm R_{p,w}(k/t)\right)^k f \qquad\text{in }L^p_{\sigma,w}(H),$$
and analogously for $S_{q,\nu}$. Resolvent consistency and invariance of the intersection give
$$\left(\frac{k}{t}\mathrm R_{p,w}(k/t)\right)^k f = \left(\frac{k}{t}\mathrm R_{q,\nu}(k/t)\right)^k f, \qquad k\in\N.$$
Since both spaces embed continuously into $\mathcal D^\prime(H)^{n+1}$, the respective limits coincide. Thus
$$S_{p,w}(t)f=S_{q,\nu}(t)f,\qquad t>0.$$
For $t=0$, both sides equal $f$.

Finally, let $u\in D(\mathcal A_{p,w})\cap D(\mathcal A_{q,\nu})$. Semigroup consistency yields
$$\frac{S_{p,w}(t)u-u}{t} = \frac{S_{q,\nu}(t)u-u}{t}, \qquad t>0.$$
By the definition of the generator,
\begin{equation*}
\begin{aligned}
\frac{S_{p,w}(t)u-u}{t}
&\longrightarrow\mathcal A_{p,w}u &&\text{in }L^p_{\sigma,w}(H),\\
\frac{S_{q,\nu}(t)u-u}{t} &\longrightarrow\mathcal A_{q,\nu}u &&\text{in }L^q_{\sigma,\nu}(H)
\end{aligned}
\qquad \text{as }t\to 0.
\end{equation*}
Uniqueness of the limit in $\mathcal D^\prime(H)^{n+1}$ therefore gives
$$\mathcal A_{p,w}u=\mathcal A_{q,\nu}u.$$
\end{proof}

To conclude this subsection, we prove that the intersection $L^p_{\sigma,\gamma}(H)\cap G^p_\gamma(H)$ is trivial on the lower range $-p-1<\gamma\leq-1.$ This is in contrast to the result on the upper range $p-1<\gamma<2p-1,$ see Proposition \ref{prop:CounterexampleHelmholtzDecomposition}. The result will be needed in the next subsection.

First, we would like to derive a local estimate of the horizontal norm of $u$ in terms of $y$ and the full norm of $u$. If $\Delta u=0\text{ in }\mathcal{D}^\prime(H)$, then by Weyl's lemma \cite[Theorem 4.4.1]{Hormander1983}, $u$ has a smooth harmonic representative. In the following all pointwise assertions are understood as assertions for this smooth representative.
\begin{lemma}\label{lem:HorizontalEstimate}
Let $1<p<\infty$ and let $-p-1<\gamma\leq-1.$ Let $u\in L^p_\gamma(H)$ satisfy
$$\Delta u=0\qquad\text{in }\mathcal{D}^\prime(H).$$
Then one has the estimate
$$\norm{u(\cdot,y)}_{L^p(\R^{n})}\leq C y^\beta\norm{u}_{L^p_\gamma(H)},\quad y>0,\qquad\beta=-\frac{\gamma+1}{p}.$$
\end{lemma}
\begin{proof} 
The mean-value theorem \cite[Theorem 2.1]{GilbargTrudinger2001} and Jensen's inequality, applied on $B_{y/2}(x,y)\subset H\subset\R^{n+1}$, give
$$\abs{u(x,y)}^p\leq C y^{-(n+1)}\int_{y/2}^{3y/2}\int_{\abs{\tau-x}<y/2}\abs{u(\tau,s)}^p\dd \tau\dd s.$$
Integrating with respect to $x$ and applying Fubini's theorem yields
\begin{align*}
\norm{u(\cdot,y)}_{L^p(\R^n)}^p
&\leq C y^{-(n+1)}\int_{y/2}^{3y/2}\int_{\R^n}\int_{\abs{\tau-x}<y/2}\abs{u(\tau,s)}^p\dd \tau\dd x\dd s\\
&=C y^{-(n+1)}\int_{y/2}^{3y/2}\int_{\R^n}\abs{u(\tau,s)}^p\int_{\R^n}\mathbf{1}_{\{\abs{\tau-x}<y/2\}}\dd x\dd \tau\dd s\\
&= C y^{-1}\int_{y/2}^{3y/2}\norm{u(\cdot,s)}_{L^p(\R^n)}^p\dd s,
\end{align*}
because the inner $\mathrm{d}x$-integral is $2^{-n}\abs{B_n}y^n$. Since $s/y\in[1/2,3/2]$ on the interval of integration,
$$\norm{u(\cdot,y)}_{L^p(\R^n)}^p\leq C_\gamma y^{-1-\gamma}\int_{y/2}^{3y/2}\norm{u(\cdot,s)}_{L^p(\R^n)}^p s^\gamma\dd s\leq C_\gamma y^{-1-\gamma}\norm{u}_{L^p_\gamma(H)}^p.$$
Taking the $p$-th root proves the estimate.
\end{proof}

\begin{lemma}\label{lem:HorizontalPoissonShift}
Let $1<p<\infty.$ Let $u\,\colon\,H\to\C$ satisfy
$$\Delta u=0\qquad\text{in }\mathcal{D}^\prime(H).$$
Suppose moreover that it holds $u(\cdot,y)\in L^p(\R^{n})$ for every $y>0$ and there exist $C>0$ and $\beta\geq 0$ such that
\begin{align}\label{eq:Assumption1}
\norm{u(\cdot,y)}_{L^p(\R^{n})}\leq C(1+y)^\beta.
\end{align}
Then it holds
$$u(\cdot,a+t)=P_tu(\cdot,a),\qquad a,t>0,$$
where $P_t=e^{-tA}$ is the Poisson semigroup from Section \ref{sec:Counterexamples}.
\end{lemma}
\begin{proof}
Fix $a>0$ and set
$$f(x)\coloneqq u(x,a),\qquad v(x,t)\coloneqq P_tf(x),\qquad x\in\R^n,\,t>0.$$
We prove, for every compact $K\subset\R^n$,
\begin{equation}\label{eq:UniformConvergence}
\sup_{x\in K}\abs{P_t f(x)-f(x)}\overset{t\to 0}{\longrightarrow} 0.
\end{equation}
Choose $0\leq\chi\leq1$, $\chi\in C_0^\infty(\R^n)$, equal to $1$ on a neighbourhood of $K$. Decompose
$$f=f_1+f_2,\qquad f_1=\chi f,\quad f_2=(1-\chi)f.$$
The function $f_1$ is smooth and compactly supported. Hence
$$P_t f_1\longrightarrow f_1\quad\text{uniformly on }K;$$
see \cite[Theorem 1.2.19(2)]{Grafakos2008}. For the remaining part, choose $\delta>0$ such that
$$\abs{x-z}\geq\delta\qquad x\in K,\,z\in\support f_2.$$
Hölder’s inequality gives
\begin{align*}
\abs{P_t f_2(x)}
&\leq\left(\int_{\abs{\tau}\geq\delta}k_t(\tau)^{p^\prime}\dd \tau\right)^{1/p^\prime}\norm{f_2}_{L^p(\R^n)}\\
&\leq c_n t\left(\int_{\abs{\tau}\geq\delta}\abs{\tau}^{-(n+1)p^\prime}\dd \tau\right)^{1/p^\prime}\norm{f_2}_{L^p(\R^n)}\\
&\leq C_{\delta,n,p}\,t\norm{f}_{L^p(\R^n)}.
\end{align*}
The integral is finite because $(n+1)p^\prime>n$, and the estimate is uniform in $x\in K$. Since $f_2=0$ on $K$, combining the two parts proves the desired convergence.

Define
$$w(x,t)=u(x,a+t)-v(x,t),\qquad x\in\R^n,\,t>0.$$
Both terms are harmonic in $(x,t)$, so $\Delta_{x,t}w=0$ in $H.$ We are going to show that $w$ extends continuously to $t=0$ with zero trace. Smoothness of $u$ gives
$$u(x,a+t)\longrightarrow u(x,a)=f(x)$$
locally uniformly in $x$. By \eqref{eq:UniformConvergence}, the same is true for $v(x,t)$. Consequently,
$$\sup_{x\in K}\abs{w(x,t)}\overset{t\to 0}{\longrightarrow} 0$$
for every compact $K$. Thus setting $w(x,0)=0$ gives a continuous extension to $\overline H$. Its odd reflection
$$W(x,t)\coloneqq
\begin{cases}
    w(x,t),&t\geq0,\\
    -w(x,-t),&t<0,
\end{cases}
$$
is continuous on $\R^{n+1}$ and harmonic for $t<0$ and $t>0.$ By \cite[Problem 2.4]{GilbargTrudinger2001}, it is therefore harmonic on all of $\R^{n+1}.$

Young's inequality and \eqref{eq:Assumption1} give, for every $t\in\R$,
\begin{equation}\label{eq:WDecayEstimate}
\begin{aligned}
\norm{W(\cdot,t)}_{L^p(\R^n)}
&=\norm{w(\cdot,\abs{t})}_{L^p(\R^n)}\\
&\leq\norm{u(\cdot,a+\abs{t})}_{L^p(\R^n)}+\norm{v(\cdot,\abs{t})}_{L^p(\R^n)}\\
&\leq C_a\bigl(1+\abs{t}^\beta\bigr).
\end{aligned}
\end{equation}

We show that this forces $W$ to be a polynomial. Fix $z_0\in\R^{n+1}$, choose an integer $k>\beta$, and let $\alpha$ be a multi-index with $\abs{\alpha}=k$. Combining the interior derivative estimate \cite[Theorem 2.10]{GilbargTrudinger2001} with the mean-value estimate \cite[Theorem 2.1]{GilbargTrudinger2001} gives, for all sufficiently large $R$,
\begin{equation}\label{eq:WInteriorEstimate}
\abs{D^\alpha W(z_0)}\leq C R^{-k-(n+1)/p}\left(\int_{B(z_0,R)}\abs{W(\tau)}^p\dd \tau\right)^{1/p}.
\end{equation}
By \eqref{eq:WDecayEstimate},
$$\int_{B(z_0,R)}\abs{W(\tau)}^p\dd \tau\leq\int_{t_0-R}^{t_0+R}\norm{W(\cdot,t)}_{L^p(\R^n)}^p\dd t \leq C_{z_0}R^{\beta p+1},$$
where $t_0$ is the last coordinate of $z_0$.  Substitution in \eqref{eq:WInteriorEstimate} yields
$$\abs{D^\alpha W(z_0)}\leq C_{z_0}R^{\beta-k-n/p}\longrightarrow 0.$$
Hence all derivatives of order $k$ vanish and $W$ is a polynomial of degree at most $k-1$.  For each fixed $t$, however, $W(\cdot,t)\in L^p(\R^n)$, so this polynomial is identically zero.  Therefore $u(\cdot,a+t)=P_tu(\cdot,a)$.
\end{proof}

\begin{remark}
The stronger condition
$$\sup_{y>0}\norm{u(\cdot,y)}_{L^p(\R^n)}<\infty,$$
defines the harmonic Hardy space $u\in \mathcal{H}^p(\R^n).$ Under this assumption, in place of \eqref{eq:Assumption1}, the statement of Lemma \ref{lem:HorizontalPoissonShift} is classical; see e.g. \cite[Theorems 7.5 and 7.13]{AxlerBourdonRamey2001}.
\end{remark}

We are now in a position to prove the desired intersection statement.
\begin{proposition}\label{prop:LowWeightHarmonicTriviality}
Let $1<p<\infty$ and let $-p-1<\gamma\leq-1.$ Suppose that $u\in L^p_\gamma(H)$ satisfies
$$\Delta u=0\qquad\text{in }\mathcal{D}^\prime(H).$$
Then $u=0$. In particular,
$$L^p_{\sigma,\gamma}(H)\cap G^p_\gamma(H)=\{0\}.$$
\end{proposition}
\begin{proof}
Lemma \ref{lem:HorizontalEstimate} shows that for $y>0,$
\begin{equation*}
\norm{u(\cdot,y)}_{L^p(\R^{n})}\leq C y^\beta\norm{u}_{L^p_\gamma(H)},\qquad\beta=-\frac{\gamma+1}{p}.
\end{equation*}
Note that the assumptions on $\gamma$ imply $0\leq\beta<1.$ Lemma \ref{lem:HorizontalPoissonShift} thus gives
\begin{equation*}
u(\cdot,a+t)=P_tu(\cdot,a),\qquad a,t>0.
\end{equation*}
Young's inequality now yields
\begin{equation}\label{eq:1}
\norm{u(\cdot,b)}_{L^p(\R^{n})}\leq \norm{u(\cdot,a)}_{L^p(\R^{n})},\qquad 0<a<b.
\end{equation}
Fix $b>0$. From \eqref{eq:1} and the weighted integrability of $u$, we obtain
\begin{equation*}
\norm{u(\cdot,b)}_{L^p(\R^{n})}^p\int_0^b y^\gamma\dd y\leq\int_0^b \norm{u(\cdot,y)}_{L^p(\R^{n})}^p y^\gamma\dd y<\infty.
\end{equation*}
But because $\gamma\leq-1$,
$$\int_0^b y^\gamma\dd y=\infty;$$
consequently, $\norm{u(\cdot,b)}_{L^p(\R^{n})}=0$. Since $b>0$ was arbitrary, $u=0$.

Finally, let $h\in L^p_{\sigma,\gamma}(H)\cap G^p_\gamma(H)$.  Then $h=\nabla\pi$ and $\divergence h=0$ in distributions, so $\Delta\pi=0$.  Every component of $h$ is therefore harmonic and belongs to $L^p_\gamma(H)$; applying the first part componentwise yields $h=0$.
\end{proof}
\subsection{Upper Range}\label{sec:StokesHigh}
In this subsection our aim is to prove the remaining case $\gamma\in[p-1,2p-1)$ of Theorem \ref{th:StokesOperator}.

Fix $1<p<\infty$ and let $p-1\leq\gamma<2p-1.$ Recall the definition of the operator $T_\lambda$
\begin{equation*}
\begin{cases}
D(T_\lambda)=C_{0,\sigma}^\infty(H),\quad T_\lambda f=\mathcal U_\lambda f,\qquad &-p-1<\gamma<p-1,\\
D(T_\lambda)=q_{p,\gamma} D_\lambda,\quad T_\lambda q_{p,\gamma}f=q_{p,\gamma}\mathcal U_\lambda f,\qquad &\hphantom{-}p-1\leq\gamma<2p-1,
\end{cases}
\end{equation*}
where
$$D_\lambda=\{f\in C_{0,\sigma}^\infty(H):\mathcal U_\lambda f\in L^p_\gamma(H)^{n+1}\}.$$
By Theorem \ref{th:ResolventEstimates}\ref{eq:UpperCoreResolventEstimate} we have
$$\norm{T_\lambda(q_{p,\gamma}f)}_{X^p_{\sigma,\gamma}(H)}=\norm{q_{p,\gamma}\mathcal{U}_\lambda f}_{X^p_{\sigma,\gamma}(H)}\leq\frac{C}{\abs{\lambda}}\norm{q_{p,\gamma}f}_{X^p_{\sigma,\gamma}(H)},\qquad f\in D_\lambda.$$
Consequently, $T_\lambda$ is well-defined, closable and extends continuously to $\overline{q_{p,\gamma}D_\lambda}^{X^p_{\sigma,\gamma}(H)}.$

We are going to show that the domain of the closure $\overline{T_\lambda}$ is $\overline{q_{p,\gamma} D_\lambda}^{X^p_{\sigma,\gamma}(H)}=X^p_{\sigma,\gamma}(H).$ In particular, it is independent of $\lambda.$ 

\medskip
\noindent Recall the definition of $\mathscr{P}_0$ from Definition \ref{def:P0}.
\begin{proposition}\label{prop:WellDefinednessDensity}
Let $1<p<\infty$ and let $p-1\leq\gamma<2p-1.$ Then
$$\mathscr{P}_0\subseteq\bigcap_{\lambda\in\C\setminus(-\infty,0]}D_\lambda\qquad\text{and}\qquad\overline{q_{p,\gamma}\,\mathscr{P}_0}^{X^p_{\sigma,\gamma}(H)}\!\!\!=X^p_{\sigma,\gamma}(H).$$
\end{proposition}
\begin{proof}
Let $1<p<\infty,\,\lambda\in\C\setminus(-\infty,0]$ and let $p-1\leq\gamma<2p-1.$ Consider $f=f[b,h]$ with $b\in C_0^\infty(\R^n)^n,\,h\in C_0^\infty(0,\infty)$ and $m_b=0.$ For $\tau_1=\frac{1}{2}$ and $\tau_2=\frac{3}{2}$ apply Lemma \ref{lem:ScaledKernelConvergenceII} with $r=p.$ By assumption it holds $m_b=0$ and hence $G^{(1)}(\cdot,t)=k_{0,t}m_b=0.$ Noting that $\norm{b}_{L^1(\R^n)}<\infty$ and $\int_{\R^n}\abs{z}\abs{b(z)}\dd z<\infty$ since $b$ is compactly supported, we get $\varepsilon_0>0$ and $C>0$ such that
\begin{equation*}
\sup_{t\in[\frac{1}{2},\frac{3}{2}]}\norm{\varepsilon^{-(n+1)}w[b](\cdot/\varepsilon,t/\varepsilon)}_{L^p(\R^n)}\leq C\varepsilon,\qquad 0<\varepsilon<\varepsilon_0.
\end{equation*}
It will be sufficient to consider $t=1$ for the further course of our argument. Letting $y=1/\varepsilon$ and $y_0=1/\varepsilon_0$ yields
$$\norm{w[b](\cdot,y)}_{L^p(\R^n)}\leq Cy^{-n/p^\prime-2}\qquad y>y_0.$$
By Remark \ref{rem:vReplacesw}, the same arguments may be performed for $v$ in place of $w,$ giving the combined estimate
$$\norm{\mathcal{U}_\lambda f(\cdot,y)}_{L^p(\R^n)^{n+1}}\leq Cy^{-n/p^\prime-2}\qquad y>y_0.$$
Consequently,
\begin{align*}
\int_{y_0}^\infty \norm{\mathcal{U}_\lambda f(\cdot,y)}_{L^p(\R^n)^{n+1}}^p y^\gamma\dd y \leq C\int_{y_0}^\infty y^{\gamma-2p-n(p-1)}\dd y<\infty
\end{align*}
due to $\gamma<2p-1.$ On $0<y<y_0$ use the unweighted $L^p$-estimate for $\mathcal{U}_\lambda f,$
$$\int_0^{y_0}\norm{\mathcal{U}_\lambda f(\cdot,y)}_{L^p(\R^n)^{n+1}}^p y^\gamma\dd y \leq (y_0)^\gamma\norm{\mathcal{U}_\lambda f}_{L^p(H)^{n+1}}^p<\infty$$
because $\gamma\geq p-1>0.$ Thus we have shown $f\in D_\lambda.$ Linearity of $\mathcal{U}_\lambda$ proves the same for finite sums. Since $\lambda$ was arbitrary, the first claim follows.

In order to show the second claim, now let $f=f[b,h]$ with $b\in C_0^\infty(\R^n)^n,\,h\in C_0^\infty(0,\infty)$ be arbitrary. Choose $\psi\in C_0^\infty(\R^n)$ with $\int_{\R^n} \psi \dd x=1$ and set
$$\psi_R(x)=R^{-n}\psi(x/R),\qquad b_R=b-m_b\psi_R,\qquad R>1.$$
By definition, $b_R\in C_0^\infty(\R^n)^n$ and $m_{b_R}=0.$ Then $f[b_R,h]$ approximates $f[b,h]$ in $L^p_\gamma(H)^{n+1}.$ Indeed,
\begin{equation}\label{eq:P0Approximation}
\begin{aligned}
\norm{f[b,h]-f[b_R,h]}_{L^p_\gamma(H)^{n+1}}
&\leq \abs{m_b} \bigl(\norm{\psi_R}_{L^p(\R^n)}\norm{h^\prime}_{L^p_\gamma(0,\infty)}+\norm{\nabla\psi_R}_{L^p(\R^n)^n}\norm{h}_{L^p_\gamma(0,\infty)}\bigr)\\
&\leq C(R^{-n/p^\prime}+R^{-n/p^\prime-1})\overset{R\to\infty}{\longrightarrow} 0.
\end{aligned}
\end{equation}
This shows $\mathscr{P}\subseteq {\overline{\mathscr{P}_0}}^{L^p_\gamma(H)^{n+1}}.$

Next, set
$$q=p^\prime,\qquad \beta=-\frac{\gamma}{p-1}.$$
Then $-q-1<\beta\leq -1.$ Let $\Psi\in \bigl(X^p_{\sigma,\gamma}(H)\bigr)^\prime$ be a functional which vanishes on $q_{p,\gamma}\mathscr{P}_0.$ Then, interpreting $\Psi\circ q_{p,\gamma}\in(L^p_\gamma(H)^{n+1})^\prime,$ we may represent this map by $g=(g_v,g_w)^\mathrm{T}\in L^q_\beta(H)^{n+1}$ using the ambient $L^p_\gamma$-duality, that is
$$\Psi(q_{p,\gamma}f)=\int_H g\cdot \overline{f}\dd x\dd y,\qquad f\in L^p_\gamma(H)^{n+1}.$$
Since $q_{p,\gamma}\nabla\pi=0$ for any $\pi\in C_0^\infty(H)$ by definition, we have
\begin{equation}\label{eq:DivergenceZero}
\divergence_{x,y}g=0\qquad\text{in }\mathcal{D}^\prime(H).
\end{equation}
By definition, $\Psi$ annihilates every $q_{p,\gamma}f,\,f\in\mathscr{P}_0,$ and by \eqref{eq:P0Approximation}, this is even true for every $f\in\mathscr{P}.$ That is,
$$0=\int_H f[b,h]\cdot\overline{g}\dd x\dd y=\int_Hbh^\prime\cdot\overline{g_v}\dd x\dd y-\int_H(\divergence b)h\overline{g_w}\dd x\dd y$$
for all $b\in C_0^\infty(\R^n)^n,\,h\in C_0^\infty(0,\infty).$ Integration by parts gives
\begin{equation}\label{eq:DerivativeIdentity}
\partial_y g_v=\nabla_x g_w\qquad \text{in }\mathcal{D}^\prime(H).
\end{equation}
A distribution vanishing on all such product test functions is zero, see \cite[Theorem 5.1.1]{Hormander1983}. Differentiating \eqref{eq:DivergenceZero} in $y$ and inserting it in the preceding equation we obtain
$$0=\partial_y\divergence_{x,y}g=\partial_y(\divergence_xg_v+\partial_y g_w)=\divergence_x\partial_y g_v+\partial_y^2 g_w=\divergence_x\nabla_xg_w+\partial_y^2 g_w=\Delta_{x,y}g_w\qquad\text{in }\mathcal{D}^\prime(H).$$
Proposition \ref{prop:LowWeightHarmonicTriviality} implies $g_w=0.$ Equation \eqref{eq:DerivativeIdentity} then shows that $g_v$ is independent of $y.$ Since $\int_0^1 y^\beta \dd y$ diverges, $g\in L^q_\beta(H)^{n+1}$ forces $g_v=0.$ Therefore, $\Psi=0.$ By the theorem of Hahn-Banach we may conclude
$$\overline{q_{p,\gamma}\,\mathscr{P}_0}^{X^p_{\sigma,\gamma}(H)}\!\!\!=X^p_{\sigma,\gamma}(H)$$
as desired.
\end{proof}

To summarize, Proposition \ref{prop:WellDefinednessDensity} shows that $q_{p,\gamma}\mathscr{P}_0$ is dense in $X^p_{\sigma,\gamma}(H)$ and it holds
$$q_{p,\gamma}\mathscr{P}_0\subseteq q_{p,\gamma}D_\lambda=D(T_\lambda),$$
so $T_\lambda$ is densely defined in $X^p_{\sigma,\gamma}(H).$ Thus, for every $\lambda\in\mathbb C\setminus(-\infty,0]$, the operator $T_\lambda$ has a unique bounded extension to $X^p_{\sigma,\gamma}(H).$ We next show that these extensions form the resolvent of a single closed, densely defined operator. For this purpose, we use the lower-range realization and the duality identification
$$\Phi_{p,\gamma}\,\colon\, L^p_\gamma(H)^{n+1}/G^p_\gamma(H)\to \bigl(L^{p^\prime}_{\sigma,\gamma^\prime}(H)\bigr)^\prime$$
from \eqref{eq:SolenoidalQuotientSpace}. Set
$$q=p^\prime,\qquad \beta=-\frac{\gamma}{p-1}.$$
Then $-q-1<\beta\leq -1.$ The Banach adjoint of $\mathrm{R}_{q,\beta}(\overline\lambda)$ acts on $\bigl(L^q_{\sigma,\beta}(H)\bigr)^\prime$, and we define
$$\mathrm{R}_{p,\gamma}(\lambda)\coloneqq\Phi_{p,\gamma}^{-1}\circ\bigl(\mathrm{R}_{q,\beta}(\overline\lambda)\bigr)^\prime\circ\Phi_{p,\gamma}
\quad\text{on}\quad X^p_{\sigma,\gamma}(H).$$
Define moreover
\begin{equation*}
\mathcal{A}_{p,\gamma}\coloneqq\Phi_{p,\gamma}^{-1}\circ\mathcal{A}_{q,\beta}^\prime\circ\Phi_{p,\gamma},\qquad\mathcal{D}(\mathcal{A}_{p,\gamma})\coloneqq\Phi_{p,\gamma}^{-1}\bigl(\mathcal{D}(\mathcal{A}_{q,\beta}^\prime)\bigr).
\end{equation*}
We now prove that $\mathcal{A}_{p,\gamma}$ satisfies all properties claimed in Theorem \ref{th:StokesOperator}.
\begin{proof}[Proof of Theorem \ref{th:StokesOperator}.]
Let $1<p<\infty;$ it remains to treat the case $p-1\leq\gamma<2p-1.$ Fix a resolvent parameter $\lambda\in\C\setminus(-\infty,0]$ and set
$$q=p^\prime,\qquad \beta=-\frac{\gamma}{p-1}.$$
Then $-q-1<\beta\leq -1.$ Since $L^q_{\sigma,\beta}(H)$ is reflexive and $\mathcal{A}_{q,\beta}$ is densely defined and closed, its adjoint is densely defined and closed on $\bigl(L^q_{\sigma,\beta}(H)\bigr)^\prime.$ The maps $\Phi_{p,\gamma}$ and its inverse are isometric isomorphisms, so $\mathcal{A}_{p,\gamma}$ is also closed and densely defined. The adjoint-resolvent identity \cite[VIII.6]{Yosida1965} moreover yields the relation 
$$\mathrm{R}_{p,\gamma}(\lambda)=(\lambda-\mathcal{A}_{p,\gamma})^{-1}.$$
Since $\Phi_{p,\gamma}$ is isometric and the adjoint preserves operator norms, it follows from the same estimates in the previous subsection that for every $\theta\in(0,\pi)$ there exists a constant $C_{n,p,\gamma,\theta}>0$ such that
\begin{equation*}
\norm{(\lambda-\mathcal{A}_{p,\gamma})^{-1}}_{\mathcal{L}(X^p_{\sigma,\gamma}(H))}=\norm{\mathrm{R}_{p,\gamma}(\lambda)}_{\mathcal{L}(X^p_{\sigma,\gamma}(H))} \leq \frac{C_{n,p,\gamma,\theta}}{\abs{\lambda}}\qquad\lambda\in\Sigma_\theta.
\end{equation*}
In particular, $\C\setminus(-\infty,0]\subseteq\rho(\mathcal{A}_{p,\gamma}).$

It remains to prove the relation between $\mathcal{A}_{p,\gamma}$ and $T_\lambda.$ Let $f,g\in C_{0,\sigma}^\infty(H)$ such that $f\in D_\lambda.$ Fix a representative $v\in L^p_\gamma(H)^{n+1}$ of $\mathrm{R}_{p,\gamma}(\lambda)q_{p,\gamma}f\in X^p_{\sigma,\gamma}$. Then $\mathcal U_\lambda f-v\in L^p_\gamma(H)^{n+1}.$ By definition
\begin{align*}
\int_H v\cdot\overline g\dd x\dd y
&=\int_H f\cdot\overline{\mathrm{R}_{q,\beta}(\overline\lambda)g}\dd x\dd y\\
&=\int_H f\cdot\overline{\mathcal U_{\overline\lambda}g}\dd x\dd y
=\int_H\mathcal U_\lambda f\cdot\overline g\dd x\dd y.
\end{align*}
These pairings are well-defined: the first two by weighted H\"older's inequality and the other two by unweighted $L^2$ theory. Since $g\in C_{0,\sigma}^\infty(H)$ was arbitrary, this shows $\mathcal U_\lambda f-v\in (C_{0,\sigma}^\infty(H))^\perp.$ By density of $C_{0,\sigma}^\infty(H)$ in $L^q_{\sigma,\beta}(H)$ this means that the above pairing vanishes on all of $L^q_{\sigma,\beta}(H)$. Proposition \ref{prop:AnnihilatorCharacterization} therefore proves that there exists $\nabla\pi\in G^p_\gamma(H)$ such that $\mathcal U_\lambda f-v=\nabla\pi$ and therefore
\begin{equation}\label{eq:ConsistencyIdentity}
\mathrm{R}_{p,\gamma}(\lambda)q_{p,\gamma}f=q_{p,\gamma}\mathcal U_\lambda f.
\end{equation}
Since $f\in D_\lambda$ was arbitrary, this implies
$$T_\lambda(q_{p,\gamma} f)=q_{p,\gamma}\mathcal{U}_\lambda f=\mathrm{R}_{p,\gamma}(\lambda)q_{p,\gamma}f,\qquad f\in D_\lambda.$$
Thus $T_\lambda$ is the restriction of a bounded operator to a dense domain. Its graph closure is the graph of that bounded operator:
$$\mathcal{D}(\overline{T_\lambda})=X^p_{\sigma,\gamma}(H),\qquad \overline{T_\lambda}=\mathrm{R}_{p,\gamma}(\lambda)=(\lambda-\mathcal{A}_{p,\gamma})^{-1}.$$
This finishes the proof.
\end{proof}

We conclude this section with a statement which guarantees $\mathcal{U}_\lambda\bigl(C_{0,\sigma}^\infty(H)\bigr) \subseteq L^p_\gamma(H)^{n+1}$ as long as
$$p-1\leq\gamma<(n+1)(p-1).$$
For $-p-1<\gamma<p-1,$ we have already seen this since it is a consequence of Theorem~\ref{th:ResolventEstimates} assertion \ref{eq:LowerCoreResolventEstimate}.
\begin{proposition}\label{prop:MappingProperties}
Fix $1<p<\infty$ and let $\lambda\in\mathbb{C}\setminus(-\infty,0].$ If $p-1\leq\gamma<(n+1)(p-1),$ then for $f\in C_{0,\sigma}^\infty(H),$
$$\mathcal{U}_\lambda f\in L^p_\gamma(H)^{n+1}.$$
\end{proposition}
\begin{proof}
Fix $1<p<\infty$ and $\lambda\in\mathbb{C}\setminus(-\infty,0].$ Choose $\theta\in(0,\pi)$ so that $\lambda\in\Sigma_\theta.$

Now assume $p-1\leq\gamma<(n+1)(p-1).$ Consider the auxiliary weight $v_\gamma(x,y)=\abs{(x,y)}^\gamma$ defined on the full space $\R^{n+1}.$ It is easily seen that $v_\gamma\in A_p(\R^{n+1})$ if and only if $-(n+1)<\gamma<(n+1)(p-1).$ Thus, in particular, $v_\gamma\in A_p(\R^{n+1})$ for the assumed choice of $p$ and $\gamma.$

Let $f\in C_{0,\sigma}^\infty(H).$ Since $f$ is compactly supported in $H,$ it belongs to $L^2(H)^{n+1}\cap L^p_{v_\gamma}(H)^{n+1}.$ Fröhlich proved in \cite[Theorem 1.1]{Froehlich2003} 
that the $L^2$-solution $(u_\lambda,\pi_\lambda)\coloneqq (\mathcal{U}_\lambda,\Pi_\lambda) f$ is thus also the unique solution to \eqref{eq:ResolventProblem} with $(u_\lambda,\pi_\lambda)\in W^{2,p}_{v_\gamma}(H)^{n+1}\times \widehat{W}^{1,p}_{v_\gamma}(H),$ and for this solution one has the estimate
$$\abs{\lambda}\norm{u_\lambda}_{L^p_{v_\gamma}(H)^{n+1}}+\norm{D^2 u_\lambda}_{L^p_{v_\gamma}(H)^{n+1}}+\norm{\nabla\pi_\lambda}_{L^p_{v_\gamma}(H)^{n+1}}\leq C_{n,p,\gamma,\theta} \norm{f}_{L^p_{v_\gamma}(H)^{n+1}}<\infty.$$

Since $0<y\leq\abs{(x,y)}$ and $\gamma>0,$
$$\norm{\mathcal{U}_\lambda f}_{L^p_\gamma(H)^{n+1}}^p=\int_H \abs{\mathcal{U}_\lambda f}^py^\gamma\dd x\dd y\leq \norm{\mathcal{U}_\lambda f}_{L^p_{v_\gamma}(H)^{n+1}}^p<\infty$$
which is what we wanted to show.
\end{proof}

\begin{remark}
The above proof only needed the lower bound $\gamma\geq 0,$ regardless of the value of $p.$
\end{remark}

\section{\texorpdfstring{$H^\infty$-calculus}{H-infinity-calculus}}
\label{sec:HinftyCalculus}
The aim of this section is to prove Theorem \ref{th:HinftyCalculus}. Our proof follows closely a line of argument in \cite{DHP2001}. Let $1<p<\infty.$ For the case $-p-1<\gamma<p-1$ recall the definitions from Section \ref{sec:PointwiseEstimates} for $\varphi\in(0,\theta)$ and $h\in H^\infty_0(\Sigma_\theta),\,f\in C_{0,\sigma}^\infty(H),$
\begin{equation*}
\begin{aligned}
    k_{h,v}(x,y,\tilde{y}) &= \frac{1}{2\pi i} \int_{\partial\Sigma_\varphi} h(\lambda) r_v(x,y,\tilde{y},-\lambda) \dd \lambda, \quad x \in \R^n,\, y, \tilde{y} > 0,\\
    (T_{h,v}f)(x,y) &= \int_0^\infty \int_{\R^n} k_{h,v}(x-\tilde{x}, y, s) f_v(\tilde{x}, s) \dd \tilde{x} \dd s, \quad x \in \R^n,\, y>0,
\end{aligned}
\end{equation*}
and likewise
\begin{equation*}
\begin{aligned}
    k_{h,w}(x,y,\tilde{y}) &= \frac{1}{2\pi i} \int_{\partial\Sigma_\varphi} h(\lambda) r_w(x,y,\tilde{y},-\lambda) \dd \lambda, \quad x \in \R^n,\, y, \tilde{y} > 0,\\
    (T_{h,w}f)(x,y) &= \int_0^\infty \int_{\R^n} k_{h,w}(x-\tilde{x}, y, s) \cdot f_v(\tilde{x}, s) \dd \tilde{x} \dd s, \quad x \in \R^n,\, y>0.
\end{aligned}
\end{equation*}
By Fubini's theorem and \eqref{eq:v2}, \eqref{eq:w2} together with \eqref{eq:kThv}, \eqref{eq:kThw}, one easily sees that
$$h(-\mathcal{A}_{p,\gamma})f=h(-\Delta_{p,\gamma}^D)f-(T_{h,v}f,\,T_{h,w}f)^\mathrm{T},\qquad f\in C_{0,\sigma}^\infty(H).$$
Define
$$T_hf\coloneqq(T_{h,v}f,T_{h,w}f)^\mathrm{T},\qquad f\in C_{0,\sigma}^\infty(H).$$
Moreover, for $F\,\colon\,[0,\infty)\to[0,\infty)$ define
\begin{equation}\label{eq:AuxiliaryOperator}
(TF)(y) \coloneqq \int_0^\infty \frac{\log(1+\frac{y}{s})}{y+s} F(s) \dd s, \quad y > 0,
\end{equation}
and for $f\in C_{0,\sigma}^\infty(H)$ let $F_p(y)=\norm{f(\cdot,y)}_{L^p(\R^n)^{n+1}},\,y>0.$
\begin{lemma}\label{lem:hTEstimate}
Let $1<p<\infty$ and let $-p-1<\gamma<2p-1.$ Then there exists a constant $C>0$ such that
$$\norm{T_hf}_{L^p_{\gamma}(H)^{n+1}} \leq C \norm{h}_{\infty} \norm{TF_p}_{L^p((0,\infty),\,y^\gamma\mathrm{d}y)},\qquad f\in C_{0,\sigma}^\infty(H).$$
\end{lemma}
\begin{proof}
Estimate
\begin{align*}
\norm{T_{h,v}f}^p_{L^p_\gamma(H)^n} &= \int_0^\infty \int_{\R^n} \abs{\int_0^\infty \int_{\R^n} k_{h,v}(x-\tilde{x}, y, s)f_v(\tilde{x}, s) \dd \tilde{x} \dd s}^p \dd x\:y^\gamma\mathrm{d}y\\
&\leq \int_0^\infty \int_{\R^n} \left(\int_0^\infty \abs{\int_{\R^n} k_{h,v}(x-\tilde{x}, y, s)f_v(\tilde{x}, s) \dd \tilde{x}} \dd s\right)^p \dd x\:y^\gamma\mathrm{d}y\\
&= \int_0^\infty \int_{\R^n} \left(\int_0^\infty \abs{(k_{h,v}\ast f_v)(x, y, s)} \dd s\right)^p \dd x\:y^\gamma\mathrm{d}y\\
&= \int_0^\infty \int_{\R^n} \norm{(k_{h,v}\ast f_v)(x, y, \cdot)}_1^p \dd x\:y^\gamma\mathrm{d}y\\
&\leq \int_0^\infty \left(\int_0^\infty \norm{(k_{h,v}\ast f_v)(\cdot, y, s)}_p \dd s\right)^p\:y^\gamma\mathrm{d}y.
\end{align*}
Here we used Minkowski's integral inequality in the last step. An application of Young's convolution inequality and the first pointwise estimate \eqref{eq:pointwiseEstimate1} from Section \ref{sec:PointwiseEstimates} gives
\begin{align*}
\int_0^\infty \left(\int_0^\infty \norm{(k_{h,v}\ast f_v)(\cdot, y, s)}_p \dd s\right)^p\:y^\gamma\mathrm{d}y
&\leq \int_0^\infty \left(\int_0^\infty \norm{k_{h,v}(\cdot,y,s)}_1 \norm{f_v(\cdot,s)}_p \dd s\right)^p\:y^\gamma\mathrm{d}y\\
&\leq C \norm{h}^p_\infty \int_0^\infty \left(\int_0^\infty \norm{k_1(\cdot,y,s)}_1 \norm{f_v(\cdot,s)}_p \dd s\right)^p\:y^\gamma\mathrm{d}y.
\end{align*}
The second pointwise estimate \eqref{eq:pointwiseEstimate2} and $\norm{f_v(\cdot,y)}_{L^p(\R^n)^n}\leq F_p(y)$ let us conclude
$$\norm{T_{h,v}f}_{L^p_{\gamma}(H)^n} \leq C \norm{h}_\infty \norm{TF_p}_{L^p((0,\infty),\,y^\gamma\mathrm{d}y)},\qquad f\in C^\infty_{0,\sigma}(H).$$

Repeating the same estimates for $T_{h,w}$ and $k_{h,w}$ in place of $T_{h,v}$ and $k_{h,v},$ respectively, yields the claim.
\end{proof}
It remains to show boundedness of the operator $T$ on $L^p((0,\infty),\,y^\gamma\mathrm{d}y).$ For the direct estimate assume $-p-1<\gamma<p-1.$ The upper range $p-1\leq\gamma<2p-1$ will be treated by duality at the end of this section.
\begin{lemma}\label{lem:L_inftyEstimate}
Let $\gamma\in\R.$ There exists a constant $C>0$ such that
\begin{equation*}
\norm{TF}_{L^\infty((0,\infty),\,y^\gamma\mathrm{d}y)} \leq C\norm{F}_{L^{\infty}((0,\infty),\,y^\gamma\mathrm{d}y)}.
\end{equation*}
\end{lemma}
\begin{proof}
It holds
\begin{align*}
\int_0^\infty \frac{\log(1+\frac{y}{s})}{y+s} \dd s &= -\int_0^\infty \frac{\log(1+\tau)}{s(\tau+1)}\left(-\frac{s}{\tau}\right) \dd \tau\\
&= \int_0^\infty \frac{\log(1+\tau)}{\tau(\tau+1)} \dd \tau\\
&= \int_0^1 \frac{\log(1+\tau)}{\tau(\tau+1)} \dd \tau + \int_1^\infty \frac{\log(1+\tau)}{\tau(\tau+1)} \dd \tau,
\end{align*}
where we substituted $\tau = \frac{y}{s}, \;\frac{\mathrm{d} s}{\mathrm{d} \tau} = -\frac{y}{\tau^2} = -\frac{s}{\tau}.$ The second integral is clearly finite. For the first, note that $\log(1+x) = x + O(x^2)$ for $x \in (-1,1)$ so that $\frac{O(\tau^2)}{\tau} \frac{1}{\tau+1}$ is bounded on $(0,1).$ We conclude $\frac{\log(1+\frac{y}{s})}{y+s} \in L^1((0,\infty),\dalt s).$ Note that $L^\infty((0,\infty),\dalt y) = L^\infty((0,\infty),\,y^\gamma\dalt y)$ since both measures share the same null sets. This concludes the claim.
\end{proof}

\begin{lemma}\label{lem:weakLpEstimate}
Let $1<p<\infty$ and let $-p-1<\gamma<p-1.$ There exists a constant $C>0$ such that
\begin{equation*}
\norm{TF}_{L^{p,\infty}((0,\infty),\,y^\gamma\dalt y)} \leq C \norm{F}_{L^p((0,\infty),\,y^\gamma\dalt y)}.
\end{equation*}
\end{lemma}
\begin{proof}
Fix $1<p<\infty$ and $-p-1<\gamma<p-1.$ For $y>0,$
\begin{align*}
\abs{(TF)(y)} &= \int_0^\infty \frac{\log(1+\frac{y}{s})}{y+s} F(s) \dd s\\
&\leq \left(\int_0^\infty \left(\frac{\log(1+\frac{y}{s})}{y+s}\right)^{p^\prime} (s)^{\gamma^\prime} \dd s\right)^{1/p^\prime} \norm{F}_{L^p((s)^\gamma \dalt s)}\\
&= \left(\int_0^\infty \left(\frac{\log(1+\tau)\,\tau}{y\,(1+\tau)}\right)^{p^\prime} \left(\frac{y}{\tau}\right)^{\gamma^\prime} \left(\frac{y}{\tau^2}\right)\dd \tau\right)^{1/p^\prime} \norm{F}_{L^p((s)^\gamma \dalt s)}\\
&= y^{\gamma^\prime/p^\prime}\,y^{-1}\,y^{1/p^\prime}\left(\int_0^\infty \left(\frac{\log(1+\tau)}{1+\tau}\right)^{p^\prime} \frac{1}{\tau^{2-p^\prime+\gamma^\prime}}\dd \tau\right)^{1/p^\prime} \norm{F}_{L^p((s)^\gamma \dalt s)}\\
&= y^{-\frac{1+\gamma}{p}}\left(\int_0^\infty \left(\frac{\log(1+\tau)}{1+\tau}\right)^{p^\prime} \frac{1}{\tau^{2-p^\prime+\gamma^\prime}}\dd \tau\right)^{1/p^\prime}\norm{F}_{L^p((s)^\gamma \dalt s)},
\end{align*}
where we first applied the weighted Hölder inequality and then substituted $\tau = \frac{y}{s}, \;\frac{\mathrm{d}s}{\mathrm{d} \tau} = -\frac{y}{\tau^2} = -\frac{s}{\tau}.$

Split the integral
$$\left(\int_0^\infty \left(\frac{\log(1+\tau)}{1+\tau}\right)^{p^\prime} \frac{1}{\tau^{2-p^\prime+\gamma^\prime}}\dd \tau\right)^{1/p^\prime}$$
at $\tau=1$ and consider both terms separately. Using $\log(1+\tau)\leq \tau$ it holds
\begin{equation*}
\int_0^1 \left(\frac{\log(1+\tau)}{1+\tau}\right)^{p^\prime} \frac{1}{\tau^{2-p^\prime+\gamma^\prime}}\dd \tau \leq \int_0^1 \frac{1}{\tau^{2-2p^\prime+\gamma^\prime}}\dd \tau.
\end{equation*}
The right-hand side is finite since $\gamma^\prime < 2p^\prime -1.$ For the remaining integral we have for $\varepsilon>0$
\begin{align*}
\int_1^\infty \left(\frac{\log(1+\tau)}{1+\tau}\right)^{p^\prime} \frac{1}{\tau^{2-p^\prime+\gamma^\prime}}\dd \tau &= \int_1^\infty \left(\frac{\log(1+\tau)}{(1+\tau)^\varepsilon}\right)^{p^\prime} \frac{1}{\tau^{2-p^\prime+\gamma^\prime}(1+\tau)^{p^\prime(1-\varepsilon)}}\dd \tau.
\end{align*}
Letting $0<\varepsilon<(\gamma^\prime+1)/p^\prime,$ the right-hand side is finite since in this case $2-p^\prime+\gamma^\prime+p^\prime(1-\varepsilon) > 1.$ To summarize, there exists a constant
$$C_{p,\gamma}\coloneqq \left(\int_0^\infty \left(\frac{\log(1+\tau)}{1+\tau}\right)^{p^\prime} \frac{1}{\tau^{2-p^\prime+\gamma^\prime}}\dd \tau\right)^{1/p^\prime}$$
such that
$$\abs{(TF)(y)}\leq C_{p,\gamma}\, y^{-\frac{1+\gamma}{p}} \norm{F}_{L^p((0,\infty),\,s^\gamma \dalt s)},\quad y>0.$$
We will now employ this estimate to show the weak-$L^p$ estimate.

If $\gamma\neq-1,$ let $F\neq 0$ and $r>0$ and set $b=\left(\frac{C_{p,\gamma}\norm{F}_{L^p((0,\infty),\,y^\gamma\dalt y)}}{r}\right)^{\frac{p}{1+\gamma}}.$ We distinguish three cases. First, let $\gamma>-1.$ Then the superlevel set $\{\abs{TF} > r\}$ is contained in $(0,b)$ and thus
\begin{align*}
\norm{TF}_{L^{p,\infty}((0,\infty),\,y^\gamma\dalt y)} &= \sup_{r>0} r\,\Bigl(\int_{\abs{TF} > r} y^\gamma \dd y \Bigl)^{\frac{1}{p}}\\
&\leq \sup_{r>0} r\,\Bigl(\int_0^b y^\gamma \dd y \Bigl)^{\frac{1}{p}}\\
&= \sup_{r>0} r \frac{1}{\abs{1+\gamma}^{1/p}} \frac{C_{p,\gamma}\norm{F}_{L^p((0,\infty),\,y^\gamma\dalt y)}}{r}\\
&= \frac{C_{p,\gamma}}{\abs{1+\gamma}^{1/p}} \norm{F}_{L^p((0,\infty),\,y^\gamma\dalt y)}.
\end{align*}
If $\gamma<-1,$ the superlevel set $\{\abs{TF} > r\}$ is contained in $(b,\infty)$ and a calculation analogous to above, replacing $\int_0^b y^\gamma \dd y$ by $\int_b^\infty y^\gamma \dd y,$ yields the same estimate.

Lastly, we consider $\gamma=-1.$ Fix $F\,\colon\,[0,\infty)\to[0,\infty).$ Substituting $t=s/y$ in \eqref{eq:AuxiliaryOperator} gives
\begin{equation*}
TF(y)=\int_0^\infty \frac{\log (1+t^{-1})}{1+t}F(yt)\dd t
\end{equation*}
and hence an application of Minkowski's integral inequality yields
\begin{align*}
\norm{TF}_{L^p((0,\infty),\,y^{-1}\dalt y)}
&\leq \int_0^\infty \frac{\log (1+t^{-1})}{1+t}t^{-\frac{1+\gamma}{p}}\dd t \:\:\norm{F}_{L^p((0,\infty),\,y^{-1}\dalt y)}\\
&=\int_0^\infty \frac{\log (1+t^{-1})}{1+t}\dd t\:\: \norm{F}_{L^p((0,\infty),\,y^{-1}\dalt y)}.
\end{align*}
One sees that the expression $\int_0^\infty \frac{\log (1+t^{-1})}{1+t}\dd t$ is finite in the same way as in the $L^\infty$-estimate of $TF.$ This last case even gives us a strong $L^p$-bound, which implies the weak one. Altogether, this proves the claim.
\end{proof}

\begin{proof}[Proof of Theorem \ref{th:HinftyCalculus}.]
We begin with the lower range case. Fix $1<p<\infty$ and let $-p-1<\gamma<p-1.$ Choose $\max(1,\gamma+1,-\gamma-1)<r_0<p.$ Such a choice exists because both weight inequalities are strict. It ensures $-r_0-1<\gamma<r_0-1,$ so $T$ is weak type-$(r_0,r_0)$ by Lemma \ref{lem:weakLpEstimate}. Together with Lemma \ref{lem:L_inftyEstimate} Marcinkiewicz interpolation shows that the operator $T$ is bounded on $L^p((0,\infty),\,y^\gamma\mathrm{d}y),$ see \cite[Theorem 2.2.3]{HNVW2016}. Lemma \ref{lem:hTEstimate} lets us conclude
$$\norm{T_h f}^p_{L^p_\gamma(H)^{n+1}} \leq C \norm{h}^p_{H^\infty(\Sigma_\theta)}\norm{f}_{L^p_\gamma(H)^{n+1}}^p,\qquad f\in C_{0,\sigma}^\infty(H).$$
The Dirichlet part has a bounded $H^\infty$-calculus of angle zero by \cite[Proposition 5.3]{LindemulderVeraar2020}. Combining the estimates for both parts and using the density of $C_{0,\sigma}^\infty(H)$ in $L^p_{\sigma,\gamma}(H)$ gives, for $\theta\in(0,\pi),$
$$\norm{h(-\mathcal{A}_{p,\gamma})}_{\mathcal{L}(L^p_{\sigma,\gamma}(H))}\leq C_{p,\gamma,\theta}\norm{h}_{H^\infty(\Sigma_\theta)},\qquad h\in H^\infty_0(\Sigma_\theta),$$
which proves the claim on the lower range.

Let us now turn to the case $p-1\leq\gamma<2p-1.$ Set
$$q=p^\prime,\qquad \beta=-\frac{\gamma}{p-1},$$
and recall
$$-\mathcal{A}_{p,\gamma}=\Phi_{p,\gamma}^{-1}(-\mathcal{A}_{q,\beta})^\prime\Phi_{p,\gamma}\qquad \text{on }X^p_{\sigma,\gamma}(H).$$
For $h\in H_0^\infty(\Sigma_\theta)$ set $\widetilde{h}(z)=\overline{h(\overline{z})}.$ By \cite[Proposition 10.2.20]{HNVW2017}, this gives
$$h(-\mathcal{A}_{p,\gamma})=\Phi_{p,\gamma}^{-1}\bigl(\tilde{h}(-\mathcal{A}_{q,\beta})\bigr)^\prime\Phi_{p,\gamma},\qquad\bigl\|\tilde{h}\bigr\|_{H^\infty(\Sigma_\theta)}=\norm{h}_{H^\infty(\Sigma_\theta)}.$$
Since $\Phi_{p,\gamma}$ is isometric and adjoints preserve operator norms,
$$\norm{h(-\mathcal{A}_{p,\gamma})}_{\mathcal{L}(X^p_{\sigma,\gamma}(H))}=\bigl\|\tilde{h}(-\mathcal{A}_{q,\beta})\bigr\|_{\mathcal{L}(L^q_{\sigma,\beta}(H))}\leq C_{q,\beta,\theta}\norm{h}_{H^\infty(\Sigma_\theta)}.$$
As $\theta\in(0,\pi)$ was arbitrary, the angle of the calculus is zero. This concludes the proof.
\end{proof}

\section*{Acknowledgement} The author would like to thank Amru Hussein for his time, advice and support throughout the writing of this article. The author would also like to thank Floris Roodenburg for fruitful discussions regarding Proposition \ref{prop:CounterexampleHelmholtzDecomposition} and Section \ref{sec:StokesHigh} and Paul Beckermann for his helpful remarks and the discussion regarding Section \ref{sec:StokesLow}.

The author has been  supported by the Deutsche Forschungsgemeinschaft (DFG)  through the projects  508634462 and 536630683.

This article is part of the author's PhD thesis \emph{The Stokes Operator in Weighted Spaces}.


\bibliographystyle{plain}
\bibliography{Ref}

\end{document}